\documentclass[11pt,letterpaper,final,twoside,leqno]{amsart}
\usepackage{amsmath,amsthm,amssymb,amsfonts,amscd,amsopn} %
\usepackage{eucal,mathrsfs,relsize}
\usepackage{xspace,chngcntr}
\usepackage{rotating,mathtools}
\usepackage[all,cmtip]{xy}\xyoption{dvips}
\usepackage{tikz-cd}
\usetikzlibrary{arrows.meta}
 
\usepackage{comment} 
\usepackage{supertabular}
\usepackage
{hyperref}
\hypersetup{
    colorlinks=true,
    linkcolor={blue!50!black},
    citecolor={green!50!black},
    urlcolor={red!80!black}
}

\usepackage{datetime,chngcntr}
\usepackage{paralist}

\usepackage{enumitem}
\usepackage{eurosym}

\usepackage{mfirstuc}
\DeclareMathAlphabet{\smallchanc}{OT1}{pzc}%
                                 {m}{it}
\DeclareFontFamily{OT1}{pzc}{}
\DeclareFontShape{OT1}{pzc}{m}{it}%
             {<-> s * [1.100] pzcmi7t}{}
\DeclareMathAlphabet{\mathchanc}{OT1}{pzc}%
                                 {m}{it}

   [2017/06/13 v0.1 finetuning mabliautoref and theorem setup]

\RequirePackage{hyperref}

\newcommand{\sectionsize}{\relsize{-.5}}
\newcommand{\theoremsize}{\relsize{-.5}}
\newcommand{\mrsize}{\relsize{-.8}}

\renewcommand{\subsectionautorefname}{\sectionsize\sf \subsectionautorefname}

\makeatletter
\@ifdefinable\equationname{\let\equationname\equationautorefname}
\def\equationautorefname~#1\@empty\@empty\null{\protect{\theoremsize\sf
    (#1\@empty\@empty\null)}}%
\@ifdefinable\AMSname{\let\AMSname\AMSautorefname}
\def\AMSautorefname~#1\@empty\@empty\null{\sf( #1\@empty\@empty\null)}%
\@ifdefinable\itemname{\let\itemname\itemautorefname}
\def\itemautorefname~#1\@empty\@empty\null{\mrsize{%
    {\sf #1}}\@empty\@empty\null%
}%
\makeatother

\RequirePackage{amsthm}
\RequirePackage{aliascnt}
\newcommand{\basetheorem}[3]{%
    \newtheorem{#1}{#2}[#3]
    \newtheorem*{#1*}{#2}
    \expandafter\def\csname #1autorefname\endcsname{#2}
}%
\newcommand{\maketheorem}[3]{%
    \newaliascnt{#1}{#2}
    \newtheorem{#1}[#1]{\theoremsize\sf #3}
    \aliascntresetthe{#1}
    \expandafter\def\csname #1autorefname\endcsname{\theoremsize\sf #3}
    \newtheorem{#1*}{#3}
}%
\newcommand{\baseremark}[3]{%
    \newtheorem{#1}{#2}{#3}
    \newtheorem*{#1*}{#2}
    \expandafter\def\csname #1autorefname\endcsname{#2}
}%
\newcommand{\makeremark}[3]{%
    \newaliascnt{#1}{#2}
    \newtheorem{#1}[equation]{#3}
    \aliascntresetthe{#1}
    \expandafter\def\csname #1autorefname\endcsname{\theoremsize\sf #3}
    \newtheorem{#1*}{#3}
}%
\theoremstyle{plain}   

\basetheorem{theorem}{Theorem}{section}

\newcommand{\longpagesize}{
\textwidth= 6.75in
\textheight=9.25in
\voffset-.85in
\hoffset-.99in
\marginparwidth=56pt
\footskip.5in
}
\newcounter{are-there-sections}
\newcommand{\zzzzz}{}

\longpagesize
\def\me{S\'ANDOR J KOV\'ACS\xspace}

\def\mythanks{Supported in part by NSF Grant DMS-2100389.%
} %

\def\myaddress{University of Washington, Department of Mathematics,
Seattle,WA 98195,USA}

\def\myemail{skovacs@uw.edu\xspace}

\usepackage{amsbsy}

\usepackage{marvosym}

\DeclareMathAlphabet{\smallchanc}{OT1}{pzc}%
                                 {m}{it}
\DeclareFontFamily{OT1}{pzc}{}
\DeclareFontShape{OT1}{pzc}{m}{it}%
             {<-> s * [1.100] pzcmi7t}{}
\DeclareMathAlphabet{\mathchanc}{OT1}{pzc}%
                                 {m}{it}
\newcommand{\mcA}{\mathchanc{A}}
\newcommand{\mcB}{\mathchanc{B}}

\newcommand{\mcD}{\mathchanc{D}}

\newcommand{\mcR}{\mathchanc{R}}

 \DeclareFontFamily{OMS}{rsfs}{\skewchar\font'60}
 \DeclareFontShape{OMS}{rsfs}{m}{n}{<-5>rsfs5 <5-7>rsfs7 <7->rsfs10 }{}
 \DeclareSymbolFont{rsfs}{OMS}{rsfs}{m}{n}
 \DeclareSymbolFontAlphabet{\scr}{rsfs}

\newcommand{\sA}{\mathscr{A}}
\newcommand{\sB}{\mathscr{B}}

\newcommand{\sE}{\mathscr{E}}
\newcommand{\sF}{\mathscr{F}}
\newcommand{\sG}{\mathscr{G}}
\newcommand{\sH}{\mathscr{H}}
\newcommand{\sI}{\mathscr{I}}

\newcommand{\sK}{\mathscr{K}}
\newcommand{\sL}{\mathscr{L}}
\newcommand{\sM}{\mathscr{M}}

\newcommand{\sO}{\mathscr{O}}

\newcommand{\sfA}{{\sf A}}
\newcommand{\sfB}{{\sf B}}
\newcommand{\sfC}{{\sf C}}

\newcommand{\sfG}{{\sf G}}

\newcommand{\sfK}{{\sf K}}

\newcommand{\sfM}{{\sf M}}
\newcommand{\sfN}{{\sf N}}

\newcommand{\sfQ}{{\sf Q}}

\newcommand{\bsfF}{\pmb{\sf F}}
\newcommand{\bsfG}{\pmb{\sf G}}

\newcommand{\bsff}{\pmb{\sf f}}
\newcommand{\bsfg}{\pmb{\sf g}}

\newcommand{\bC}{\mathbb{C}}
\newcommand{\bD}{\mathbb{D}}

\newcommand{\bH}{\mathbb{H}}

\newcommand{\bN}{\mathbb{N}}

\newcommand{\bZ}{\mathbb{Z}}

\def\frm{\mathfrak{m}}

\newcommand{\frp}{\mathfrak{p}}

\newcommand{\tf}{\widetilde{f}}

\DeclareSymbolFont{largesymbolsA}{U}{jkpexa}{m}{n}
\SetSymbolFont{largesymbolsA}{bold}{U}{jkpexa}{bx}{n}
\DeclareMathSymbol{\varprod}{\mathop}{largesymbolsA}{16}
\makeatletter
\newcommand{\LeftEqNo}{\let\veqno\@@leqno}
\makeatother
\newcommand{\ol}{\overline}

\newcommand{\ul}{\underline}

\newcommand{\into}{\hookrightarrow}

\newcommand{\onto}{\twoheadrightarrow}

\newcommand{\ideal}{\vartriangleleft}
\newcommand{\properideal}%
        {\subsetneq}

\newcommand{\homideal}{\vartriangleleft_{\mathrm{hom}}\!}

\newcommand{\wt}{\widetilde}

\newcommand{\leteq}{\colon\!\!\!=}
\newcommand{\col}{\colon}

\newcommand{\Irr}{{\mathchanc{Irr}}}
\newcommand\dash[1]{\rule[-.2ex]{#1}{.4pt}}

\DeclareMathOperator{\an}{{an}}

\DeclareMathOperator{\Ass}{{Ass}}

\DeclareMathOperator{\codim}{codim}
\newcommand{\coh}{\mathrm{coh}}

\DeclareMathOperator{\coker}{{coker}}

\DeclareMathOperator{\cone}{{Cone}}
\DeclareMathOperator{\depth}{{depth}}

\DeclareMathOperator{\exc}{Exc}

\newcommand{\filt}{{\operatorname{filt}}}

\DeclareMathOperator{\Hom}{Hom}
\newcommand{\sHom}[0]{{\mathchanc{Hom}}}

\DeclareMathOperator{\id}{{id}}
\DeclareMathOperator{\Id}{{Id}}

\DeclareMathOperator{\im}{{im}}

\DeclareMathOperator{\length}{{length}}

\newcommand{\lotimes}{\overset{L}{\otimes}}

\newcommand{\hotimes}[0]{%
  \protect{\ensuremath{\kern.1em{{%
  \raisebox{.5\depth}{$\scriptstyle[$}
  }}\kern-.25em\otimes\kern-.25em{%
  \raisebox{.5\depth}{$\scriptstyle]$}
  }\kern.1em}}}

\DeclareMathOperator{\Mor}{{Mor}}

\DeclareMathOperator{\Ob}{{Ob}}
\DeclareMathOperator{\obj}{{Obj}}

\DeclareMathOperator{\ob}{{Ob}}

\newcommand{\qc}{\mathrm{qc}}

\newcommand{\red}{\mathrm{red}}

\DeclareMathOperator{\Sing}{{Sing}}

\DeclareMathOperator{\Spec}{{Spec}}

\DeclareMathOperator{\supp}{{supp}}

\DeclareMathOperator{\skvert}{{\,\vert\,}}

\newcommand{\factor}[2]{\left. \raise .2em\hbox{\ensuremath{#1}\vphantom{$I^d$}}
\hskip -.1em \right/ \hskip -.4em \raise -.3em\hbox{\ensuremath{#2}}}%
\newcommand\mtimes[3]{{\varprod_{#1}^{#2}}_{\raise 1ex \hbox{\scriptsize #3}}}%

\newcommand{\myR}{{\mcR\!}}

\newcommand{\blank}{\dash{1em}}

\newcommand{\kdot}{{{\,\begin{picture}(1,1)(-1,-2)\circle*{2}\end{picture}\,}}}

\newcommand{\cmx}[1]{{#1}^{\raisebox{.15em}{\ensuremath\kdot}}}

\newcommand{\cx}[1]{{\sf #1}}

\newcommand{\dcx}[1]{{\omega}^\kdot_{#1}}

\newcommand{\Om}{\underline{\Omega}}

\def\dimcoh#1.#2.#3.{h^{#1}(#2,#3)}
\def\hypcoh#1.#2.#3.{\mathbb H_{\vphantom{l}}^{#1}(#2,#3)}
\def\loccoh#1.#2.#3.#4.{H^{#1}_{#2}(#3,#4)}
\def\dimloccoh#1.#2.#3.#4.{h^{#1}_{#2}(#3,#4)}
\def\lochypcoh#1.#2.#3.#4.{\mathbb H^{#1}_{#2}(#3,#4)}
\def\seslong#1.#2.#3.{0  \longrightarrow  #1   \longrightarrow 
 #2 \longrightarrow #3 \longrightarrow 0} 
\def\sesshort#1.#2.#3.{0
 \rightarrow #1 \rightarrow #2 \rightarrow #3 \rightarrow 0}
\def\dist#1.#2.#3.{  #1   \longrightarrow 
 #2 \longrightarrow #3 \stackrel{+1}{\longrightarrow} } 
\def\CDdist#1.#2.#3.{  #1   @>>>  #2  @>>>   #3 @>+1>> }  
\def\shortses#1.#2.#3.{0  \rightarrow  #1   \rightarrow 
 #2  \rightarrow   #3 \rightarrow  0}
\def\shortdist#1.#2.#3.{  #1   \rightarrow 
 #2  \rightarrow   #3 \stackrel{+1}{\rightarrow} }  
\def\ddist#1.#2.#3.#4.#5.#6.{\CD
#1 @>>> #2 @>>> #3 @>+1>> \\
@VVV @VVV @VVV \\
#4 @>>> #5 @>>> #6 @>+1>> 
\endCD}
\def\ddistun#1.#2.#3.#4.#5.#6.{\CD
#1 @>>> #2 @>>> #3 @>+1>> \\
@. @VVV @VVV  \\
#4 @>>> #5 @>>> #6 @>+1>> 
\endCD}
\def\Iff#1#2#3{
\hfil\hbox{\hsize =#1
\vtop{\noin #2}
\hskip.5cm 
\lower.5\baselineskip\hbox{$\Leftrightarrow$}\hskip.5cm
\vtop{\noin #3}}\hfil\medskip}
\newcommand{\union}\cup
\newcommand{\intersect}\cap
\newcommand{\Union}\bigcup
\newcommand{\Intersect}\bigcap
\def\myoplus#1.#2.{\underset #1 \to {\overset #2 \to \oplus}}

\newcommand{\resto}[1]{\raise -.5ex\hbox{$\vert$}_{#1}}

\def\qis{\,{\simeq}_{\text{qis}}\,}
\def\filtqis{\,{\simeq}_{\text{filt qis}}\,}

\newcommand{\ses}{short exact sequence\xspace}
\newcommand{\sess}{short exact sequences\xspace}

\newcommand{\dt}{distinguished triangle\xspace}
\newcommand{\dts}{distinguished triangles\xspace}

\newcommand{\DDB}{Deligne-Du~Bois\xspace}
\newcommand{\DB}{Du~Bois\xspace}
\newcommand{\rtl}{rational\xspace}

\newcommand{\sings}{singularities\xspace}

\newcommand\egpx{\ensuremath{\sfG^{X,D}_{p}}}

\newcommand\wegpx{\ensuremath{\wt\sfG\vphantom{\sfG}^{X,D}_{p}}}

\newcommand\hegpxz{\ensuremath{\what\sfG\vphantom{\sfG}^{X,\emptyset}_{p}}}

\newcommand\egnu[2]{\ensuremath{\nu^{#1}_{#2}}}

\newcommand\wegnu[2]{\ensuremath{\wt\nu^{#1}_{#2}}}

\newcommand\oegnu[2]{\ensuremath{\ol\nu^{#1}_{#2}}}
\newcommand\f{{\bsff}} 
\newcommand\uf{\ul{\bsff}\vphantom{\bsff}} 
\renewcommand\tf{\tilde{\bsff}} 

\newcommand\hz{\ensuremath{\footnotesize\protect{\textsf{\!\textit{h}}^{\,0}\!}}}

\begin{document}
\makeatletter
\definecolor{brick}{RGB}{204,0,0}
\def\@cite#1#2{{%
 \m@th\upshape\mdseries[{\sffamily\relsize{-.5}#1}{\if@tempswa,
   \sffamily\relsize{-.5}\color{brick} #2\fi}]}}
\newcommand{\sandor}{{\color{blue}{S\'andor \mdyydate\today}}}
\newenvironment{refmr}{}{}
%
\definecolor{refblue}{RGB}{0,0,128}
        \newcommand\refblue{\color{refblue}}
\newcommand\james{M\hskip-.1ex\raise .575ex \hbox{\text{c}}\hskip-.075ex Kernan\xspace}
\newcommand\ifft{if and only if\xspace}
\newcommand\tiff{if and only if\xspace}
\newcommand\sfref[1]{{\sf\protect{\relsize{-.5}\ref{#1}}}}
\newcommand\stepref[1]{{\sf\protect{\relsize{-.5}\refblue Step~\ref{#1}}}}
\newcommand\demoref[1]{{\sf(\protect{\relsize{-.5}\ref{#1}})}}
\renewcommand\eqref{\demoref}

%
\renewcommand\thesubsection{\thesection.\Alph{subsection}}
\renewcommand\subsection{
  \renewcommand{\sfdefault}{phv}
  \@startsection{subsection}%
  {2}{0pt}{-\baselineskip}{.2\baselineskip}{\raggedright
    \sffamily\itshape\relsize{-.5}
  }}
\renewcommand\section{
  \renewcommand{\sfdefault}{phv}
  \@startsection{section} %
  {1}{0pt}{\baselineskip}{.2\baselineskip}{\centering
    \sffamily
    \scshape
}}

\setlist[enumerate, 1]{itemsep=3pt,topsep=3pt,leftmargin=1.5em,font=\upshape,
  label={(\roman*)}}
\newlist{enumfull}{enumerate}{1}
\setlist[enumfull]{itemsep=3pt,topsep=3pt,leftmargin=3.25em,font=\upshape,
  label={(\thethm.\arabic*\/)}}
\newlist{enumbold}{enumerate}{1}
\setlist[enumbold]{itemsep=3pt,topsep=3pt,leftmargin=1.95em,font=\upshape,
  label={
    \textbf{({\it\textbf{\roman*}}\/)}}}
\newlist{enumalpha}{enumerate}{1}
\setlist[enumalpha]{itemsep=3pt,topsep=3pt,leftmargin=2em,font=\upshape,
  label=(\alph*\/)}
\newlist{widemize}{itemize}{1}
\setlist[widemize]{itemsep=3pt,topsep=3pt,leftmargin=1em,label=$\bullet$}
\newlist{widenumerate}{enumerate}{1}
\setlist[widenumerate]{itemsep=3pt,topsep=3pt,leftmargin=1.75em,label=(\roman*)}
\newlist{widenumalpha}{enumerate}{1}
\setlist[widenumalpha]{itemsep=3pt,topsep=3pt,leftmargin=1.5em,label=(\alph*)}
\newcounter{parentthmnumber}
\setcounter{parentthmnumber}{0}
\newcounter{currentparentthmnumber}
\setcounter{currentparentthmnumber}{0}
\newcounter{nexttag}
\newcommand{\setnexttag}{%
  \setcounter{nexttag}{\value{enumi}}%
  \addtocounter{nexttag}{1}%
}
\newcommand{\placenexttag}{%
\tag{\roman{nexttag}}%
}

\newenvironment{thmlista}{%
\label{parentthma}
\begin{enumerate}
}{%
\end{enumerate}
}
\newlist{thmlistaa}{enumerate}{1}
\setlist[thmlistaa]{label=(\arabic*), ref=\autoref{parentthm}\thethm(\arabic*)}
\newcommand*{\parentthmlabeldef}{%
  \expandafter\newcommand
  \csname parentthm\the\value{parentthmnumber}\endcsname
}
\newcommand*{\ptlget}[1]{%
  \romannumeral-`\x
  \ltx@ifundefined{parentthm\number#1}{%
    \ltx@space
    \parentthmundefined
  }{%
    \expandafter\ltx@space
    \csname mymacro\number#1\endcsname
  }%
}  
\newcommand*{\parentthmundefined}{\textbf{??}}
\parentthmlabeldef{parentthm}
\newenvironment{thmlistr}{%
\label{parentthm}
\begin{thmlistrr}}{%
\end{thmlistrr}}
\newlist{thmlistrr}{enumerate}{1}
\setlist[thmlistrr]{label=(\roman*), ref=\autoref{parentthm}(\roman*)}
\newcounter{proofstep}%
\setcounter{proofstep}{0}%
\newcommand{\pstep}[1]{%
  \smallskip
  \noindent
  \emph{{\sc Step \arabic{proofstep}:} #1.}\addtocounter{proofstep}{1}}
\newcounter{lastyear}\setcounter{lastyear}{\the\year}
\addtocounter{lastyear}{-1}
\newcommand\sideremark[1]{%
\normalmarginpar
\marginpar
[
\hskip .45in
\begin{minipage}{.75in}
\tiny #1
\end{minipage}
]
{
\hskip -.075in
\begin{minipage}{.75in}
\tiny #1
\end{minipage}
}}
\newcommand\rsideremark[1]{
\reversemarginpar
\marginpar
[
\hskip .45in
\begin{minipage}{.75in}
\tiny #1
\end{minipage}
]
{
\hskip -.075in
\begin{minipage}{.75in}
\tiny #1
\end{minipage}
}}
\newcommand\Index[1]{{#1}\index{#1}}
\newcommand\inddef[1]{\emph{#1}\index{#1}}
\newcommand\noin{\noindent}
\newcommand\hugeskip{\bigskip\bigskip\bigskip}
\newcommand\tinyskip{\vskip.25em}
\newcommand\smc{\sc}
\newcommand\dsize{\displaystyle}
\newcommand\sh{\subheading}
\newcommand\nl{\newline}
\newcommand\get{\input /home/kovacs/tex/latex/} 
\newcommand\Get{\Input /home/kovacs/tex/latex/} 
\newcommand\toappear{\rm (to appear)}
\newcommand\mycite[1]{[#1]}
\newcommand\myref[1]{(\ref{#1})}
\newcommand{\parref}[1]{\eqref{\bf #1}}
\newcommand\myli{\hfill\newline\smallskip\noindent{$\bullet$}\quad}
\newcommand\vol[1]{{\bf #1}\ } 
\newcommand\yr[1]{\rm (#1)\ } 
\newcommand\cf{cf.\ \cite}
\newcommand\mycf{cf.\ \mycite}
\newcommand\te{there exist\xspace}
\newcommand\tes{there exists\xspace}
\newcommand\st{such that\xspace}
\newcommand\CM{Cohen-Macaulay\xspace}
\newcommand\GR{Grauert-Riemenschneider\xspace}
\newcommand\notinclass{{\relsize{-.5}\sf [not discussed in class]}\xspace}
\newcommand\myskip{3pt}
\newtheoremstyle{bozont}{3pt}{3pt}%
     {\itshape}
     {}
     {\bfseries}
     {.}
     {.5em}
     {\thmname{#1}\thmnumber{ #2}\thmnote{\normalsize
        \ \rm #3}}
\newtheoremstyle{bozont-sub}{3pt}{3pt}%
     {\itshape}
     {}
     {\bfseries}
     {.}
     {.5em}
     {\thmname{#1}\ \arabic{section}.\arabic{thm}.\thmnumber{#2}\thmnote{\normalsize
 \ \rm #3}}
\newtheoremstyle{bozont-named-thm}{3pt}{3pt}%
     {\itshape}
     {}
     {\bfseries}
     {.}
     {.5em}
     {\thmname{#1}\thmnumber{#2}\thmnote{ #3}}
\newtheoremstyle{bozont-named-bf}{3pt}{3pt}%
     {}
     {}
     {\bfseries}
     {.}
     {.5em}
     {\thmname{#1}\thmnumber{#2}\thmnote{ #3}}
\newtheoremstyle{bozont-named-sf}{3pt}{3pt}%
     {}
     {}
     {\sffamily}
     {.}
     {.5em}
     {\thmname{#1}\thmnumber{#2}\thmnote{ #3}}
\newtheoremstyle{bozont-named-sc}{3pt}{3pt}%
     {}
     {}
     {\scshape}
     {.}
     {.5em}
     {\thmname{#1}\thmnumber{#2}\thmnote{ #3}}
\newtheoremstyle{bozont-named-it}{3pt}{3pt}%
     {}
     {}
     {\itshape}
     {.}
     {.5em}
     {\thmname{#1}\thmnumber{#2}\thmnote{ #3}}
\newtheoremstyle{bozont-sf}{3pt}{3pt}%
     {}
     {}
     {\sffamily}
     {.}
     {.5em}
     {\thmname{#1}\thmnumber{ #2}\thmnote{\normalsize
 \ \rm #3}}
\newtheoremstyle{bozont-sc}{3pt}{3pt}%
     {}
     {}
     {\scshape}
     {.}
     {.5em}
     {\thmname{#1}\thmnumber{ #2}\thmnote{\normalsize
 \ \rm #3}}
\newtheoremstyle{bozont-remark}{3pt}{3pt}%
     {}
     {}
     {\sffamily}
     {.}
     {.5em}
     {\thmname{#1}\thmnumber{ #2}\thmnote{\normalsize
 \ \rm #3}}
\newtheoremstyle{bozont-subremark}{3pt}{3pt}%
     {}
     {}
     {\scshape}
     {.}
     {.5em}
     {\thmname{#1}\ \arabic{section}.\arabic{thm}.\thmnumber{#2}\thmnote{\normalsize
 \ \rm #3}}
\newtheoremstyle{bozont-def}{3pt}{3pt}%
     {}
     {}
     {\bfseries}
     {.}
     {.5em}
     {\thmname{#1}\thmnumber{ #2}\thmnote{\normalsize
 \ \rm #3}}
\newtheoremstyle{bozont-reverse}{3pt}{3pt}%
     {\itshape}
     {}
     {\bfseries}
     {.}
     {.5em}
     {\thmnumber{#2.}\thmname{ #1}\thmnote{\normalsize
 \ \rm #3}}
\newtheoremstyle{bozont-reverse-sc}{3pt}{3pt}%
     {\itshape}
     {}
     {\scshape}
     {.}
     {.5em}
     {\thmnumber{#2.}\thmname{ #1}\thmnote{\normalsize
 \ \rm #3}}
\newtheoremstyle{bozont-reverse-sf}{3pt}{3pt}%
     {\itshape}
     {}
     {\sffamily}
     {.}
     {.5em}
     {\thmnumber{#2.}\thmname{ #1}\thmnote{\normalsize
 \ \rm #3}}
\newtheoremstyle{bozont-remark-reverse}{3pt}{3pt}%
     {}
     {}
     {\sc}
     {.}
     {.5em}
     {\thmnumber{#2.}\thmname{ #1}\thmnote{\normalsize
 \ \rm #3}}
\newtheoremstyle{bozont-def-reverse}{3pt}{3pt}%
     {}
     {}
     {\sffamily}
     {.}
     {.5em}
     {\thmnumber{{\relsize{-.5}(#2)}}\thmname{ {\relsize{-.25}#1}}\thmnote{{\relsize{-.5}\ \sf #3}}}
\newtheoremstyle{bozont-def-newnum-reverse}{3pt}{3pt}%
     {}
     {}
     {\bfseries}
     {}
     {.5em}
     {\thmnumber{#2.}\thmname{ #1}\thmnote{\normalsize
 \ \rm #3}}
\newtheoremstyle{bozont-def-newnum-reverse-plain}{3pt}{3pt}%
   {}
   {}
   {}
   {}
   {.5em}
   {\thmnumber{\!(#2)}\thmname{ #1}\thmnote{\normalsize
 \ \rm #3}}
\newtheoremstyle{bozont-number}{3pt}{3pt}%
   {}
   {}
   {}
   {}
   {0pt}
   {\thmnumber{\!(#2)} }
\newtheoremstyle{bozont-say}{3pt}{3pt}%
     {}
     {}
     {\sffamily}
     {.}
     {.5em}
     {\thmnumber{{\relsize{-.5}\S#2}}\thmname{ {\relsize{-.25}#1}}%
       \ \thmnote{
         \it #3}}
\newtheoremstyle{bozont-subsay}{3pt}{3pt}%
     {}
     {}
     {\sffamily}
     {.}
     {.5em}
     {\thmnumber{{\relsize{-.5}\S\S#2}}\thmname{ {\relsize{-.25}#1}}\thmnote{{\relsize{-.5}\ \sf #3}}}
\newtheoremstyle{bozont-step}{3pt}{3pt}%
   {\itshape}
   {}
   {\scshape}
   {}
   {.5em}
   {$\boxed{\text{\sc \thmname{#1}~\thmnumber{#2}:\!}}$}
\theoremstyle{bozont}    
\basetheorem{trouble}{trouble}{}
\basetheorem{cavalry}{cavalry}{}
\ifnum \value{are-there-sections}=0 {%
  \basetheorem{proclaim}{Theorem}{}
} 
\else {%
  \basetheorem{proclaim}{Theorem}{section}
} 
\fi
\maketheorem{thm}{proclaim}{Theorem}
\maketheorem{mainthm}{proclaim}{Main Theorem}
\maketheorem{cor}{proclaim}{Corollary} 
\maketheorem{cors}{proclaim}{Corollaries} 
\maketheorem{lem}{proclaim}{Lemma} 
\maketheorem{prop}{proclaim}{Proposition} 
\maketheorem{conj}{proclaim}{Conjecture}
\basetheorem{subproclaim}{Theorem}{proclaim}
\maketheorem{sublemma}{subproclaim}{Lemma}
\newenvironment{sublem}{%
\setcounter{sublemma}{\value{equation}}
\begin{sublemma}}
{\end{sublemma}}
\theoremstyle{bozont-sub}
\maketheorem{subthm}{equation}{Theorem}
\maketheorem{subcor}{equation}{Corollary} 
\maketheorem{subprop}{equation}{Proposition} 
\maketheorem{subconj}{equation}{Conjecture}
\theoremstyle{bozont-named-thm}
\maketheorem{namedthm}{proclaim}{}
\theoremstyle{bozont-sc}
\newtheorem{proclaim-special}[proclaim]{\specialthmname}
\newenvironment{proclaimspecial}[1]
     {\def\specialthmname{#1}\begin{proclaim-special}}
     {\end{proclaim-special}}
\theoremstyle{bozont-subremark}
\basetheorem{subremark}{Remark}{proclaim}
\maketheorem{subrem}{equation}{Remark}
\maketheorem{subnotation}{equation}{Notation} 
\maketheorem{subassume}{equation}{Assumptions} 
\maketheorem{subobs}{equation}{Observation} 
\maketheorem{subexample}{equation}{Example} 
\maketheorem{subex}{equation}{Exercise} 
\maketheorem{inclaim}{equation}{Claim} 
\maketheorem{subquestion}{equation}{Question}
\theoremstyle{bozont-remark}
\basetheorem{remark}{Remark}{proclaim}
\makeremark{subclaim}{subremark}{Claim}
\maketheorem{rem}{proclaim}{Remark}
\maketheorem{claim}{proclaim}{Claim} 
\maketheorem{notation}{proclaim}{Notation} 
\maketheorem{assume}{proclaim}{Assumptions} 
\maketheorem{assumeone}{proclaim}{Assumption} 
\maketheorem{obs}{proclaim}{Observation} 
\maketheorem{example}{proclaim}{Example} 
\maketheorem{examples}{proclaim}{Examples} 
\maketheorem{complem}{equation}{Complement}
\maketheorem{const}{proclaim}{Construction}   
\maketheorem{ex}{proclaim}{Exercise} 
\newtheorem{case}{Case} 
\newtheorem{subcase}{Subcase}   
\newtheorem{step}{Step}
\newtheorem{approach}{Approach}
\maketheorem{Fact}{proclaim}{Fact}
\newtheorem{fact}{Fact}
\newtheorem*{SubHeading*}{\SubHeadingName}%
\newtheorem{SubHeading}[proclaim]{\SubHeadingName}
\newtheorem{sSubHeading}[equation]{\sSubHeadingName}
\newenvironment{demo}[1] {\def\SubHeadingName{#1}\begin{SubHeading}}
  {\end{SubHeading}}%
\newenvironment{subdemo}[1]{\def\sSubHeadingName{#1}\begin{sSubHeading}}
  {\end{sSubHeading}} %
\newenvironment{demor}[1]{\def\SubHeadingName{#1}\begin{SubHeading-r}}
  {\end{SubHeading-r}}%
\newenvironment{subdemor}[1]{\def\sSubHeadingName{#1}\begin{sSubHeading-r}}
  {\end{sSubHeading-r}} %
\newenvironment{demo-r}[1]{%
  \def\SubHeadingName{#1}\begin{SubHeading-r}}
  {\end{SubHeading-r}}%
\newenvironment{subdemo-r}[1]{\def\sSubHeadingName{#1}\begin{sSubHeading-r}}
  {\end{sSubHeading-r}} %
\newenvironment{demo*}[1]{\def\SubHeadingName{#1}\begin{SubHeading*}}
  {\end{SubHeading*}}%
\maketheorem{defini}{proclaim}{Definition}
\maketheorem{defnot}{proclaim}{Definitions and notation}
\maketheorem{question}{proclaim}{Question}
\maketheorem{terminology}{proclaim}{Terminology}
\maketheorem{crit}{proclaim}{Criterion}
\maketheorem{pitfall}{proclaim}{Pitfall}
\maketheorem{addition}{proclaim}{Addition}
\maketheorem{principle}{proclaim}{Principle} 
\maketheorem{condition}{proclaim}{Condition}
\maketheorem{exmp}{proclaim}{Example}
\maketheorem{hint}{proclaim}{Hint}
\maketheorem{exrc}{proclaim}{Exercise}
\maketheorem{prob}{proclaim}{Problem}
\maketheorem{ques}{proclaim}{Question}    
\maketheorem{alg}{proclaim}{Algorithm}
\maketheorem{remk}{proclaim}{Remark}          
\maketheorem{note}{proclaim}{Note}            
\maketheorem{summ}{proclaim}{Summary}         
\maketheorem{notationk}{proclaim}{Notation}   
\maketheorem{warning}{proclaim}{Warning}  
\maketheorem{defn-thm}{proclaim}{Definition--Theorem}  
\maketheorem{convention}{proclaim}{Convention}  
\maketheorem{hw}{proclaim}{Homework}
\maketheorem{hws}{proclaim}{\protect{${\mathbb\star}$}Homework}
\maketheorem{obstacle}{trouble}{Road Block \#\!}
\maketheorem{avoid}{cavalry}{Detour \#\!}
\newtheorem*{ack}{\small Acknowledgment}
\newtheorem*{acks}{\small Acknowledgments}
\theoremstyle{bozont-number}
\theoremstyle{bozont-def}    
\maketheorem{defn}{proclaim}{Definition}
\maketheorem{subdefn}{equation}{Definition}
\theoremstyle{bozont-reverse}    
\maketheorem{corr}{proclaim}{Corollary} 
\maketheorem{lemr}{proclaim}{Lemma} 
\maketheorem{propr}{proclaim}{Proposition} 
\maketheorem{conjr}{proclaim}{Conjecture}
\theoremstyle{bozont-remark-reverse}
\newtheorem{SubHeading-r}[proclaim]{\SubHeadingName}
\newtheorem{sSubHeading-r}[equation]{\sSubHeadingName}
\newtheorem{SubHeadingr}[proclaim]{\SubHeadingName}
\newtheorem{sSubHeadingr}[equation]{\sSubHeadingName}
\theoremstyle{bozont-reverse-sc}
\newtheorem{proclaimr-special}[proclaim]{\specialthmname}
\newenvironment{proclaimspecialr}[1]%
{\def\specialthmname{#1}\begin{proclaimr-special}}%
{\end{proclaimr-special}}
\theoremstyle{bozont-remark-reverse}
\maketheorem{remr}{proclaim}{Remark}
\maketheorem{subremr}{equation}{Remark}
\maketheorem{notationr}{proclaim}{Notation} 
\maketheorem{assumer}{proclaim}{Assumptions} 
\maketheorem{obsr}{proclaim}{Observation} 
\maketheorem{exampler}{proclaim}{Example} 
\maketheorem{exr}{proclaim}{Exercise} 
\maketheorem{claimr}{proclaim}{Claim} 
\maketheorem{inclaimr}{equation}{Claim} 
\maketheorem{definir}{proclaim}{Definition}
\theoremstyle{bozont-def-newnum-reverse}    
\maketheorem{newnumr}{proclaim}{}
\theoremstyle{bozont-def-newnum-reverse-plain}
\maketheorem{newnumrp}{proclaim}{}
\theoremstyle{bozont-def-reverse}    
\maketheorem{defnr}{proclaim}{Definition}
\maketheorem{questionr}{proclaim}{Question}
\newtheorem{newnumspecial}[proclaim]{\specialnewnumname}
\newenvironment{newnum}[1]{\def\specialnewnumname{#1}\begin{newnumspecial}}{\end{newnumspecial}}
\theoremstyle{bozont-say}    
\newtheorem{say}[proclaim]{}
\theoremstyle{bozont-subsay}    
\newtheorem{subsay}[equation]{}
\theoremstyle{bozont-step}
\newtheorem{bstep}{Step}
\newcounter{thisthm} 
\newcounter{thissection} 
\newcommand{\ilabel}[1]{%
  \newcounter{#1}%
  \setcounter{thissection}{\value{section}}%
  \setcounter{thisthm}{\value{proclaim}}%
  \label{#1}}
\newcommand{\iref}[1]{%
  (\the\value{thissection}.\the\value{thisthm}.\ref{#1})}
\newcounter{lect}
\setcounter{lect}{1}
\newcommand\resetlect{\setcounter{lect}{1}\setcounter{page}{0}}
\newcommand\lecture{\newpage\centerline{\sfbf Lecture \arabic{lect}}
  \addtocounter{lect}{1}}
\newcommand\nnplecture{\hugeskip\centerline{\sfbf Lecture \arabic{lect}}
\addtocounter{lect}{1}}
\newcounter{topic}
\setcounter{topic}{1}
\newenvironment{topic}
{\noindent{\sc Topic 
\arabic{topic}:\ }}{\addtocounter{topic}{1}\par}
\counterwithin{equation}{proclaim}
\counterwithin{enumfulli}{equation}
\counterwithin{figure}{section} 
\newcommand\equinsect{\numberwithin{equation}{section}}
\newcommand\equinthm{\numberwithin{equation}{proclaim}}
\newcommand\figinthm{\numberwithin{figure}{proclaim}}
\newcommand\figinsect{\numberwithin{figure}{section}}
\newenvironment{sequation}{%
\setcounter{equation}{\value{thm}}
\numberwithin{equation}{section}%
\begin{equation}%
}{%
\end{equation}%
\numberwithin{equation}{proclaim}%
\addtocounter{proclaim}{1}%
}
\newcommand{\num}{\arabic{section}.\arabic{proclaim}}
\newenvironment{pf}{\smallskip \noindent {\sc Proof. }}{\qed\smallskip}
\newenvironment{enumerate-p}{
  \begin{enumerate}}
  {\setcounter{equation}{\value{enumi}}\end{enumerate}}
\newenvironment{enumerate-cont}{
  \begin{enumerate}
    {\setcounter{enumi}{\value{equation}}}}
  {\setcounter{equation}{\value{enumi}}
  \end{enumerate}}
\let\lenumi\labelenumi
\newcommand{\rmlabels}{\renewcommand{\labelenumi}{\rm \lenumi}}
\newcommand{\rmlabelsoff}{\renewcommand{\labelenumi}{\lenumi}}
\newenvironment{heading}{\begin{center} \sc}{\end{center}}
\newcommand\subheading[1]{\smallskip\noindent{{\bf #1.}\ }}
\newlength{\swidth}
\setlength{\swidth}{\textwidth}
\addtolength{\swidth}{-,5\parindent}
\newenvironment{narrow}{
  \medskip\noindent\hfill\begin{minipage}{\swidth}}
  {\end{minipage}\medskip}
\newcommand\nospace{\hskip-.45ex}
\newcommand{\sfbf}{\sffamily\bfseries}
\newcommand{\sfbfs}{\sffamily\bfseries\relsize{-.5}
}
\newcommand{\twidle}{\textasciitilde}
\makeatother

\newcounter{stepp}
\setcounter{stepp}{0}
\newcommand{\nextstep}[1]{%
  \addtocounter{stepp}{1}%
  \begin{bstep}%
    {#1}
  \end{bstep}%
  \noindent%
}
\newcommand{\resetsteps}{\setcounter{stepp}{0}}
\newcommand\change[1]{{\begin{color}{mydarkblue}\sfbf #1\end{color}}}
\newcommand\mchange[1]{{\begin{color}{mydarkblue}\mathbf #1\end{color}}}
\newcommand\keep[1]{{\begin{color}{fuchsia}\sfbf #1\end{color}}}
\newcommand\mrtl{$m$-rational\xspace}
\newcommand\premrtl{pre-$m$-rational\xspace}
\newcommand\smrtl{strict-$m$-rational\xspace}
\newcommand\premdb{pre-$m$-Du~Bois\xspace}%
\newcommand\premmdb{pre-$(m-1)$-Du~Bois\xspace}%
\newcommand\wmdb{weakly-$m$-Du~Bois\xspace}%
\newcommand\wmmdb{weakly-$(m-1)$-Du~Bois\xspace}%
\newcommand\mdb{$m$-Du~Bois\xspace}
\newcommand\smdb{strict-$m$-Du~Bois\xspace}%
\newcommand{\isomap}{\overset\simeq\longrightarrow}
\newcounter{startbiblioat}
\setcounter{startbiblioat}{0}
\let\oldthebibliography=\thebibliography
\let\oldendthebibliography=\endthebibliography
\renewenvironment{thebibliography}[1]{
    \oldthebibliography{#1}
    \setcounter{enumiv}{\value{startbiblioat}}
}{\oldendthebibliography}


\title
{Complexes of differential forms and singularities: The injectivity theorem} 
\author{\me} 
\date{\usdate\today} 
\thanks{\mythanks} 
\address{\myaddress}
\email{\myemail} 
%
\maketitle
\begin{abstract}
\vskip-2em
  Conjecture G of Popa, Shen, and Vo [PSV24] is confirmed: it is proved that for
  varieties with (m-1)-Du Bois singularities, the natural morphism from the
  Grothendieck dual of the m-th graded Du Bois complex to the Grothendieck dual of
  its zero-th cohomology sheaf is injective on cohomology.
\end{abstract}
\setcounter{tocdepth}{1}
\tableofcontents

\section{Introduction}
\noindent
The notion of rational singularities has been studied for quite a long time and it
has proved to be extremely useful. An extension of this notion, that of \DB \sings,
was introduced by Steenbrink \cite{Steenbrink83}.  \DB \sings started to become
better known and generate interest after Koll\'ar's conjecture that log canonical
\sings are \DB was confirmed in \cite{KK10}. This had opened up a slew of
applications in birational geometry and moduli theory 
cf.~\cite{SingBook,ModBook}.

Recently, following the original Hodge theoretic motivation of Steenbrink, Musta\c
t\u a, Olano, Popa, and Witaszek initiated the study of ``higher'' versions \DB
\sings for hypersurfaces in \cite{MR4583654}. This inspired \cite{MR4480883}, where
the terminology was coined. As a next step, the lci case was studied in
\cite{MR4491455}.  Friedman and Laza introduced higher rational singularities in
analogy with higher \DB \sings and studied the connections between the two notions in
\cite{FL-isolated,FL-lci}.  This connection was further studied in \cite{MP22}.
Most 
results in these 
papers were restricted to lci or isolated \sings and the natural question whether
those results hold without the lci or isolated assumptions arose.

The main goal of this paper is to extend some of the results obtained in these papers
to the general case. In addition, some other results are proved that are new even in
the lci case.


It turns out that the preparation of this paper took an unexpectedly long time and in
the meantime some of the results mentioned above were also extended to the general
case in \cite{SVV23}. As a result, there are some overlaps between this and that
paper (\cite{SVV23}). However, there are some small differences in the basic
definitions adopted and a definite difference in approach and philosophy, so it
seemed that these overlaps are not duplicating anything. In fact, the proof presented
here of the result analogous to the main result of \cite{SVV23} is different from the
approach taken there.
Another related paper, \cite{PSV24}, was also posted recently. 
In fact, the main result of this paper is stated there as a conjecture, along with
several other conjectures that follow from that one. In this paper, the main
conjecture of \cite{PSV24} is confirmed (for the relevant definitions and notation
see \autoref{sec:compl-diff-forms} and \autoref{sec:singularities}):
\begin{thm}[=\protect{\autoref{thm:key-injectivity}
    \cite[Conjecture~G]{PSV24}}]\label{conjG}
  Let $X$ be a variety with \premmdb \sings. Then the following natural map is
  injective on cohomology:
  \[
    \myR\sHom_X(\Om_X^m,\dcx X)\longrightarrow\myR\sHom_X(h^0(\Om_X^m),\dcx X).
  \]
\end{thm}
\begin{rem}
  An important feature of this result is that it is about injectivity of sheaves, and
  hence it is a \emph{local} statement. A major difficulty in the proof is to use the
  \emph{global} surjectivity coming from the degeneration of the Fr\"olicher spectral
  sequence in a local setting.
\end{rem}
%
This statement for $m=0$ first appeared in \cite[Thm.~3.3]{MR3617778} which was later
generalized to the case of pairs in \cite[Thm.~3.2]{KS13}. The natural extension to
higher degree forms was formulated and proved for local complete intersections in
\cite{MR4491455,MP22}. The same result for isolated singularities was proved in
\cite[Thm~D]{PSV24}.
As explained in \cite{PSV24}, \autoref{conjG} implies several other conjectures.
\begin{cor}[\protect{\cite[Conjecture H]{PSV24}}]\label{conjH}
  Let $X$ be a variety with only \premmdb \sings and assume that $X$ has \premdb
  \sings away from a closed subset of dimension $s$. Then
  \[
    h^i(\Om_X^m)=0 \qquad \text{for} \qquad 0<i<\depth h^0(\Om_X^m)-s-1.
  \]
\end{cor}
This follows by the argument \cite[p.14]{PSV24} proving that Theorem~D (of [ibid.])
implies Theorem~A (of [ibid.]),
Furthermore, 
the analogues of \autoref{conjG} and \autoref{conjH} for the \emph{intersection \DB
  complexes}, \cite[Conjectures~10.1 and 11.1]{PSV24}, also follow from
\autoref{conjG} and \autoref{conjH} via \cite[Thm.~10.3]{PSV24}. (For the definition
of the intersection \DB complexes, see \cite[Def.~3.3]{PP24}).

We also show that, as another application of \autoref{conjG}, there is a surjectivity
statement for local cohomology for \premmdb \sings as well:
\begin{thm}[=\protect{~\autoref{cor:surj-for-loc-coh-on-dual-of-Om-for-mdb}}]
  Let $X$ be a variety and $x\in X$ a point. Assume that $X$ is \premmdb near $x$.
  Then for each $q$ and $p\leq m$ the natural morphism is surjective:
  \[
    \xymatrix{%
      H^q_x(X,h^0(\Om_X^p)) \ar@{->>}[r] & \bH^q_x(X,\Om_X^p) }
  \]
\end{thm}
\noin We also obtain splitting criteria, reminiscent of \cite[Thm~2.3]{Kovacs99} and
\cite[Thm.~1]{Kovacs00b} for \premdb \sings in \autoref{thm:splitting-for-fX} and
\autoref{thm:splitting-for-Om}. The following is a simple consequence:
\begin{thm}\label{thm:m-rtl-implies-m-DB}
  Let $X$ be a variety with \premrtl (respectively \mrtl, respectively strict \mrtl)
  \sings. Then $X$ has \wmdb \sings (and hence also \premdb \sings), (respectively
  \mdb, respectively strict \mdb) \sings.
\end{thm}
\autoref{thm:m-rtl-implies-m-DB} was also obtained in \cite[Thm.~B, Cor.~C,
Thm.~D(b)]{SVV23} using different arguments. 

We will work with schemes essentially of finite type over $\bC$, but the results
easily extend to schemes essentially of finite type over any algebraically closed
field of characteristic zero.

The structure of the paper is as follows. After setting up some notations we review
some simple, but useful ancillary results in \autoref{sec:filtr-co-filtr}.
%
Then we review \emph{filtrations} and \emph{co-filtrations} and their \emph{hyper-}
analogues which are better suited for derived categories.

We review the various complexes of differential forms, their (hyper)filtration, and
(hyper)co-filtrations in \autoref{sec:compl-diff-forms} and define the classes of
\sings in which we are interested in \autoref{sec:singularities}. The behavior of the
complexes defined earlier with respect to hyperplane sections is studied in
\autoref{sec:hyperplane-sections} and with respect to cyclic covers in
\autoref{sec:cyclic-covers}. In \autoref{sec:filtered-deligne-db} we discuss some
Hodge theoretic aspects of these complexes and some consequences.
%
%
The technical core of the paper is \autoref{sec:eminence-grise}, which is where the
key surjectivity statement is proved. Due to the lack of exactness at a crucial
point, this had to be done by introducing an ancillary object, which however turns
out to be very powerful, so one might call it \emph{l'\'eminence grise}
\cite{GE-wiki}.
After this, there still remains an obstacle:
turning the surjectivity statement of \autoref{cor:surj-for-Om} into the main
injectivity result requires the use of Serre's vanishing.
In turn, Serre's vanishing requires the ambient space to be projective, while in the
main result we do not want to assume projectivity. This requires a funambulist's care
to balance the conditions and conclude the desired injectivity statement.
This obstacle is dealt with and the main result is also
proved 
in \autoref{sec:eminence-grise}.  Finally, several applications are presented in
\autoref{sec:applications}.

\begin{rem}
  In an earlier version of this paper a powerful method of Koll\'ar was adapted for
  derived categories. After 
  simplifying the proof, this was no longer necessary, so it is omitted from this
  version, however, it will be included in a forthcoming article
  \cite{Kovacs-in-prep}, exhibiting the method's versatility.
\end{rem}

\begin{ack}
  Participating in the AIM workshop \emph{Higher Du Bois and higher rational
    singularities} in October, 2024 was beneficial to me and useful for this
  article. I am grateful to everyone with whom I had a chance to talk about related
  or unrelated ideas. A subsequent visit to Harvard University was also very
  helpful. I would like to thank Mihnea Popa, Rosie Shen, and Duc Vo for many hours
  of interesting discussions. I am thankful to Sung Gi Park and Duc Vo for pointing
  out subtle, but crucial mistakes in earlier versions of this paper.  Last, but
  definitely not least, I want to thank Bradley Dirks, J\'anos Koll\'ar, Pat Lank,
  Haoming Ning, and Brian Nugent for useful comments and suggestions.
\end{ack}


\section{Preliminaries}\label{sec:filtr-co-filtr}
\noindent
In this article, a \emph{(complex) variety} will mean a reduced scheme of finite type
over $\bC$, the field of complex numbers. In particular, a variety (in this paper) is
not necessarily irreducible.

\begin{defini}
  Let $X$ be a complex scheme (i.e., a scheme essentially of finite type over $\bC$)
  of dimension n. Let $D_{\rm filt}(X)$ denote the derived category of filtered
  complexes of $\sO_{X}$-modules with ($\bC$-linear) differentials of order $\leq 1$
  and $D_{\rm filt, coh}(X)$ the subcategory of $D_{\rm filt}(X)$ of complexes
  $\cx K$, such that for all $i$, the cohomology sheaves of $Gr^{i}_{\rm filt}\cx K$
  are coherent cf.~\cite{DuBois81}, \cite{GNPP88}.  Note that these categories are
  the derived categories (with the corresponding restrictions) of differential graded
  modules over the differential graded algebra $\Omega_X^\kdot$. Let $D(X)$ and
  $D_{\rm coh}(X)$ denote the derived categories with the same definition except that
  the complexes are assumed to have the trivial filtration.  The superscripts
  $+, -, b$ carry the usual meaning (bounded below, bounded above, bounded). Recall
  that isomorphism in these categories is defined by quasi-isomorphism of complexes.
  A sheaf $\sF$ is also considered a complex $\sF^\kdot$ with $\sF^0=\sF$ and
  $\sF^i=0$ for $i\neq 0$.  If $\cx K$ is a complex in any of these categories, then
  $h^i(\cx K)$ denotes the $i$-th cohomology sheaf of $\cx K$.

  Recall that if $\imath:\Sigma \into X$ is a closed embedding of schemes then
  $\imath_*$ is exact and hence $\myR\imath_*=\imath_*$.
  Accordingly 
  if $\sfA\in\ob D(\Sigma)$, then, as usual for sheaves, we will drop $\imath_*$ from
  the notation of the object $\imath_*\sfA$. In other words, we will, without further
  warning, consider $\sfA$ an object in $D(X)$.

  Let $X$ be a variety and $\Sigma\subseteq X$ a closed subset. A \emph{log
    resolution} of the pair $(X,\Sigma)$ is a proper birational morphism $\pi:Y\to X$
  such that $(\pi^{-1}_*\Sigma+E)_\red$ is an snc 
  divisor where $E=\exc(\pi)$ is the exceptional set of $\pi$. A \emph{strong log
    resolution} of $(X,\Sigma)$ is a log resolution that is an isomorphism over the
  locus where $(X,\Sigma)$ is an snc pair.
  Cubical varieties and cubical hyperresolutions will be used following the
  terminology of \cite[Chapter~5]{MR2393625} and \cite[Appendix~2]{MR2796408}. In
  particular, a \emph{hyperresolutions} will always refer to a cubical
  hyperresolution, which (here) always have finitely many components.
  
  We will also use the following notation which seems to be becoming standard: Let
  $X$ be a scheme of pure dimension $n$ that admits a (normalized) dualizing complex,
  denoted by $\dcx X$. Define the \emph{Grothendieck duality functor} on the bounded
  derived category of quasi-coherent sheaves:
  \begin{equation}
    \label{eq:64}
    \bD_X(\blank)\leteq \myR\sHom_X(\blank, \dcx X)[-n]. 
  \end{equation}
  Note that an advantage of this notation is that it shifts the Grothendieck dual to
  human readable form: For instance, if $X$ is a smooth irreducible variety of
  dimension $n$, then $\bD_X(\Omega_X^p)\simeq \Omega_X^{n-p}$.
\end{defini}

\subsection{Prime avoidance}
We will use the following notation throughout this subsection.
\begin{notation}\label{assumption}
  Let $X$ be a noetherian scheme, $\sF$ a coherent $\sO_X$-module, $\sL$ a semi-ample
  invertible sheaf on $X$,  $s\in\Gamma(X,\sL^m)$ a general section for some
  $m\gg 0$, and $Z(s)$ its zero locus.
\end{notation}

\noin
The following is a well-known statement. It is stated and proved here to make its use
simpler.

\begin{lem}[(Global prime avoidance)]\label{lem:prime-avoidance}
  $Z(s)$ does not contain any associated point of $\sF$.
\end{lem}

\begin{proof}
  This proof is modeled after the proof of
  \cite[\href{https://stacks.math.columbia.edu/tag/09NV}{Tag 09NV}]{stacks-project}.
  As $X$ is noetherian, $\Ass(\sF)$ is finite. %
  Let $S\leteq \oplus_{m\in\bN}H^0(X,\sL^m)$ and for each $x\in\Ass(\sF)$ let
  $\frp_x\homideal S$ denote the homogenous ideal of sections vanishing at $x$. This
  is clearly a prime ideal in $S$ and as some power of $\sL$ is generated by global
  sections, $S_+\not\subseteq \frp_x$ for any $x\in\Ass(\sF)$. Then it follows from
  the homogenous prime avoidance lemma
  \cite[\href{https://stacks.math.columbia.edu/tag/00JS}{Tag 00JS}]{stacks-project},
  that there exists a homogenous element $s\in S_+$ such that $s\not\in\frp_x$ for
  any $x\in\Ass(\sF)$. This is equivalent to the statement.
\end{proof}

\begin{cor}\label{cor:prime-avoidance--injectivity}
  Using \autoref{assumption}, the natural morphism induced by multiplication by $s$
  is injective:
  \[
    \xymatrix{
      \sF\otimes\sL^{-m}\,\ar@{^(->}[r]^- {s\cdot} & \sF.  
    }%
  \]
\end{cor}

\begin{proof} Assume that there exists an $x\in X$ such that $
  \sF_x{\color{gray}(\otimes\sL_x^{-m})}\overset {s_x\cdot}\longrightarrow\sF_x$ is
  not injective. It follows that $s$ is not invertible at $x$ and hence, in
  particular, $x\in Z(s)$. Furthermore, then $x\in\Ass(\sF)$ by
  \cite[\href{https://stacks.math.columbia.edu/tag/0AVL}{Tag
    0AVL}]{stacks-project}, 
  which contradicts \autoref{lem:prime-avoidance}.
\end{proof}

\noin A similar statement holds without assuming that $s$ is general if $\sF$ is
torsion-free:

\begin{lem}\label{lem:torsion-free-tensor}
  If $\sF$ is torsion-free and $H$is an effective Cartier divisor on $X$, then the
  natural morphism
  \[
    \xymatrix{%
      \sF\otimes\sO_X(-H)\,\ar@{^(->}[r] & \sF }
  \]
   induced by $H$ is injective.
\end{lem}
\begin{proof}
  This morphism is an isomorphism on $X\setminus H$ and hence its kernel is supported
  on $H$.  As $\sF$, and hence $\sF\otimes\sO_X(-H)$ is torsion-free, this morphism
  is injective everywhere.
\end{proof}

\noin The following applies in both situations of
\autoref{cor:prime-avoidance--injectivity} and \autoref{lem:torsion-free-tensor}.
\begin{lem}\label{lem:ltimes-is-times}
  Let $H$ be an effective Cartier divisor on $X$, and assume that the natural
  morphism
  \[
    \xymatrix{%
      \sF\otimes\sO_X(-H)\,\ar@{^(->}[r] & \sF 
    }
  \]
  induced by $H$ is injective.  Then $\sF$ and $\sO_H$ are Tor-independent, i.e.,
  $\sF\lotimes\sO_H\simeq\sF\otimes\sO_H.$ 
\end{lem}

\begin{proof}
    The \ses,
  \begin{equation}
    \label{eq:75}
    \xymatrix{%
      0\ar[r] & \sO_X(-H)\ar[r] & \sO_X\ar[r] & \sO_H\ar[r] & 0
    }    
  \end{equation}
  gives a locally free resolution of $\sO_H$, so $\xymatrix{%
    \sF\lotimes\sO_H\simeq\bigg[\sF\otimes\sO_X(-H)\ar[r] &
    \sF\bigg]\simeq\sF\otimes\sO_H}$.
\end{proof}



\subsection{Serre duality in $D_{\rm coh}^b(X)$}
\noindent
\label{subsec:groth-dual}%
%
%
%
Recall the \emph{Grothendieck duality functor}
from \autoref{eq:64} ($X$ is of pure dimension $n$):
\[
  \bD_X(\blank)\leteq \myR\sHom_X(\blank, \dcx X)[-n].
\]

\begin{lem}\label{lem:duality-over-a-field}
  Let $k$ be a field and $\cx A\in \Ob D^b(\Spec k)$. Then
  \[
    h^{-j}\left(\myR\Hom_k(\cx A,k)\right) \simeq \Hom_k(h^j(\cx A), k).
  \]
\end{lem}

\begin{proof}
  $\Hom_k(\blank ,k)$ is an exact contravariant functor, so it commutes with
  cohomology.
\end{proof}

\noin
The following is essentially Serre duality for complexes.

\begin{lem}\label{lem:grothendieck-duality}
  Let $X$ be proper of pure dimension $n$ over a field $k$,
  and $\cx A\in\Ob D_{\rm coh}^b(X)$. Then for $\forall j$, 
  \[
    \bH^{j}(X,\cx A
    )^{\vee} \simeq \bH^{n-j}(X,\bD_X(\cx A)
    )
  \]
\end{lem}

\begin{proof}
  Let $\pi:X\to\Spec k$ denote the structure morphism of $X$. Then by Grothendieck
  duality, 
  \begin{equation*}
    \label{eq:2}
    \myR\pi_*\bD_X(\cx A)[n]\simeq \myR\pi_*\myR\sHom_X(\cx A, \omega_X^\kdot)\simeq
    \myR \Hom_k(\myR\pi_*(\cx A), k). 
  \end{equation*}
  The $h^{-j}$-th cohomology of the left hand side 
  is $\bH^{n-j}(X,\bD_X(\cx A)
  )$. By \autoref{lem:duality-over-a-field} the $h^{-j}$-th cohomology of the right
  hand side is $\bH^j(X,\cx A
  )^{\vee}$.
\end{proof}

\noin Recall that Grothendieck duality implies that the actions of $\bD_X$ and
$\bD_Z$ agree on any object that is supported on a closed subscheme $Z\subseteq X$:

\begin{lem}\label{lem:Dz-is-same-as-Dx}
  Let $\jmath: Z\into X$ be a closed subscheme of pure codimension $r$ of a complex
  scheme $X$ of pure dimension $n$ and $\cx A\in\obj D_{\rm coh}^b(Z)$. Then
  $\jmath_*\bD_Z(\cx A)\simeq \bD_X(\jmath_*\cx A)[r]$. \hfill \qed
\end{lem}

\begin{proof}
  Let $n\leteq\dim X$. Then Grothendieck duality implies that
  \begin{equation*}
    \jmath_*\bD_Z(\cx A)\simeq
    \myR\jmath_*\myR\sHom_Z(\cx A,
    \omega_Z^\kdot)[-(n-r)]\simeq \myR\sHom_X(\myR\jmath_*\cx A,
    \omega_X^\kdot)[-n+r]\simeq \bD_X(\jmath_*\cx A)[r]. \qedhere
  \end{equation*}
\end{proof}

\noin
Let us also recall the following.

\begin{lem}\label{lem:conjugate-ss}
  Let $X$ be a scheme and $\sfA\in\obj D(X)$. Then for any $q\in\bN$ there exists a
  natural map $\bH^q(X,\sfA)\to H^0(X, h^q(\sfA))$.
\end{lem}

\begin{rem}
  This map appears in the conjugate spectral sequence corresponding to
  $\myR\Gamma(X,\blank)$.
\end{rem}

\begin{proof}
  Consider the canonical truncation defined in
  \cite[\href{https://stacks.math.columbia.edu/tag/0118}{Tag 0118}]{stacks-project}:
  $\tau_{\leq q}\sfA\to\sfA$ and note that this induces an isomorphism
  $\bH^q(X,\tau_{\leq q}\sfA)\simeq\bH^q(X,\sfA)$.
  On the other hand, by definition, there exists a natural map
  $\tau_{\leq q}\sfA\to h^q(\sfA)[-q]$.  Combining these implies that there is a
  natural map
  \begin{equation}
    \label{eq:92}
    \xymatrix{%
      \bH^q(X,\sfA) \simeq \bH^q(X,\tau_{\leq q}\sfA)\ar[r] &
      \bH^q(X,h^q(\sfA)[-q])\simeq H^0(X,h^q(\sfA)).
    }\qedhere
  \end{equation}
\end{proof}

\subsection{Filtrations and co-filtrations}
\noindent
\begin{defini}
  By analogy with filtrations,
  cf.~\cite[\href{https://stacks.math.columbia.edu/tag/0121}{Tag
    0121}]{stacks-project}, we define \emph{co-filtrations} as follows: Let $\sfA$ be
  an object of an abelian category. A \emph{co-filtration} of $\sfA$ is a sequence of
  epimorphisms.
  \[
    \sfA \onto \dots \onto F_{n}\sfA \onto F_{n-1}\sfA \onto \dots \onto 0.
  \]
  {The reader is invited to formulate the analogues of the statements in
    \cite[\href{https://stacks.math.columbia.edu/tag/0121}{Tag 0121}]{stacks-project}
    for co-filtrations.}
  In this article we will only consider finite, separated, and exhaustive
  co-filtrations, i.e., such that for some appropriate $n,m\in\bZ$, $\sfA=F_nA$ and
  $F_m\sfA=0$ and we will simply call them co-filtrations.  After possibly relabeling
  our 
  co-filtration we may assume that there is a sequence of epimorphisms,
  \[
    \sfA = F_{n}\sfA \onto F_{n-1}\sfA \onto \dots \onto F_0\sfA= 0.
  \]
  A morphisms of co-filtered objects is a \emph{co-filtered morphism} if it respects
  the co-filtration. 
\end{defini}

\begin{lem}\label{lem:filtr-co-filtr}
  Let $(\sfA,F^\kdot)$ be a filtered object,
  $\sfA\supseteq \dots \supseteq F^n \supseteq F^{n+1} \supseteq \dots$, as in
  \cite[\href{https://stacks.math.columbia.edu/tag/0121}{Tag
    0121(2)}]{stacks-project}. Then there exists a unique (up to isomorphism) natural
  co-filtration, $F_\kdot$, of $\sfA$ which is dual to $F^\kdot$ in the sense that
  for each $p\in\bZ$ there exists a \ses, 
  \begin{equation}
    \label{eq:19}
    \xymatrix{%
      0 \ar[r] & F^{p+1}\sfA \ar[r] & \sfA \ar[r] & F_p\sfA \ar[r] & 0.
    }
  \end{equation}
\end{lem}

\begin{proof}
  For $p\in\bZ$ let $F_p\sfA\leteq \coker(F^{p+1}\sfA \to \sfA)$. This implies the
  existence of the short exact sequence in the statement. Next, consider this \ses
  for $p$ and $p-1$ and the morphism on the left hand side given by the original
  filtration as indicated on the diagram:
  \[
    \xymatrix{%
      0 \ar[r] & \ar@{^(->}[d] F^{p+1}\sfA \ar[r] & \ar[d]^{\id_\sfA} \sfA \ar[r] &
      F_p\sfA \ar[r] \ar@{-->}[d]^{\phi_p} & 0 \\
      0 \ar[r] & F^{p}\sfA \ar[r] & \sfA \ar[r]^-\alpha & F_{p-1}\sfA \ar[r] & 0.  }
  \]
  It follows that there exists a morphism $\phi_p$ (indicated by the dashed arrow).
  Furthermore, as $\alpha\circ\id_\sfA$ is surjective by definition, so is $\phi_p$
  and hence we obtain a co-filtration of $\sfA$ by the $F_p\sfA$.
  The uniqueness of the co-filtration is straightforward from \autoref{eq:19}.
\end{proof}

\begin{defini}\label{def:associated-co-filtration}
  Let $(\sfA,F^\kdot)$ be a filtered object. Then the co-filtration $F_\kdot$
  constructed in \autoref{lem:filtr-co-filtr} (satisfying \autoref{eq:19}) will be
  called the \emph{associated co-filtration of $F^\kdot$}.

  From this point forward filtrations will be considered along with their
  co-filtrations and accordingly we will drop the ``dot'' in the super- or subscript
  unless we need to distinguish between filtrations or co-filtrations. This is
  consistent with our other terminology: 
\end{defini}

\begin{lem}\label{lem:morphs-of-filtr-co-filtr}
  Let $\psi:(\sfA,F)\to (\sfB,G)$ be a filtered morphism of filtered
  objects. Consider $\sfA$ and $\sfB$ as co-filtered objects via the associated
  co-filtrations of $F$ and $G$ respectively. Then the morphism $\psi:\sfA\to\sfB$ is
  also a co-filtered morphism.
\end{lem}

\begin{proof}
  The fact that $\psi$ is a filtered morphism implies that for each $p\in\bZ$ there
  exists a commutative diagram with the solid arrows:
  \[
    \xymatrix{%
      0 \ar[r] & \ar[d] F^{p+1}\sfA \ar[r] & \ar[d]^{\psi} \sfA \ar[r] &
      F_p\sfA \ar[r] \ar@{-->}[d] & 0 \\
      0 \ar[r] & G^{p+1}\sfB \ar[r] & \sfB \ar[r] & G_{p}\sfB \ar[r] & 0.  }
  \]
  This in turn implies the existence of the dashed arrow which commutes with the rest
  of the diagram and hence $\psi$ is indeed a co-filtration morphism.
\end{proof}

\subsection{Hyperfiltrations and co-hyperfiltrations}\label{subsec:hyperf-co-hyperf}

The derived category of filtered objects is a mixed bag. The filtration is defined on
the representing complex and hence dealing with these filtrations is sometimes
cumbersome. Furthermore, filtrations may not be compatible with arbitrary functors,
so applying such functors can kill the filtration.  For these reasons, it is often
convenient to treat filtrations as %
\emph{hyperfiltration}s (see below)
which removes the potential dependence on actual complexes representing the original
filtered object and are compatible with arbitrary functors.  This philosophy has
already been adopted for example in \cite{Kovacs97c,Kovacs05a,KT21}.

\begin{notation}
  In this section we will work with objects in a triangulated
  category, $\mcD$.
  The main example to keep in mind is the derived category of an abelian category:
  let $\mcA$ be an abelian category, $C(\mcA)$ the category of complexes of objects
  in $\mcA$, $K(\mcA)$ the homotopy category of complexes of objects in $\mcA$, and
  $\mcD(\mcA)$ the derived category of $\mcA$.
\end{notation}

\begin{defini}[\protect{\cite[\S\S1.2]{Kovacs05a}}]\label{def:hyperfiltrations}%
  A \emph{hyperfiltration}, $\bsfF^\kdot=\bsfF^\kdot(\sfA)$, of an object
  $\sfA\in \Ob\mcD$ is a set of objects $\bsfF^p\in\Ob\mcD$ and a set of morphisms
  $(\phi^p=\phi^p(\bsfF^\kdot):\bsfF^{p+1}\to \bsfF^{p})\in\Mor\mcD$ for each
  $p\in\bN$ such that $\bsfF^0\simeq\sfA$.  This last condition implies that our
  hyperfiltrations are \emph{exhaustive}. One could make the definition more general,
  or even exhaustive without having to declare that $\sfA$ is part of the
  filtration. However, we will only use this type of hyperfiltrations, so there is no
  need for the more general setup, at least not in this article.  This setup has the
  advantage that $\sfA$ is actually part of the hyperfiltration $\cmx\bsfF$, which
  allows us to suppress $\sfA$ from the notation. 

  In order to avoid having to worry about the range of indices, we extend our
  hyperfiltrations (and filtrations as well) with the following definition: Let
  $\bsfF^p\leteq\sfA$ and $(\phi^p:\bsfF^{p+1}\to \bsfF^{p})\leteq\id_{\sfA}$ for
  each $p\in\bZ$, $p<0$. 
  A hyperfiltration is \emph{finite} if there exists a $p_0\in\bN$ such that
  $\bsfF^p\simeq 0$ for each $p\geq p_0$. In this case the \emph{length} of a
  hyperfiltration $\cmx\bsfF$, denoted by $\length{\cmx\bsfF}$, is the smallest $p_0$
  for which the above property holds.
  For a $p\in\bZ$ the \emph{$p^\text{th}$-associated graded complex} of the
  hyperfiltration $\cmx{\bsfF}$ is defined as the mapping cone of $\phi^p$:
  $\bsfG^p\leteq\bsfG^p_{\bsfF}(\sfA)\leteq \cone(\phi^p)$.  The associated natural
  morphisms will be denoted by notation $\varrho^p:\bsfF^p\to\bsfG^p$ and
  $\varepsilon^p:\bsfG^p\to \bsfF^{p+1}[1]$.
  The morphisms $\phi^p$ will be called the \emph{interior morphisms} of
  $\bsfF^\kdot$.  A hyperfiltration is called \emph{non-redundant} if an interior
  morphism $\phi^p:\bsfF^{p+1}\to\bsfF^p$ is an isomorphism only if
  $p\geq\length\cmx\bsfF$ or $p<0$.
  %
  A \emph{morphism} of hyperfiltrations $\sigma^\kdot:\bsfF_1^\kdot\to\bsfF_2^\kdot$
  is a collection of morphisms $\sigma^p:\bsfF_1^p\to\bsfF_2^p$ for each $p\in\bZ$,
  which are compatible with the interior morphisms of the hyperfiltrations
  $\bsfF_1^\kdot$ and $\bsfF_2^\kdot$, i.e., such that
  $\sigma^p\circ\phi_1^p=\phi_2^p\circ\sigma^{p+1}$.

  A \emph{hyperfiltered morphism} of hyperfiltered objects
  $\sigma:(\sfA,\cmx\bsfF_\sfA)\to (\sfB,\cmx\bsfF_\sfB)$ is a morphism of
  hyperfiltrations $\sigma^\kdot:\cmx\bsfF_\sfA\to\cmx\bsfF_\sfB$. 
  Because of the convention that $\bsfF_\sfA^0=\sfA$ and $\bsfF_\sfB^0=\sfB$, this
  includes a morphism $\sigma^0:\sfA\to\sfB$ of the underlying objects.  A
  \emph{hyperfiltered isomorphism} of hyperfiltered objects is a hyperfiltered
  morphism that has an inverse which is also a hyperfiltered morphism.

  A \emph{(finite) co-hyperfiltration}, $\bsfF_\kdot=\bsfF_\kdot(\sfA)$, of an object
  $\sfA\in \Ob\mcD$ is a set of objects $\bsfF_p\in\Ob\mcD$ and a set of
  \emph{interior morphisms}
  $(\phi_{p}=\phi_p(\bsfF_\kdot):\bsfF_{p+1}\to \bsfF_{p})\in\Mor\mcD$ for each
  $p\in\bN$, such that $\bsfF_p\simeq 0$ for $p<0$, and for some $n\in\bN$,
  $\bsfF_m\simeq\sfA$ and $\phi_m=\id_\sfA$ for each $m\geq n$. The smallest such $n$
  will be called the \emph{height} of $\bsfF_\kdot$.  The \emph{length} of a
  co-hyperfiltration $\bsfF_\kdot$, denoted by $\length{\bsfF_\kdot}$, is $n-p_0$,
  where $n$ is the height of $\bsfF_\kdot$ and $p_0\in\bN$ is the largest integer
  such that $\bsfF_p\simeq 0$ for each $p< p_0$.  As in the case of hyperfiltrations,
  we extend our co-hyperfiltrations to negative indices by $\bsfF_p\leteq 0$ and
  $\phi_p\leteq 0: \bsfF_{p+1}\to\bsfF_p$ for $p\in\bZ$, $p<0$.  A co-hyperfiltration
  is called \emph{non-redundant} if an interior morphism
  $\phi_p:\bsfF_{p+1}\to\bsfF_p$ is an isomorphism only if $p< n-\length\bsfF_\kdot$
  or $p\geq n$, where $n$ is the height of $\bsfF_\kdot$.
\end{defini}

\begin{rem}%
  Philosophically, a co-hyperfiltration should start with $\bsfF^0=\sfA$ and the
  indexing should go with negative integers. This would allow for not necessarily
  finite co-hyperfiltrations. However, in this article we will only use finite
  (co-)hyperfiltrations and the above termninology makes the already complex notation
  more bearable and human readable.%
\end{rem}

\noin
Hyperfiltrations are better suited for derived categories than filtrations in many
ways. For instance, the following is a straightforward consequence of the definition.

\begin{lem}\label{lem:image-of-hyperfilt}
  Let $\Phi:\mcD_1\to\mcD_2$ be a (covariant) functor between derived categories and
  $(\sfA,\bsfF^\kdot)$ a hyperfiltered object in $\mcD_1$ with associated graded
  complexes $\bsfG^p$. Then $(\Phi(\sfA),\Phi(\bsfF^\kdot))$ is a hyperfiltered
  object in $\mcD_2$ and the $p^\text{th}$-associated graded complex of
  $\Phi(\bsfF^\kdot)$ is $\Phi(\bsfG^p)$. \hfill\qed
\end{lem}

\noin
We also have the analogue of \autoref{lem:filtr-co-filtr}:

\begin{lem}\label{lem:hyperf-co-hyperf}
  Let $(\sfA,\bsfF^\kdot)$ be a finite hyperfiltered object,
  $\dots\to\bsfF^{n+1}\to\bsfF^n\to\dots\to\bsfF^0\simeq\sfA$, as in
  \autoref{def:hyperfiltrations}.  Then there exists a unique (up to isomorphism)
  natural co-hyperfiltration, $\bsfF_\kdot$, of $\sfA$ which is dual to $\bsfF^\kdot$
  in the sense that for $\forall p\in\bZ$ there exists a \dt of objects,
  \begin{equation}
    \label{eq:21}
    \xymatrix{%
      \bsfF^{p+1}\sfA \ar[r] & \sfA \ar[r] & \bsfF_p\sfA \ar[r]^-{+1} & .
    }
  \end{equation}
\end{lem}

\begin{proof}
  For $p\in\bZ$ let $\bsfF_p\sfA\leteq \cone(\bsfF^{p+1}\sfA \to \sfA)$ be the
  mapping cone of the indicated morphism. This implies the existence of the \dt in
  the statement. Furthermore, consider this \dt for $p$ and $p-1$ and the morphism on
  the left hand side given by the original hyperfiltration as indicated on the
  diagram:
  \[
    \xymatrix{
      \ar[d] \bsfF^{p+1}\sfA \ar[r] & \ar[d]^{\id_\sfA} \sfA \ar[r] &
      \bsfF_p\sfA \ar[r]^-{+1} \ar@{-->}[d]^{\phi_p} &  \\
      \bsfF^{p}\sfA \ar[r] & \sfA \ar[r] & \bsfF_{p-1}\sfA \ar[r]^-{+1} & .  }
  \]
  It follows that there exists a morphism $\phi_p$ (indicated by the dashed arrow),
  hence we obtain a co-filtration of $\sfA$ by the $\bsfF_p\sfA$.
  The uniqueness (in the derived category) of the co-filtration is straightforward
  from \autoref{eq:21}.
  Note that the finiteness of $\bsfF^\kdot$ is (only) required for $\bsfF_\kdot$ to
  be finite and exhaustive.
\end{proof}

\begin{defini}\label{def:associated-co-hyperfiltration}
  Let $(\sfA,\bsfF^\kdot)$ be a hyperfiltered object. Then the co-hyperfiltration
  $\bsfF_\kdot$ constructed in \autoref{lem:hyperf-co-hyperf} (satisfying
  \autoref{eq:21}) will be called the \emph{associated co-hyperfiltration of
    $\bsfF^\kdot$}.
  As with filtrations, from this point forward hyperfiltrations will be considered to
  be a unit with their co-hyperfiltrations and accordingly we will drop the ``dot''
  in the super- or subscript unless we need to distinguish between hyperfiltrations
  or co-hyperfiltrations. Again, this is consistent with our other terminology as
  shown by the next statement, an analogue of \autoref{lem:morphs-of-filtr-co-filtr}.
\end{defini}

\begin{lem}\label{lem:morphs-of-hyperfiltr-co-hyperfiltr}
  Let $\sigma:(\sfA,\cmx\bsfF_\sfA)\to (\sfB,\cmx\bsfF_\sfB)$ be a hyperfiltered
  morphism of hyperfiltered objects. Consider $\sfA$ and $\sfB$ as co-hyperfiltered
  objects via the associated co-hyperfiltrations $\bsfF^\sfA_\kdot$ and
  $\bsfF^\sfB_\kdot$ respectively. Then the morphism $\sigma:\sfA\to\sfB$ is also a
  co-hyperfiltered morphism.
\end{lem}

\begin{proof}
  The fact that $\sigma$ is a hyperfiltered morphism implies that for each $p\in\bZ$
  there exists a commutative diagram with the solid arrows:
  \[
    \xymatrix@R1.5em{%
      \ar[d] \bsfF_\sfA^{p+1}\sfA \ar[r] & \ar[d]^{\sigma} \sfA \ar[r] &
      \bsfF^\sfA_p\sfA \ar@{-->}[d] \ar[r]^-{+1} &   \\
      \bsfF_\sfB^{p+1}\sfB \ar[r] & \sfB \ar[r] & \bsfF^\sfB_{p}\sfB \ar[r]^-{+1} & .
    }
  \]
  This in turn implies the existence of the dashed arrow which commutes with the rest
  of the diagram and hence $\sigma$ is indeed a co-hyperfiltered morphism.
\end{proof}

One could also define the associated graded complexes of a co-hyperfiltration, but it
turns out that these are the same as the associated graded complexes of the original
hyperfiltration. In fact, using the notation of \autoref{def:hyperfiltrations},
consider the following commutative diagram:
\begin{equation}
  \label{eq:65}
  \begin{aligned}
    \xymatrix@!@R.0em{%
      && \bsfF^{p}\ar[llddd]|!{[dll];[drr]}\hole|!{[dddd];[dll]}\hole
      \ar[drr]  && \\
      \bsfF^{p+1} \ar[rrrr]\ar@[][urr] &&&& \sfA
      \ar@[][dddll] \ar[dd] \\
      &&&& \\
      \bsfG^p \ar[uu]^{+1} \ar@{-->}[drr]
      \ar@{<--}[rrrr]|!{[uu];[drr]}\hole|!{[drr];[uurrrr]}\hole^{+1} &&&& \bsfF_{p-1}
      \ar@{<--}[llll]|!{[dll];[uu]}\hole|!{[uullll];[dll]}\hole
      \ar[lluuu]|!{[uu];[dll]}\hole|!{[uu];[uullll]}\hole^{+1}
      \\
      && \bsfF_{p} \ar@[][uuull]_{+1} \ar@{-->}@[mauve][urr] }.
  \end{aligned}
\end{equation}
Then the octahedral axiom implies that there is a \dt
\[
  \xymatrix{%
    \bsfG^p\ar[r] & \bsfF_{p}\ar[r] & \bsfF_{p-1}\ar[r]^-{+1} & .
  }
\]

\noin
The notion of co-hyperfiltrations allows the contravariant version of
\autoref{lem:image-of-hyperfilt}:

\begin{lem}\label{lem:contravariant-image-of-hyperfilt}
  Let $\Psi:\mcD_1\to\mcD_2$ be a contravariant functor between derived categories
  and let $(\sfA,\bsfF^\kdot)$ be a finite hyperfiltered object in $\Ob\mcD_1$ of
  length $n$ with associated graded complexes $\bsfG^p$. Then $\Psi(\sfA)$ admits a
  co-hyperfiltration given by $\bsfF_p\leteq \Psi(\bsfF^{n-p}))$ in $\Ob\mcD_2$ and
  the $p^\text{th}$-associated graded complex of the co-hyperfiltration $\bsfF_\kdot$
  is isomorphic to $\Psi(\bsfG^{n-p})$
\end{lem}

\begin{proof}
  Applying the contravariant functor $\Psi$ on the hyperfiltration
  $\bsfF^n\to\dots\to\bsfF^0\simeq\sfA$ gives a sequence of morphisms:
  $\Psi(\sfA)\simeq
  \Psi(\bsfF^0)\to\dots\to\Psi(\bsfF^j)\to\dots\to\Psi(\bsfF^n)$. The assignment
  $\bsfF_p\leteq \Psi(\bsfF^{n-p})$ turns this sequence into a co-hyperfiltration
  $\bsfF_\kdot$ as stated.
\end{proof}

\noin Hyperfiltrations also lead to a spectral sequence the same way filtrations do:

\begin{prop}[\protect{\cite[Thm~1.2.2, Appendix]{Kovacs05a}}]\label{prop:E_1-ss}
  Let $\mcA$ and $\mcB$ be abelian categories, $\Phi:\mcA\to\mcB$ a left exact
  additive functor, and $(\sfA,\bsfF)$ a hyperfiltered object in the derived category
  $\mcD(\mcA)$ with associated graded complexes $\bsfG^p$. Then there exists an $E_1$
  spectral sequence,
  \[
    \xymatrix{%
      E_1^{p,q}=\myR^{p+q}\Phi(\bsfG^p) \ar@{=>}[r] &  \myR^{p+q}\Phi(\sfA)
    }
  \]
  abutting to $\myR^{p+q}\Phi(\sfA)$. \hfill\qed
\end{prop}

\noin
And we also have the following:

\begin{prop}\label{prop:coh-of-filt-surj}
  Using the notation and the assumptions of \autoref{prop:E_1-ss}, further let
  $j\in\bN$ and let $F^p\leteq F^p\myR^{j}\Phi(\sfA)$ for $0\leq p\leq j$ denote the
  filtration on $\myR^{j}\Phi(\sfA)$ corresponding to $E_\infty$. In other words, we
  have that $F^p/F^{p+1}\simeq E_\infty^{p,j-p}$.  Assume that the spectral sequence
  established in \autoref{prop:E_1-ss} degenerates at $E_1$.
  Then there exists a natural isomorphism
  $\myR^{j}\Phi(\bsfF^p)\overset\simeq\longrightarrow F^p\myR^{j}\Phi(\sfA)$.
\end{prop}


\begin{proof}
  Recall that, by the defintion of abutting, there exists a \ses,
  \begin{equation}
    \label{eq:22}
    \xymatrix{%
      & 0 \ar[r] & F^{p+1} \ar[r] & F^{p} \ar[r] & E_\infty^{p,j-p}\ar[r] & 0.
    }
  \end{equation}
  By assumption, $E_\infty^{p,j-p}\simeq E_1^{p,j-p}=\myR^j\Phi(\bsfG^p)$.
  For every $p$, there exists a morphism, $\bsfF^p\to \sfA$, coming from the original
  hyperfiltration. This induces a morphism
  \begin{equation}
    \label{eq:18}
    \myR^j\Phi(\bsfF^p) \to \myR^j\Phi(\sfA).
  \end{equation}
  Examining the proof of \cite[Thm~1.2.2, p.28, Appendix]{Kovacs05a}, we see that
  $F^p$ is defined as the image of $\myR^j\Phi(\bsfF^p)$ in $\myR^j\Phi(\sfA)$.
  Actually, we only need that there exists a natural morphism
  \[
    \alpha_p:\myR^j\Phi( \bsfF^p) \to F^p,
  \]
  which follows from the fact that the filtration on $E_\infty$ comes from the
  original hyperfiltration.

  We claim that this natural morphism is an isomorphism. To do this, we will use
  induction on pairs of natural numbers $(j,p)$ with $p\leq j$, ordered
  lexicographically.  For $p=0$ and an arbitrary $j$, $F^p=\myR^j\Phi(\sfA)$ and
  $\bsfF^p=\sfA$. Hence, the statement for $(j,0)$ for any $j\in\bN$ is trivially
  true.

  Assume that we know the statement for pairs $(i,q)\leq (j,p)$ lexicographically,
  i.e., pairs such that either $i<j$ or $i=j$ and $q\leq p$.  We already know that
  the statement holds for $(j+1,0)$, so we only need to prove it for $(j,p+1)$.
  Consider the distinguished triangle 
  \begin{equation}
    \label{eq:1}
    \xymatrix{%
      \bsfF^{p+1} \ar[r] & \bsfF^p \ar[r] & \bsfG^p \ar[r]^-{+1} & , }
  \end{equation}
  and the following induced diagram 
  cf.~\autoref{eq:22}: 
  \[
    \xymatrix@R1.5em@C2em{%
      \dots \ar[r]^-{\sigma_p^{j-1}} & \myR^{j-1}\Phi(\bsfG^p) \ar[r] &
      \myR^{j}\Phi(\bsfF^{p+1}) \ar[d]_-{\alpha_{p+1}} \ar[r]^-{\phi_p} &
      \myR^{j}\Phi(\bsfF^{p}) \ar[r]^-{\sigma_p^j} \ar[d]_-{\alpha_p}^\simeq &
      \myR^{j}\Phi(\bsfG^p) \ar[r] \ar[d]^= & \dots \\
      & 0 \ar[r] & F^{p+1} \ar[r] & F^{p} \ar[r] & \myR^{j}\Phi(\bsfG^p)\ar[r] & 0 }
  \]
  By the construction of the vertical arrows, this is a commutative diagram.  The
  vertical arrow on the right is the identity, and the vertical arrow in the middle
  is an isomorphism by the inductive assumption. This implies that $\sigma_p^j$ is
  surjective. The same argument applied for the pair $(j-1,p)$, using the inductive
  assumption, shows that $\sigma_p^{j-1}$ is surjective, which implies that $\phi_p$
  is injective. Therefore the above commutative diagram is actually a diagram of
  \sess,
  and hence $\alpha_{p+1}$ is also an isomorphism by the $5$-lemma. This
  proves the desired statement.
\end{proof}

\noin In the proof above 
we encountered the following corollary, which is worth stating on its own.

\begin{cor}\label{cor:ss-surjectivity}
  Using the notation and the assumptions of \autoref{prop:E_1-ss}, further assume
  that the spectral sequence established in \autoref{prop:E_1-ss} degenerates at
  $E_1$.
  Then the natural morphism\linebreak
  $\sigma_p^j:\myR^{j}\Phi(\bsfF^{p})\onto\myR^{j}\Phi(\bsfG^p)$ is surjective.
\end{cor}


\subsection{Filtered connections}
\noindent

\begin{defini}
  Let $X$ be a complex scheme, $\sE$ a locally free $\sO_X$-module and $(\sfA,\bsfF)$
  a hyperfiltered object in $D_{\rm filt, coh}(X)$ with associated graded complexes
  $\bsfG^p_{\bsfF}(\sfA)$.  Then a \emph{hyperfiltered complex of connections on
    $\sfA$ with respect to $\sE$}, or simply a \emph{complex of filtered
    connections}, is a hyperfiltered object in $D_{\rm filt, coh}(X)$, denoted by
  $\sfA\otimes\sE$, with hyperfiltration $\bsfF\otimes\sE$ and associated graded
  complexes $\bsfG^p_{\bsfF\otimes\sE}(\sfA\otimes\sE)$ such that
  $\bsfG^p_{\bsfF\otimes\sE}(\sfA\otimes\sE)\simeq \bsfG^p_{\bsfF}(\sfA)\otimes\sE$.
  If such a complex exists, then we say that $\sfA$ \emph{admits a hyperfiltered
    complex of connections with respect to $\sE$}.  If $\phi:\sfA\to\sfB$ is a
  hyperfiltered morphism in $D_{\rm filt, coh}(X)$, and there exists a hyperfiltered
  morphism between $\sfA\otimes\sE$ and $\sfB\otimes\sE$ such that the induced map on
  $\bsfG^p(\sfA\otimes\sE)\to\bsfG^p(\sfB\otimes\sE)$ agrees with the induced map on
  $\bsfG^p(\sfA)\to\bsfG^p(\sfB)$ twisted by $\sE$, then we say that $\sfA$ and
  $\sfB$ \emph{admit compatible hyperfiltered complexes of connections with respect
    to $\sE$}.
\end{defini}

\begin{lem}\label{lem:connections-for-dts}
  Let $X$ be a complex scheme, $\sE$ a locally free $\sO_X$-module and
  $\xymatrix{\sfA\ar[r] & \sfB\ar[r] & \sfC\ar[r]^-{+1} & }$ a distinguished triangle
  of hyperfiltered objects in $D_{\rm filt, coh}(X)$. If two of these objects admit
  compatible filtered complexes of integrable connections with respect to $\sE$, then
  so does the third.
\end{lem}

\begin{proof}
  This is straightforward from the definition and basic properties of \dts.
\end{proof}


%


%

\numberwithin{equation}{subsection}
\section{Complexes of differential forms}\label{sec:compl-diff-forms}
\noindent
We will be studying several complexes of differential forms and their
interactions.
Throughout this section $X$ will denote a scheme that is essentially of finite type
over
$\bC$.
We use the usual convention that if $\sF$ is a sheaf on $X$, then it is considered a
complex with $\sF$ in degree $0$ and $0$ everywhere else. Then this complex maybe
shifted, so for instance $\sF[j]$ means a complex with $\sF$ in degree $-j$ and $0$
everywhere else.
In the rest of this section, unless otherwise stated, $p\in\bZ$ will denote an
arbitrary integer. We will see that we could restrict to natural numbers, but it will
be convenient later to allow $p$ to take negative values. This will not cause any
issues or a need to change anything.

\subsection{K\"ahler differentials; the de~Rham complex}

The sheaf of 
differentials 
and its exterior powers, $\Omega_X^p=\!\bigwedge^p\Omega_X$, give rise to the
\emph{de Rham complex} of $X$:
\[
  \Omega_X^\kdot= \dots\to 0\to \sO_X\to \Omega_X\to \dots \to \Omega_X^p \to \dots,
\]
where $\Omega_X^p$ sits at the degree $p$ position in the complex $\Omega_X^\kdot$.

\subsection{The filtration and co-filtration complexes of the de~Rham complex}

$\Omega_X^\kdot$ is a \emph{filtered complex}, with its ``filtration b\^ete'',
denoted by
\[
  \f^p_X\leteq F^p\Omega_X^\kdot \leteq \dots\to 0\to \dots\to 0\to \Omega_X^{p}\to
  \Omega_X^{p+1} \to \dots.
\]
We will call $\f_X^p$ the \emph{$p^\text{th}$-de~Rham filtration complex} of
$X$.
Note that $\f_X^p=0$ if $p>\dim X$ and $\f_X^p=\Omega_X^\kdot$ if $p\leq 0$.
Recall that $\f_X^{p+1}\subseteq \f_X^p$ is a subcomplex and the quotient is the
complex $\Omega_X^p[-p]$. In other words, we have a \ses of complexes:
\begin{equation}
  \label{eq:5}
  \xymatrix{%
    0\ar[r] & \f_X^{p+1} \ar[r] & \f_X^p \ar[r] &  \Omega_X^p[-p] \ar[r] & 0.
  }
\end{equation}
Next, apply \autoref{lem:filtr-co-filtr} to $(\Omega_X^\kdot, F)$, i.e., 
denote the quotient $\Omega_X^\kdot/\f_X^{p+1}$ by $\f^X_p$. It follows that
\[
  \f^X_p\simeq \dots\to 0\to \sO_X\to \Omega_X\to\dots\to \Omega_X^{p}\to 0 \to
  \dots,
\]
i.e., it consists of the first $p+1$ terms of $\Omega_X^\kdot$ (starting at degree
$0$). We will call $\f^X_p$ the \emph{$p^\text{th}$-de~Rham co-filtration complex} of
$X$.  Note 
that $\f^X_p=\Omega_X^\kdot$ if $p\geq\dim X$ and $\f^X_p=0$ if $p< 0$.

As in \autoref{lem:filtr-co-filtr}, we may encode the defining relationship in the
\ses,
\begin{equation}
  \label{eq:6}
  \xymatrix{%
    0\ar[r] & \f_X^{p+1} \ar[r] & \Omega_X^\kdot \ar[r] & \f^X_p \ar[r] & 0.  }
\end{equation}
Contrary to a filtration, by design, $\f^X_{p-1}$ is a quotient complex of $\f^X_{p}$
with kernel isomorphic to the complex $\Omega_X^p[-p]$. In other words, we also have
a \ses with the $\f^X_p$'s:
\begin{equation}
  \label{eq:7}
  \xymatrix{%
    0\ar[r] & \Omega_X^p[-p] \ar[r] & \f^X_{p} \ar[r] & \f^X_{p-1} \ar[r] \ar[r] & 0.  }
\end{equation}

\subsection{The de~Rham complex of a pair}
Let 
$\imath:\Sigma\into X$ denote a closed subscheme. Then there exists a natural
filtered morphism $\Omega_X^\kdot\to\Omega_\Sigma^\kdot$ and we define the de~Rham
complex of $(X,\Sigma)$, denoted by $\Omega_{X,\Sigma}^\kdot$, as the $(-1)$-shifted
mapping cone of this morphism, i.e., such that it fits into the \dt:
\[
  \xymatrix{%
    \Omega_{X,\Sigma}^\kdot \ar[r] & \Omega_{X}^\kdot \ar[r] &
    \Omega_{\Sigma}^\kdot\ar[r]^-{+1} & }
\]
\begin{example}\label{ex:snc-pair-dR}
  Let $(X,\Sigma)$ be an snc pair. Then
  $\Omega^\kdot_{X,\Sigma} \qis \Omega_X^\kdot(\log\Sigma)(-\Sigma)$.
\end{example}

We define the filtration and co-filtration complexes of $\Omega^\kdot_{X,\Sigma}$ the
same way as in the case of the de~Rham complex and denote them by $\f^p_{X,\Sigma}$
and $\f_p^{X,\Sigma}$.

\subsection{The Deligne-\DB complex}

Unfortunately, the nice properties of the de~Rham complex are not preserved if the
underlying space is singular.
The \DDB complex is a generalization of the de~Rham complex of a (complex) manifold
to arbitrary complex varieties.  This was first introduced in \cite{DuBois81}
following Deligne's ideas \cite{MR0441965,MR0498551,MR0498552}. For more details the
reader should consult \cite{Steenbrink85,GNPP88,MR2796408,SingBook}. Here we follow
the notation and terminology of \cite[\S6]{SingBook}.
In particular, we follow the usual convention of denoting the \DDB complex by
$\Om_X^\kdot$ %
and the shifted graded pieces of it by $\Om_X^p$. More precisely, recall that
$\Om_X^\kdot$ is defined as an object in the \emph{filtered} derived category of
quasi-coherent sheaves on $X$ with $\bC$-linear differentials that are differential
operators of order at most $1$ \cite[p.43]{DuBois81}, and there exists a natural
filtered morphism from the de~Rham complex to the \DDB complex.
\begin{equation}
  \label{eq:8}
  \Omega_X^\kdot \longrightarrow \Om_X^\kdot.
\end{equation}
We will follow the philosophy that led to the definition of hyperfiltrations
\autoref{def:hyperfiltrations} and treat filtrations as hyperfiltrations to avoid the
cumbersome technicalities stemming from defining filtrations in derived categories
and taking advantage of considering the members of the filtration as legitimate
objects in the derived category. Recall that each member of a filtration is itself a
filtered object given by the part of the filtration that comes ``after'' that member.

The construction of the \DDB complex is quite involved and the interested reader
should consult \cite{DuBois81,MR2393625,MR2796408}.  Nonetheless, let us recall that
if\/ $\varepsilon_\kdot\col X_\kdot\to X$ is a hyperresolution, then
\begin{equation}
  \label{eq:38}
  \Om_X^\kdot \filtqis \myR{\varepsilon_\kdot}_* \Omega^\kdot_{X_\kdot}.
\end{equation}
A variant of this will be useful later.
We will use the definitions of cubical varieties and related notions from
\cite[Chapter~5]{MR2393625}.
\begin{defini}
  Let $X$ be a scheme of finite type over $\bC$.  A \emph{cubical partial
    hyperresolution} of $X$ is a cubical variety $\varepsilon_\kdot\col X_\kdot\to X$
  that has all the properties of a \emph{cubical hyperresolution}
  \cite[Def.~5.10]{MR2393625}, except that the individual varieties $X_\alpha$ are
  not assumed to be non-singular.  More precisely, it is a
  \emph{$\square^+_r$-scheme} over $X$ for some $r\in\bN$ as in
  \cite[I.2.12]{GNPP88}, cf.~\cite[2.12]{MR2796408}, of \emph{cohomological descent}
  \cite[\S5.3]{MR0498552},
  \cite[Defs.~5.6,5.10]{MR2393625}.  This is called a
  \emph{polyhedral resolution} (as opposed to a \emph{smooth polyhedral resolution})
  in \cite[p.~596]{Carlson85}.
\end{defini}

\begin{example}
  Let $\Sigma\into X$ be a subscheme and $\pi_\kdot\col X_\kdot \to X$ a cubical
  hyperresolution. Let $\Sigma_\kdot\leteq X_\kdot\times_X\Sigma$. Then
  $\Sigma_\kdot\to\Sigma$ is a cubical partial hyperresolution.
\end{example}

\begin{lem}[\protect{\cite[V.3.6(5)]{GNPP88}}]\label{lem:Rnu-of-DDB-is-DDB}%
  Let $\varepsilon_\kdot\col X_\kdot\to X$ be a cubical partial hyperresolution. Then
  \[
    \Om_X^\kdot \filtqis \myR{\varepsilon_\kdot}_* \Om^\kdot_{X_\kdot}.
  \]
\end{lem}




\subsection{The filtration and co-filtration complexes of the \DDB complex}

As mentioned already, just as in the case of the de~Rham complex, $\Om_X^\kdot$ is a
filtered complex. The image of the members of its filtration in the filtered derived
category will be denoted by
\[
  \uf_X^p:= F^p\Om_X^\kdot,
\]
and called the \emph{$p^\text{th}$-\DDB filtration complex} of $X$.
As in the case of the de~Rham complex, we have that $\uf_X^p=0$ if $p>\dim X$ and
$\uf_X^p=\Om_X^\kdot$ if $p\leq 0$.
While this is an actual filtration,
it is sometimes more convenient to think of it as a \emph{hyperfiltration}
cf.~\S\S\ref{subsec:hyperf-co-hyperf}, \cite[\S\S1.2]{Kovacs05a}, \cite[\S2.2]{KT21}.
This will not change anything in the sequel.  By definition, the associated graded
complexes of this (hyper)filtration are the \DB analogues of the sheaves
$\Omega_X^p$. More precisely, we have the distinguished triangles,
\begin{equation}
  \label{eq:9}
  \xymatrix{%
    \uf_X^{p+1} \ar[r] & \uf_X^p \ar[r] &  \Om_X^p[-p] \ar[r]^-{+1} & ,
  }
\end{equation}
and with a slight abuse of language, $\Om_X^p$ will be referred to as the
\emph{$p^\text{th}$-associated graded \DDB complex}.
In addition, similarly to the de~Rham case, we define the \emph{$p^\text{th}$-\DDB
  co-filtration complex} of $X$ as the $p^\text{th}$ term of the associated
co-filtration of $\f_X^p$, i.e., the cone of the morphism
$\uf_X^{p+1}\to\Om_X^\kdot$, i.e., we have the distinguished triangle,
\begin{equation}
  \label{eq:10}
  \xymatrix{%
    \uf_X^{p+1} \ar[r] & \Om_X^\kdot \ar[r] & \uf^X_p \ar[r]^-{+1} & .  }
\end{equation}
And, again, we have that $\uf^X_p=\Om_X^\kdot$ if $p\geq\dim X$ and $\uf^X_p=0$ if
$p<0$.
We also have the following \dt, which is an analogue of \autoref{eq:7}:
\begin{equation}
  \label{eq:11}
  \xymatrix{%
    \Om_X^p[-p] \ar[r] & \uf^X_{p} \ar[r] & \uf^X_{p-1} \ar[r]^-{+1} & .  }
\end{equation}
This can be seen either by writing down the definition of the filtration, or
directly, by using the octahedral axiom, as in \autoref{eq:65}. 
\vskip-1.75em
\[
  \xymatrix@!C@R1.5em{%
    && \uf_X^{p}\ar[llddd]|!{[dll];[drr]}\hole|!{[dddd];[dll]}\hole
    \ar[drr]  && \\
    \uf_X^{p+1} \ar[rrrr]\ar@[][urr] &&&& \Om_X^\kdot
    \ar@[][dddll] \ar[dd] \\
    &&&& \\
    \Om_X^p[-p] \ar[uu]^{+1} \ar@{-->}[drr]
    \ar@{<--}[rrrr]|!{[uu];[drr]}\hole|!{[drr];[uurrrr]}\hole^{+1} &&&& \uf^X_{p-1}
    \ar@{<--}[llll]|!{[dll];[uu]}\hole|!{[uullll];[dll]}\hole
    \ar[lluuu]|!{[uu];[dll]}\hole|!{[uu];[uullll]}\hole^{+1}
    \\
    && \uf^X_{p} \ar@[][uuull]_{+1} \ar@{-->}@[mauve][urr] }
\]
\subsection{The \hz-complex of the \DDB complex}\label{subsec:hz-complex-ddb}
It will be useful later to have a notation for the $0^\text{th}$ cohomology sheaves
of the associated graded \DDB complexes. We will use the following notation:
\[
  \wt\Omega_X^p\leteq h^0(\Om_X^p).
\]

\begin{lem}\label{lem:hz-is-tf}
  Let $X$ be a variety. Then $\wt\Omega_X^p$ is torsion-free for all $p$.
\end{lem}

\begin{proof}
  Let $\pi_\kdot:X_\kdot\to X$ be a hyperresolution of $X$ such that there is an
  object, $X_0$ in $X_\kdot$ that the induced morphism, $\pi_0:X_0\to X$ is a
  resolution of \sings of $X$. This is not necessarily true for all hyperresolutions,
  but we may choose one with this property. Then by the construction of the \DDB
  complex, $\wt\Omega_X^p\subseteq \pi_{0*}\Omega_{X_0}^p$, and the latter sheaf is
  torsion-free. 
\end{proof}

\begin{lem}\label{cor:ltensor-of-twidle}
  Let $X$ be a variety and $H$ a Cartier divisor on $X$.  Then for each
  $p\in\bN$ there is a \ses,
  $\xymatrix{
      0 \ar[r] & \wt\Omega_X^p\otimes\sO_X(-H)\ar[r] & \wt\Omega_X^p \ar[r] &
      \wt\Omega_X^p\otimes\sO_H \ar[r] & 0, \text{ and hence}}$
    $\wt\Omega_X^p\lotimes\sO_H\simeq\wt\Omega_X^p\otimes\sO_H.$
\end{lem}
\begin{proof}
  This follows directly from \autoref{lem:torsion-free-tensor} and
  \autoref{lem:ltimes-is-times}.
\end{proof}

\begin{cor}\label{cor:h0-of-lotimes}
  Let $X$ be a variety and $H$ a general member of a basepoint-free linear
  system.  Then for each $p\in\bN$,
  $h^0(\Om^p_X\lotimes\sO_H)\simeq \wt\Omega^p_X\otimes\sO_H$.
\end{cor}

\begin{proof}
  Apply $\Om_X^p\lotimes\blank$ to \autoref{eq:75} and consider the associated long
  exact cohomology sequence. Observe that
  $h^1(\Om_X^p\lotimes\sO_X(-H))\simeq h^1(\Om_X^p)\otimes\sO_X(-H)\to h^1(\Om_X^p)$ 
  is injective by \autoref{cor:prime-avoidance--injectivity} and hence the statement
  follows from \autoref{cor:ltensor-of-twidle}
\end{proof}

\noin Consider the following diagram, where the rows are distinguished triangles
cf.~\autoref{eq:9}:
\[
  \xymatrix@R1.25em{%
    \uf_X^{p+2} \ar[r] & \uf_X^{p+1} \ar[r] \ar@{=}[d] & \Om_X^{p+1}[-p-1]
    \ar[r]^-{+1} &
    \\
    & \uf_X^{p+1} \ar[r] & \uf_X^p \ar[r] & \Om_X^p[-p] \ar[r]^-{+1} & .  }
\]
Taking the long exact cohomology of both distinguished triangles we obtain natural
morphisms
\[
  \xymatrix{%
    \wt\Omega_X^p= h^p\left(\Om_X^p[-p]\right)\ar[r] &
    h^{p+1}\left(\uf_X^{p+1}\right) \ar[r] & h^{p+1}\left(\Om_X^{p+1}[-p-1]\right)
    =\wt\Omega_X^{p+1}.  }
\]
Hence there exists a natural morphism $\wt\Omega_X^p\to
\wt\Omega_X^{p+1}$
and we get the following factorization:
\[
  \xymatrix{%
    \wt\Omega_X^p\ar[r] & h^{p+1}\left(\uf_X^{p+1}\right) \ar[r] \ar@/_1em/[rr]_0 &
    \wt\Omega_X^{p+1}
    \ar[r] & h^{p+2}\left(\uf_X^{p+2}\right) \ar[r] & \wt\Omega_X^{p+2}.  }
\]
Observe that the composition of the two morphisms in the middle is $0$, because they
are consecutive morphisms in the long exact cohomology sequence of the first
distinguished triangle in the previous diagram. This shows that the composition
\[
  \xymatrix{%
    \wt\Omega_X^p\ar[r] & \wt\Omega_X^{p+1} \ar[r] & \wt\Omega_X^{p+2}.  }
\]
is the zero morphism. In other words, the sheaves $\wt\Omega_X^p$ form a complex,
which will be denoted by $\wt\Omega_X^\kdot$ and called the \emph{\hz-complex of the
  \DDB complex} of $X$.

\subsection{The filtration and co-filtration complexes of the \hz-complex}

As before, $\wt\Omega_X^\kdot$ is also a filtered complex with its own ``filtration
b\^ete'', denoted by
\[
  \tf^p_X\leteq F^p\wt\Omega_X^\kdot \leteq \dots\to 0\to \dots\to \wt\Omega_X^p\to
  \wt \Omega_X^{p+1} \to \dots.
\]
We will call $\tf_X^p$ the \emph{$p^\text{th}$-$h^0$-filtration complex} of $X$.
And, of course, $\tf_X^p=0$ if $p>\dim X$ and $\tf_X^p=\wt\Omega_X^\kdot$ if
$p\leq 0$.
Furthermore, $\tf_X^{p+1}\subseteq \tf_X^p$ is a subcomplex and the quotient is the
complex $\wt\Omega_X^p[-p]$. In other words, we have a distinguished triangle:
\begin{equation}
  \label{eq:5tf}
  \xymatrix{%
    \tf_X^{p+1} \ar[r] & \tf_X^p \ar[r] &  \wt\Omega_X^p[-p] \ar[r]^-{+1} & .
  }
\end{equation}
As $\tf_X^{p}\subseteq \wt\Omega_X^\kdot$ is again a subcomplex, we will denote the
quotient by $\tf^X_p$. It is straightforward that
\[
  \tf^X_p\simeq \dots\to 0\to \wt\Omega_X^0\to \wt\Omega_X^1 \dots\to
  \wt\Omega_X^{p}\to 0 \to \dots,
\]
i.e., it consists of the first $p+1$ terms of $\wt\Omega_X^\kdot$ (starting at degree
$0$). We will call $\tf^X_p$ the \emph{$p^\text{th}$-$h^0$-co-filtration complex} of
$X$.  Again, $\tf^X_p=\wt\Omega_X^\kdot$ if $p\geq\dim X$ and $\tf^X_p=0$ if $p< 0$,
and we may encode the defining relationship in the distinguished triangle:
\begin{equation}
  \label{eq:6tf}
  \xymatrix{%
    \tf_X^{p+1} \ar[r] & \wt\Omega_X^\kdot \ar[r] & \tf^X_p \ar[r]^-{+1} & .  }
\end{equation}
Finally, $\tf^X_{p-1}$ is a quotient complex of $\tf^X_p$ with kernel isomorphic to
the complex $\wt\Omega_X^p[-p]$. In other words, we have a distinguished triangle
with the $\tf^X_p$'s:
\vskip-1.5em
\begin{equation}
  \label{eq:7tf}
  \xymatrix{%
    \wt\Omega_X^p[-p] \ar[r] & \tf^X_{p} \ar[r] & \tf^X_{p-1} \ar[r] \ar[r]^-{+1} & .  }
\end{equation}

\subsection{The filtered \DDB complex of a pair}\label{sec:ddb-complex-pair}
Let $X$ be a complex scheme of finite type and let $\imath:\Sigma\into X$ denote a
closed subscheme. %
The \DDB-complex of the pair $(X,\Sigma)$ is defined by the following \dt
(cf.~\cite{Steenbrink85},\cite[3.9]{Kovacs10a}, \cite[6.4,6.5]{SingBook}):
\vskip-1.5em
\begin{equation}
  \label{eq:14}
  \xymatrix{%
    \Om_{X,\Sigma}^\kdot \ar[r] & \Om_X^\kdot \ar[r] & \Om_\Sigma^\kdot \ar[r]^-{+1} & 
  }%
\end{equation}
\begin{example}\cite[3.11]{Kovacs10a}\label{ex:snc-pair}
  Let $(X,\Sigma)$ be an snc pair. Then
  $\Om^\kdot_{X,\Sigma}\qis\Omega^\kdot_{X,\Sigma}\qis
  \Omega_X^\kdot(\log\Sigma)(-\Sigma)$.
\end{example}

\noin%
Assume that $\Sing X\subseteq \Sigma$, and let $\pi:Y\to X$ be a strong log
resolution of 
$(X,\Sigma)$. Let $\Gamma\leteq\pi^{-1}(\Sigma)$ with the
reduced induced scheme structure. Then there exists a commutative diagram
(cf.~\cite[3.11]{Kovacs10a}),
\[
  \xymatrix{%
    \Om_{X,\Sigma}^\kdot \ar[r]\ar[d] & \Om_X^\kdot \ar[r]\ar[d] & \Om_\Sigma^\kdot \ar[d]
    \ar[r]^-{+1} & \\
    \myR\pi_* \Om_{Y,\Gamma}^\kdot \ar[r] & \myR\pi_*\Om_Y^\kdot \ar[r] &
    \myR\pi_*\Om_\Gamma^\kdot \ar[r]^-{+1} & .}
\]
\cite[4.11]{DuBois81} and \cite[(2.1.4)]{KK10} imply that the morphism
$\Om_{X,\Sigma}^\kdot\to\myR\pi_*\Om_{Y,\Gamma}^\kdot$ is an isomorphism, and then
by \autoref{ex:snc-pair}
they are also isomorphic to
$\myR\pi_*\Omega_Y^\kdot(\log\Gamma)(-\Gamma)$.
In other words, we have following:
\begin{lem}
  \label{lem:DB-to-Irr}
  Let $X$ be a reduced scheme and $\Sigma\subseteq X$ a reduced subscheme such that
  $\Sing X\subseteq \Sigma$, and let $\pi:Y\to X$ be a strong log resolution of the
  pair $(X,\Sigma)$. Let $\Gamma\leteq\pi^{-1}(\Sigma)$ with the reduced induced
  scheme structure.  Then there exists a distinguished triangle,
  \[
    \xymatrix{%
      \myR\pi_*\Omega_Y^\kdot(\log\Gamma)(-\Gamma)\ar[r] & \Om_X^\kdot \ar[r] &
      \Om_\Sigma^\kdot \ar[r]^-{+1} & .}
  \]
\end{lem}

We define the filtration and co-filtration complexes for the \DDB complex of a pair
as before and denote them by $\uf_{X,\Sigma}^p$ and $\uf_p^{X,\Sigma}$. They satisfy
the same relations with respect to $\Om_{X,\Sigma}^\kdot$ as in the case of
$\Sigma=\emptyset$. We also define the $\hz$-complex of a pair and its filtration and
co-filtration complexes analogously to the $\Sigma=\emptyset$ case, which will be
denoted by $\tf_{X,\Sigma}^p$ and $\tf_p^{X,\Sigma}$. We have the same relations
between these filtration complexes as in the $\Sigma=\emptyset$ case.

Recall that Deligne's Hodge theory in this situation gives the following theorem:

\begin{thm}\cite{MR0498552}%
  \label{thm:hodge} 
  Let $X$ be a complex scheme of finite type, $\imath:\Sigma\into X$ a closed
  subscheme and $\jmath:U\leteq X\setminus \Sigma\into X$.  Then
  \begin{enumerate}
  \item\label{item:14} The natural composition map
    $\jmath_{!}\bC_{U}\to \sI_{\Sigma\subseteq X} \to \Om^\kdot_{X,\Sigma}$ is a
    quasi-isomorphism, i.e., $\Om^\kdot_{X,\Sigma}$ is a resolution of the sheaf
    $\jmath_{!}\bC_{U}$.
  \item\label{item:15} The natural map
    $H_{\rm c}^\kdot(U,\bC)\to \bH^\kdot(X, \Om^\kdot_{X,\Sigma})$ is an isomorphism.
  \item\label{item:9} If in addition $X$ is proper, then the spectral sequence,
    \[
      E_1^{p,q}= \bH^q(X, \Om^p_{X,\Sigma}) \Rightarrow H_{\rm c}^{p+q}(U,\bC)
    \]
    degenerates at $E_1$ and abuts to the Hodge filtration of Deligne's mixed Hodge
    structure. 
  \end{enumerate}
\end{thm}

\begin{proof}
  Consider an embedded hyperresolution of $\Sigma\subseteq X$:
  $ \xymatrix{%
    \text{\phantom{$\kdot$}}\Sigma_\kdot \hskip-.5ex\ar[r]^-{\varrho_\kdot}
    \ar[d]_{\varepsilon_\kdot\hskip-.5ex} & X_\kdot \ar[d]^{\varepsilon_\kdot}
    \hskip-1ex \\
    \Sigma \ar[r]_\varrho & X }$
    
  \noin
  Then by \autoref{eq:8} and by definition
  $\Om_{X,\Sigma}^\kdot \qis \myR{\varepsilon_\kdot}_* \Omega^\kdot_{X_\kdot,
    \Sigma_\kdot}$.  The statements then follow from \cite[8.1, 8.2,
  9.3]{MR0498552}. See also \cite[IV.4]{GNPP88}.
\end{proof}

\subsection{The filtered log \DDB complex}

Let $X$ be a complex scheme of finite type and let $\imath:\Sigma\into X$ denote a
closed subscheme. %
Let $\pi_\kdot:X_\kdot\to X$ be a hyperresolution of $X$ such that for each $\alpha$,
$\Sigma_\alpha\leteq\pi_\alpha^*\Sigma\subseteq X_\alpha$ is either empty, equal to
$X_\alpha$, or is an snc divisor on $X_\alpha$.
The log \DDB complex of the pair $(X,\Sigma)$, $\Om^\kdot_X(\log \Sigma)$, as in
\cite[\S6]{DuBois81}, is defined as
$\myR(\pi_\kdot)_*\Omega_{X_\kdot}(\log\Sigma_\kdot)$ using a hyperresolution as the
above one and where $\Omega_{X_\alpha}(\log\Sigma_\alpha)\leteq 0$ for each $\alpha$
for which $\Sigma_\alpha=X_\alpha$.
In addition, we define the log filtration and co-filtration complexes and denote them
by $\uf_X^p(\log \Sigma)$ and $\uf_p^X(\log \Sigma)$. They satisfy the same relations
with respect to $\Om_X^\kdot(\log \Sigma)$ as in the case of $\Sigma=\emptyset$. We
also define the log $\hz$-complex and its log filtration and co-filtration complexes
analogously to the $\Sigma=\emptyset$ case, which will be denoted by
$\tf_X^p(\log \Sigma)$ and $\tf_p^X(\log \Sigma)$. Note that essentially identical
arguments imply that the results of
\autoref{lem:hz-is-tf},~\autoref{cor:ltensor-of-twidle}, and
\autoref{cor:h0-of-lotimes} remain true if one replaces $\wt\Omega_X^p$ with
$\wt\Omega_X^p(\log\Sigma)$.

\begin{notation}
  If $\Sigma$ is the union of two closed subsets, $\Sigma=\Sigma_1\cup\Sigma_2$, then
  instead of $\Om^\kdot_X\left(\log (\Sigma_1\cup\Sigma_2)\right)$, we will also
  write $\Om^\kdot_X\left(\log (\Sigma_1+\Sigma_2)\right)$.
\end{notation}

\subsection{The associated analytic space}

For $X$, a scheme as above, $X^{\an}$ will denote the \emph{associated complex
  analytic space} of $X$ \cite{MR0082175}, \cite[App.~B]{Hartshorne77}. We define the
above complexes with their filtrations and co-filtrations for $X^{\an}$ in place of
$X$.

\subsection{The irrationality complexes}

Let $X$ be a variety of pure dimension $n$ such that its irreducible components are
disjoint. (This happens for instance if $X$ is normal). We define the
\emph{$p^\text{th}$-irrationality complex} of $X$ by
\begin{equation}
  \label{eq:13}
  \Irr_X^p\leteq \bD_X(\Om_X^{n-p}) 
\end{equation}
We also define the \hz-complex of $\Irr^\kdot_X$ by
$\wt\Irr^p_X\leteq h^0(\Irr^p_X)$.

Let $\Sigma\leteq \Sing X\subsetneq X$ and consider a strong log resolution
$\pi:Y\to X$ of the pair $(X,\Sigma)$. Further let
$\Gamma\leteq \pi^{-1}\left(\Sigma\right)_\red$ which is an snc divisor by
assumption.  Then
\begin{equation}
  \label{eq:61}
  \wt\Irr^p_X\simeq\pi_*\Omega_Y^p(\log\Gamma)
\end{equation}
by \cite[Lemma~2.4(2)]{SVV23} and if $p<\codim_X\Sigma$, then
\begin{equation}
  \label{eq:62}
  \Irr^p_X\simeq\myR\pi_*\Omega_Y^p(\log\Gamma)
\end{equation}
by \cite[Lemma~3.14]{FL-lci}.
As a corollary to \autoref{eq:61} we obtain the following:
\begin{cor}\label{cor:tw-irr-p-is-tf}
  Let $X$ be a normal variety. Then $\wt\Irr^p_X$ is torsion-free for all $p$.
\end{cor}

\begin{rem}
  It is possible to define a hyperfiltered complex $\Irr_X^\kdot$
  whose associated graded complexes are the $\Irr_X^p$, but in this article there is
  no need for them, so this is left to the interested reader.
\end{rem}

\subsection{Relations among the complexes of differential forms}

By the definition of the \DDB complex there exists a natural filtered morphism
$\Omega_X^\kdot \to \Om_X^\kdot$.  Being a filtered morphism, this induces natural
morphisms $\Omega_X^p \to \Om_X^p$ for each $p$.  The latter is still a morphism of
complexes, so it maps $\Omega_X^p$ to $h^{0}\left(\Om_X^{p}\right)$.
This morphism is compatible with the differentials of the complexes and hence we
obtain that the above natural filtered morphism factors through $\wt\Omega_X^\kdot$:
\begin{equation}
  \label{eq:12}
  \Omega_X^\kdot \to  \wt\Omega_X^\kdot \to \Om_X^\kdot.
\end{equation}
Note that these are filtered morphisms and hence co-filtered morphisms as well, and
so they induce similar morphisms on the respective filtration and co-filtration
complexes as well.


\begin{lem}\label{lem:Omega-is-S2}
  Assume that $X$ is normal and (for some $p$) $\Omega_X^p$ is reflexive. Then the
  natural morphism, $\eta:\Omega_X^p\to\wt\Omega_X^p$, induced by \autoref{eq:12}, is
  an isomorphism.
\end{lem}

\begin{proof}
  As $X$ is normal, we may assume that it is irreducible.  Recall that
  $\wt\Omega_X^p$ is torsion-free by \autoref{lem:hz-is-tf}, so the natural morphism
  $\nu: \wt\Omega_X^p\to \big(\wt\Omega_X^p\big)^{**}$ is injective. Similarly, as
  $\Omega_X^p$ is reflexive, it is also torsion-free and hence
  $\eta:\Omega_X^p\to\wt\Omega_X^p$ is also injective. However, as $X$ is normal,
  both $\eta$ and $\nu$ are isomorphisms in codimension $1$, and then it follows that
  $\nu\circ\eta$ is an injective morphism between two reflexive sheaves which is an
  isomorphism in codimension $1$, so it has to be an isomorphism.  Then $\eta$ and
  $\nu$, both being injective, have to be isomorphisms as well.
\end{proof}

\begin{prop}
  \label{prop:morphs-of-cxs}%
  Let $X$ be a scheme essentially of finite type over $\bC$, of dimension
  $n\leteq\dim X$ and $\pi:Y\to X$ a strong log resolution of \sings with
  $E\leteq\exc(\pi)$.  Then for every $p\in\bZ$ there exist natural morphisms as
  follow:
  \[
    \xymatrix{%
      \Omega_X^p\ar[r] & \wt\Omega_X^p\ar[r] & \Om_X^p\ar[r] &
      \myR\pi_*\Omega_Y^p\ar[r] & \bD_X(\Om_X^{n-p})\simeq \Irr_X^p \ar[r] &
      \myR\pi_*\Omega_Y^p(\log E).  }
  \]
  Furthermore, the isomorphism class of the object $\myR\pi_*\Omega_Y^p$ in $D(X)$
  and the composite morphism $\Om_X^p\to$ $\bD_X(\Om_X^{n-p})\simeq \Irr_X^p$ is
  independent of the choice of $\pi$. Similarly, the composition
  $\myR\pi_*\Omega_Y^p\to\myR\pi_*\Omega_Y^p(\log E)$ agrees with the usual natural
  morphism.  Note that it also follows trivially (by taking $h^0\!$ of both) that the
  induced morphism $\wt\Omega_X^p\to \Irr^p_X$ factors through $\wt\Irr^p_X$.
\end{prop}

\begin{rem}
  The existence and independece from $\pi$ of the natural morphism
  $\Om_X^p\to \bD_X(\Om_X^{n-p})$ was already observed in \cite[Lemma~3.11]{FL-lci}.
\end{rem}

\begin{proof}
  The existence of the first two morphisms follows from \autoref{eq:12}. The third
  morphism follows from the functoriality of the \DDB complex and the fact that
  $\Om_Y^p\simeq \Omega_Y^p$.
  
  Note that $\Omega_Y^p\simeq\bD_Y(\Omega_Y^{n-p})$ and by Grothendieck duality,
  $\myR\pi_*\bD_Y(\Omega_Y^{n-p})\simeq \bD_X(\myR\pi_*\Omega_Y^{n-p})$.  Applying
  $\bD_X$ to the morphism $\Om_X^{n-p}\to\myR\pi_*\Omega_Y^{n-p}$ yields
  $\myR\pi_*\Omega_Y^p \simeq \bD_X(\myR\pi_*\Omega_Y^{n-p})\to \bD_X(\Om_X^{n-p})$,
  which is the desired fourth morphism in the diagram. The independence of the
  induced morphism $\Om_X^p\to \bD_X(\Om_X^{n-p})$ follows from the proof of
  \cite[Lemma~1.6]{MP22}. More precisely, the first paragraph of that proof starts by
  stating a slightly more general statetement, which implies that
  $\myR\pi_*\Omega_Y^p$ is independent of the choice of $\pi$. It is easy to see from
  the construction that then so is the composition morphism
  $\Om_X^p\to \bD_X(\Om_X^{n-p})$.
  
  The existence of the last morphism follows by applying $\bD_X$ to the first
  morphism of the distinguished triangle in \autoref{lem:DB-to-Irr} (with $\Gamma=E$)
  and noting that
  $\bD_X(\myR\pi_*\Omega_Y^{n-p}(\log E)(-E))\simeq
    \myR\pi_*\Omega_Y^p(\log E).$ 
\end{proof}

\noin
We also obtain the following simple observation:

\begin{cor}\label{cor:wtOm-of-rtl-sing}
  If $X$ is normal and $\wt\Omega_X^p$ is reflexive, then
  $\wt\Omega_X^p\simeq\Omega_X^{[p]} \simeq\wt\Irr_X^p$. This holds for instance if
  $X$ has rational \sings.
\end{cor}

\begin{proof}
  If $\wt\Omega_X^p$ is reflexive, then the natural morphism
  $\Omega_X^p\to\Omega_X^{[p]}$ factors through the natural morphism
  $\Omega_X^p\to\wt\Omega_X^p$. Therefore, by \autoref{prop:morphs-of-cxs} \te
  morphisms,
  \[
    \xymatrix{%
      \Omega_X^{[p]} \ar[r] & \wt\Omega_X^p \ar[r] & \wt\Irr_X^p\ar[r] &
      \left(\wt\Irr_X^p\right)^{\vee\vee}.  }%
  \]
  As $X$ is normal, it is nonsingular in codimension $1$, and these morphisms are
  isomorphisms on the complement of $\Sing X$.  Each sheaf in this diagram is
  torsion-free by \autoref{lem:hz-is-tf} and \autoref{cor:tw-irr-p-is-tf}, so the
  morphisms are injective.  The composition is between two reflexive sheaves, which
  is an isomorphism on the complement of a closed subset of codimension at least $2$,
  and hence it is an isomorphism on the entire $X$. Then the intermediate morphisms,
  which are injective, are also isomorphisms on the entire $X$.  This proves the
  first statement.
  If $X$ has \rtl \sings, then $\wt\Omega_X^p\simeq\Omega_X^{[p]}$ is reflexive by
  \cite[7.12]{MR3272910} and \cite[1.11]{MR4280862}.
  This has already been observed in \cite[2.5]{SVV23}.
\end{proof}


\numberwithin{equation}{thm}

\section{Singularities}\label{sec:singularities}
\noindent
\subsection{Defintions}
We are now ready to define the singularities we want to work with and prove our main
results.
\begin{rem}
  \emph{Rational} singularities were defined by Artin \cite{MR0199191}, initially for
  surfaces and then his definition was extended to higher dimensions. \emph{\DB}
  singularities were defined by Steenbrink \cite{Steenbrink83} and this notion may be
  viewed as a generalization of the notion of rational \sings. These were generalized
  to \emph{higher rational} and \emph{higher \DB} singularities in increasing
  generality in a series of papers by several authors
  \cite{MR4583654,MR4480883,MP22,FL-isolated,FL-lci}.  These definitions were
  initially made for lci singularities.
  Then it was pointed out by Shen, Venkatesh, and Vo \cite{SVV23}, and Tighe
  \cite{Tighe23} that these original definitions were too restrictive in the non-lci
  case. This observation opens the door to several alternatives in the non-lci case.
  %
  Several arguments, already in the rational and \DB case, only require the vanishing
  of higher cohomologies. Hence, it makes sense to study singularities satisfying
  that condition. This leads to the notions of \emph{\premrtl} and \emph{\premdb}.
  These might prevail as the most important class of those discussed here.

  Shen, Venkatesh, and Vo \cite{SVV23}, and Tighe \cite{Tighe23} also suggested
  alternative definitions in general which reduce to the previously used versions of
  these higher \DB and higher \rtl \sings in the lci case.  Unfortunately, in order
  to preserve compatibility of these new definitions with the original lci
  definitions, one also needs to impose a condition on the codimension of the
  singular set in general.  It seems reasonable that the same notions without this
  additional restriction are also worth studying. It turns out that in the case of
  higher rational singularities this is already taken care of by the notion of
  {\premrtl} \sings, but in the higher \DB case there seems to be another set of
  conditions that deserves its own name. 
  This additional notion, christened \emph{\wmdb} is added below.

  Finally, note that these higher rational and higher \DB \sings have been called
  $p$-\rtl, $p$-\DB, and more recently $k$-\rtl and $k$-\DB. Neither of these are
  perfect. First of all, there had already been a notion of $k$-\rtl \sings which was
  a weakening of the notion of \rtl \sings, only requiring the vanishing of the
  higher direct images of the sturcture sheaf of a resolution up to $k$, see e.g.,
  \cite{MR1891205,MR2856154}.  It seems that this new usage has already taken over,
  but one has to be careful when consulting references (currently) older than five
  years.  The actual letter to use could also be troublesome. As ``$k$'' is often
  used to denote the base field, ``$k$-\rtl'' overwhelmingly refers to \emph{rational
    points over the field $k$}. This makes searching for results in this area very
  difficult.
  To remedy this situation, I suggest the use of $m$-\rtl and $m$-\DB.
\end{rem}

In the following, we will use the extended definitions that $\Omega^p\leteq 0$ for
$p<0$.

\begin{defini}\label{def:mdb-mrtl}%
  cf.~\cite{SVV23}.
  Let $X$ be a reduced scheme of finite type over $\bC$ (or more generally over an
  algebraically closed field of characteristic zero) and let $m\in\bN$.
  \begin{enumerate}
  \item $X$ is said to have \emph{\premdb} \sings if the natural morphism
    $\wt\Omega^p_X\isomap\Om^p_X$ is an isomorphism for each $p\leq m$.
    This
    is equivalent to requiring that $h^i(\Om^p_X)=0$ for each $i>0$ and
    $p\leq m$.
  \end{enumerate}
  \begin{enumerate}[resume]
  \item $X$ is said to have \emph{\wmdb} \sings if
    \begin{itemize}
    \item $X$ is semi-normal,
    \item $X$ is \premdb, and
    \item $\wt\Omega^p_X$ is $S_2$ for each $1\leq p\leq m$. (If
      $m=0$, then there is no such $p$, of course.)
    \end{itemize}
  \end{enumerate}
  \begin{enumerate}[resume]
  \item $X$ is said to have \emph{\mdb} \sings if
    \begin{itemize}
    \item $X$ is {\wmdb}, 
    \item $\wt\Omega^p_X$ is reflexive for each $1\leq p\leq m$, and
    \item $\codim_X\Sing X\geq 2m+1$.
    \end{itemize}
  \end{enumerate}
  \begin{enumerate}[resume]
  \item $X$ is said to have \emph{\smdb} \sings if
    \begin{itemize}
    \item the natural morphism $\Omega_X^p\to \Om_X^p$ is an isomorphism for
      each $p\leq m$.
    \end{itemize}
  \end{enumerate}
  \begin{enumerate}[resume]
  \item $X$ is said to have \emph{\premrtl} \sings if
    \begin{itemize}
    \item the natural morphism $\wt\Irr^p_X\isomap \Irr^p_X$ is an isomorphism for
      each $p\leq m$.
    \end{itemize}
    As above, this is equivalent to requiring that $h^i(\Irr^p_X)=0$ for each $i>0$
    and $p\leq m$.
  \end{enumerate}
  \begin{enumerate}[resume]
  \item $X$ is said to have \emph{\mrtl} \sings if
    \begin{itemize}
    \item $X$ is normal, {\premrtl}, and
    \item $\codim_X\Sing X\geq 2m+2$.
    \end{itemize}
  \end{enumerate}
  \begin{enumerate}[resume]
  \item $X$ is said to have \emph{\smrtl} \sings if
    \begin{itemize}
    \item the natural morphism $\Omega_X^p\to \Irr_X^p$ is an isomorphism for each
      $p\leq m$.
    \end{itemize}
  \end{enumerate}  
\end{defini}

\begin{rem}\label{rem:for-normal-S2+tf=refl}
  By \autoref{lem:hz-is-tf}, $\wt\Omega_X^p$ is always torsion-free. This implies
  that if $X$ is normal, then $\wt\Omega_X^p$ is $S_2$ if and only if it is reflexive
  \cite[\href{https://stacks.math.columbia.edu/tag/0AVB}{Tag
    0AVB}]{stacks-project}.
\end{rem}

\begin{rem}
  As it has already been pointed out by several authors, if $X$ has rational
  singularities, then \cite{MR4280862} implies that
  $\wt\Irr^p_X\simeq \Omega^{[p]}_X$ is reflexive, which is the reason there is no
  need for an analogue of \emph{\wmdb} in the higher \rtl case.
\end{rem}

\begin{prop}
  For $m=0$, the notions of weakly $0$-\DB, $0$-\DB, strict $0$-\DB and \DB agree.
  If $X$ is semi-normal then they also agree with pre-0-\DB \sings. Similarly, if $X$
  is normal, the notions of pre-$0$-\rtl, $0$-rational, strict-$0$-\rtl, and rational
  agree. (Three of these notions, other than pre-$0$-\rtl, imply that $X$ is normal).
\end{prop}

\begin{proof}
  By \cite[5.2]{MR1741272}, $\wt\Omega_X^0\simeq\sO_{X_{\text{sn}}}$, where
  $X_{\text{sn}}$ is the semi-normalization of $X$. This implies the \DB case.
  Regarding the \rtl case, recall that $\Irr_X^0=\bD_X(\Om_X^n)$ and if
  $\pi:\wt X\to X$ is a resolution of \sings, then
  $\Om_X^n\simeq\pi_*\omega_{\wt X}$. It follows that if $X$ has pre-$0$-\rtl \sings,
  then
  $\pi_*\omega_{\wt X}\simeq\Om_X^n\simeq\bD_X(\Irr_X^0)\simeq\bD_X(\wt\Irr_X^0)$
\end{proof}


\begin{rem}
  The defintions of \premdb, \mdb, strict-$m$-\DB, \premrtl, \mrtl, and
  strict-$m$-\rtl here agree with the defintions of the corresponding notions (with
  $k\leteq m$) in \cite{SVV23}.  The defintion of \wmdb does not appear in
  \cite{SVV23}.
  %
  %
  %
  %
  It is suggested here as a potentially good notion that resembles the original \mdb
  definition in the lci case without the unnatural codimension condition. In this
  definition $\wt\Omega_X^p$ is required to be $S_2$ as opposed to the definition of
  \mdb that $\wt\Omega_X^p$ be reflexive. As $\wt\Omega_X^p$ is always torsion-free
  cf.~\autoref{lem:hz-is-tf}, these two requirements agree when $X$ is normal
  cf.~\autoref{rem:for-normal-S2+tf=refl}.
\end{rem}

\subsection{Equivalent characterization}

The class of \premdb \sings may be defined slightly differently. It may seem more
technical at first, but this is arguably the more natural way of thinking about these
singularities. The main theme of this equivalent characterization is to move the
focus from the associated graded quotients to the terms of the (hyper)filtrations.
\begin{lem}\label{lem:equiv-char-of-def-mdb-etc}
  Let $X$ be a reduced scheme of finite type over $\bC$ (or more generally over an
  algebraically closed field of characteristic zero) and let $m\in\bN$.  Then $X$ has
  \premdb \sings \ifft
  the natural morphism induced by the co-hyperfiltered morphism in \autoref{eq:12},
  $\tf^X_{m}\longrightarrow\uf^X_{m}$, is a co-hyperfiltered isomorphism.
\end{lem}
\begin{proof}
  This is straightforward from the definition and a repeated use of the $5$-lemma. 
\end{proof}

\subsection{Examples}

\begin{prop}\label{prop:wmdb+rtl=mdb}
  If $X$ has \premdb and \rtl \sings, then it is \wmdb.
\end{prop}

\begin{proof}
  The sheaf $\wt\Omega_X^p$ is reflexive for each $p$ by
  \autoref{cor:wtOm-of-rtl-sing}.
\end{proof}

%

\begin{prop}
  Let $X$ be 
  an snc variety. Then $X$ has \wmdb \sings for all $m$.
\end{prop}

\begin{proof}
  It is well known (cf.~\cite[7.23]{MR2393625}) that in this case,
  $\Om_X^\kdot\simeq \Omega_X^\kdot/\text{torsion}$, which implies immediately that
  $X$ has \premdb \sings.  Furthermore, the \v Cech resolution induced by the
  irreducible components of $X=\cup_iX_i$ give a cubic hyperresolution of $X$, which
  means that for each $p\in\bN$, \tes an exact sequence
  \begin{equation}
    \label{eq:59}
    \xymatrix{%
      0\ar[r] & \wt\Omega_X^p\ar[r] & \oplus_i \Omega_{X_i}^p \ar[r] &
      \oplus_{i,j}\Omega_{X_i\cap X_j}^p\ar[r] & \dots \ar[r] & \Omega_{\cap X_i}^p
      \ar[r] & 0.  }
  \end{equation}
  Then it follows from \autoref{lem:s2-for-les} (below) that $\wt\Omega_X^p$ is
  \CM 
  and torsion-free, and hence $X$ has \wmdb \sings for any $m\in\bN$.
\end{proof}

\begin{lem}\label{lem:s2-meta} Let
  $\xymatrix{%
      0\ar[r] & M \ar[r] & M_0\ar[r] & M_1\ar[r] & 0  }$
  be a short exact sequence of finite modules over a noetherian local ring $A$.  Then
  \begin{enumerate}
  \item\label{eq:58} $\depth M\geq \min ( \depth M_0, \depth M_1+1 )$.
  \item\label{eq:55} If $\dim M_1\geq\dim M_0-1$, then
    $\min \big(\min(n,\dim M_0), \min(n,\dim M_1+1) \big) \geq \min(n,\dim M)$.
  \item\label{item:24} If $M_0$ is $S_n$ and $M_1$ is $S_{n-1}$ for some $n\in\bN$,
    and $\dim (M_1)_{\frp}\geq\dim (M_0)_{\frp}-1$ for every prime ideal
    $\frp\ideal A$, then $M$ is also $S_n$.
  \end{enumerate}
\end{lem}

\begin{proof}
  The first statement, \autoref{eq:58}, is well-known. See for instance
  \cite[1.2.9]{MR1251956}.

  For the second statement, note that by the assumption the left hand side of
  \autoref{eq:55} equals $\min(n,\dim M_0)$. As $M\subseteq M_0$, it follows that
  $\dim M\leq \dim M_0$. This proves \autoref{eq:55}.

  Finally, let $\frp\ideal A$ be a prime ideal. Replace $A$ with $A_\frp$ and the
  modules $M,M_0$, and $M_1$ with their localization at $\frp$. Then
  $\depth M_0\geq \min(n,\dim M_0)$ and $\depth M_1+1\geq \min(n,\dim M_1+1)$ by
  assumption, so $\depth M\geq \min(n,\dim M)$ by \autoref{eq:58} and
  \autoref{eq:55}.
\end{proof}

\begin{lem}\label{lem:s2-for-les}
  Consider an exact sequence of modules or sheaves,
  \[
    \xymatrix{%
      0\ar[r] & M \ar[r] & M_0 \ar[r] & M_1\ar[r] & \dots \ar[r] & M_r\ar[r] & 0.  }
  \]
  Assume that $\dim M_i\geq\dim M_0-i$ for each $i=0,\dots,r$ and this remains true
  after localization at any prime.  Further assume that $M_i$ is $S_{n-i}$ for each
  $i=0,\dots,r$ for some $n$.  Then $M$ is $S_n$. Furthermore, if $M_0$ is
  torsion-free, then so is $M$.
\end{lem}

\begin{proof}
  The above exact sequence may be broken up into \sess and hence, by induction, it is
  enough to prove the statement for a \ses,
  $\xymatrix{%
      0\ar[r] & M_0 \ar[r] & M_1\ar[r] & M_2\ar[r] & 0.
    }$
  This case follows from \autoref{lem:s2-meta}.
  The last statement about torsion-freeness is trivial.
\end{proof}






\section{Hyperplane sections}\label{sec:hyperplane-sections}
\noindent
The following assumptions will be in effect for the entire section. 
\begin{assume}\label{assume-XHSL}
  Let $X$ be a reduced scheme essentially of finite type over $\bC$,
  $\Sigma\subseteq X$ a closed subset, $H$ a general member of a basepoint-free
  linear system, and $L$ an effective Cartier divisor.
\end{assume}

\noin Let us first record a few well-known observations.

\begin{lem}\label{lem:Bertini-for-Sing}
  $\Sing H=H\cap\Sing X$.
\end{lem}

\begin{proof}
  By the second theorem of Bertini cf.~\cite[Cor.~1]{MR0825140}
  $\Sing H\subseteq H\cap\Sing X$. On the other hand, let $x\in H\setminus\Sing H$
  and $f\in\sO_{X,x}$ be a local defining equation for $H$. Then $f$ is a regular
  element of $\sO_{X,x}$ and by the choice of $x$, the quotient ring
  $\sO_{X,x}/(f)\simeq \sO_{H,x}$ is regular. It follows that then $\sO_{X,x}$ is
  regular. This implies that $H\setminus\Sing H\subseteq X\setminus\Sing X$, i.e.,
  that $H\cap\Sing X\subseteq \Sing H$.
\end{proof}

\noin The following must be known to experts. It is included for completeness.

\begin{lem}\label{lem:ses-for-CM}
  If $X$ is connected and \CM, then there is a \ses,
  \[
    \xymatrix{%
      0\ar[r] & \omega_X \ar[r] & \omega_X(L) \ar[r] & \omega_L \ar[r] & 0.  }
  \]
\end{lem}

\begin{proof}
  Let $\dcx X$ denote the dualizing complex of $X$ and consider the \ses,
  \begin{equation}
    \label{eq:74}
    \xymatrix{%
      0\ar[r] & \sO_X(-L) \ar[r] & \sO_X \ar[r] & \sO_L \ar[r] & 0.
    }
  \end{equation}
  Apply the duality functor $\myR\sHom_X(\blank, \dcx X)$ to this to obtain the \dt,
  \[
    \xymatrix{%
       \dcx L \ar[r] &  \dcx X \ar[r] & \dcx X (L) \ar[r]^-{+1} & .
    }
  \]
  Considering the long exact cohomology sequence associated to this \dt combined with
  the fact that $\dcx X$ has only one non-zero cohomology yields the desired
  statement.
\end{proof}

The following is used ubiquitously, yet it is difficult to find a simple reference
for it. It follows from the more general (and more complicated) Th\'eor\'eme 4 of
\cite{Elkik78}. Below is a simple proof.

\begin{lem}\label{lem:hyperplane-sections-of-rtl}
  If $X$ has \rtl \sings then so does $H$.
\end{lem}

\begin{proof}
  As $X$ is normal, so is $H$ by \cite[3.4.9]{Flenner-OCarrol-Vogel} and $H$ is \CM
  trivially.
  Let $\pi:\wt X\to X$ be a resolution of \sings and let $\wt H\leteq
  \pi^{-1}H$. Then $\pi'\leteq\pi\resto{\wt H}\col \wt H\to H$ is a resolution of
  singularities of $H$ and as a divisor $\wt H\subseteq \wt X$ is linearly equivalent
  to $\pi^*H$.  Consider the pushforward of the adjunction morphism,
  $\omega_{\wt X}(\wt H)\to \omega_{\wt H}$ and its natural $2$-morphism to the
  corresponding morphism on $X$:
  \[
    \xymatrix{%
      \pi_*\omega_{\wt X}(\wt H) \simeq \left(\pi_*\omega_{\wt X}\right)(H)
      \ar[d]^\alpha \ar[r]
      &  \pi'_*\omega_{\wt H} \ar[d]^\beta \\
      \omega_{X}(H) \ar[r]^-\varrho & \omega_{H} }
  \]
  By assumption, $\alpha$ is an isomorphism, and as $X$ is \CM, $\varrho$ is
  surjective by \autoref{lem:ses-for-CM}. It follows that then $\varrho\circ\alpha$
  is surjective, which implies that then so is $\beta$. However, $\beta$ is
  injective, because it is generically injective and $\omega_{\wt H}$ is
  torsion-free, so $\beta$ is an isomorphism and hence $H$ has rational singularities
  by Kempf's criterion.
\end{proof}

\begin{lem}\label{lem:ses-for-hyperplane-section}
  For each $p\in\bN$ there is a distinguished triangle,
  \[
    \xymatrix{%
      \Om^{p-1}_H(\log\Sigma\resto H)\otimes\sO_H(-H) \ar[r] &
      \Om^p_X(\log\Sigma)\lotimes\sO_H \ar[r] & \Om^p_H(\log\Sigma\resto H)
      \ar[r]^-{+1} & .}
  \]
\end{lem}

\begin{proof}
  Let $\pi_\kdot:X_\kdot\to X$ be a hyperresolution of $X$ such that for each
  $\alpha$, $\Sigma_\alpha\leteq\pi_\alpha^*\Sigma\subseteq X_\alpha$ is either
  empty, equal to $X_\alpha$, or is an snc divisor on $X_\alpha$. As $H$ is a general
  member of a basepoint-free linear system,
  $\pi^H_\kdot: H_\kdot\leteq H\times_XX_\kdot\to H$ is also a hyperresolution.
  Furthermore, $H_\alpha+\Sigma_\alpha$ is an snc divisor on $X_\alpha$ for each
  $\alpha$ for which $\Sigma_\alpha\neq X_\alpha$.  Then on each component of the
  hyperresolution for which $\Sigma_\alpha\neq X_\alpha$, there is a \ses,
  \[
    \xymatrix{%
      0\ar[r] & \Omega^{p-1}_{H_\alpha}(\log{\Sigma_\alpha}\resto{H_\alpha})
      \otimes\sO_{H_\alpha}(-{H_\alpha}) \ar[r] &
      \Omega^p_{X_\alpha}(\log\Sigma_\alpha) \otimes\sO_{H_\alpha} \ar[r] &
      \Omega^p_{H_\alpha}(\log{\Sigma_\alpha}\resto{H_\alpha}) \ar[r] & 0.}
  \]
  Applying $\myR(\pi_\kdot)_*$ to this 
  leads to the desired distinguished triangle once we establish that
  \begin{equation}
    \label{eq:15}
    \myR(\pi_\kdot)_*\left(\Omega^p_{X_\kdot}(\log\Sigma_\kdot)
      \otimes\sO_{H_\kdot}\right)\simeq\Om^p_X(\log\Sigma)\lotimes\sO_H.
  \end{equation}
  This follows easily by considering the system of short exact sequences on
  $X_\kdot$, 
  \[
    \xymatrix{%
      0\ar[r] &
      \Omega^p_{X_\kdot}(\log\Sigma_\kdot)\otimes\sO_{X_\kdot}(-H_\kdot)\ar[r] &
      \Omega^p_{X_\kdot}(\log\Sigma_\kdot) \ar[r] &
      \Omega^p_{X_\kdot}(\log\Sigma_\kdot)\otimes\sO_{H_\kdot} \ar[r] & 0. }
  \]
  Applying $\myR(\pi_\kdot)_*$ to this \ses and using the projection formula for the
  first term gives
  \[
    \xymatrix{%
      \Om^p_X(\log\Sigma)\otimes \sO_X(-H) \ar[r] & \Om^p_X(\log\Sigma) \ar[r] &
      \myR(\pi_\kdot)_*\left(\Omega^p_{X_\kdot}(\log\Sigma_\kdot)\otimes\sO_{H_\kdot}\right)
      \ar[r]^-{+1} &. }
  \]
  This proves the required isomorphism \autoref{eq:15} and hence the desired
  statement.
\end{proof}

\begin{cor}\label{cor:premmdb-ses}
  If $H$ has \premmdb \sings, then for each $p\leq m$ there is a \ses,
  \[
    \xymatrix{%
      0\ar[r] & \wt\Omega_H^{p-1} \otimes \sO_H(-H) \ar[r] &
      \wt\Omega_X^p\otimes\sO_H \ar[r] & \wt\Omega_H^p \ar[r] & 0.  }
  \]
\end{cor}

\begin{proof}
  Consider the long exact cohomology sequence associated to the \dt in
  \autoref{lem:ses-for-hyperplane-section} with $\Sigma=\emptyset$. The sheaves in
  the above sequence are the $h^0$ sheaves of those complexes by
  \autoref{cor:h0-of-lotimes} and $h^1(\Om_H^{p-1}\otimes\sO_H(-H))=0$ by assumption.
\end{proof}

\noin We can also compare the logarithmic complexes with different loci of poles.

\begin{lem}\label{lem:ses-for-log}
  For each $p\in\bN$ there exist the following distinguished triangles:
  \begin{gather}
    \label{eq:20}\xymatrix{%
      \Om^p_X(\log\Sigma) \ar[r]^-\phi & \Om^p_X\left(\log (\Sigma+H)\right)\ar[r] &
      \Om^{p-1}_H(\log\Sigma\resto H)
      \ar[r]^-{+1} & 
    }\\
    \xymatrix{%
      \Om^p_X\left(\log (\Sigma+H)\right)\otimes\sO_X(-H) \ar[r]^-\psi &
      \Om^p_X(\log\Sigma)\ar[r] & \Om^{p}_H(\log\Sigma\resto H) \ar[r]^-{+1} & .}
  \end{gather}
  Furthermore, the composition of the morphisms $\phi$ and $\psi$ in the above
  diagrams agree with the morphism induced by the embedding $\sO_X(-H)\into\sO_X$:
  \begin{equation}
    \label{eq:50}
    \phi\circ\psi: \Om^p_X\left(\log (\Sigma+H)\right)\otimes\sO_X(-H)
    \longrightarrow
    \Om^p_X\left(\log (\Sigma+H)\right).
  \end{equation}
\end{lem}

\begin{proof}
  Let $\pi_\kdot:X_\kdot\to X$ be a hyperresolution of $X$ such that for each
  $\alpha$, $\Sigma_\alpha\leteq\pi_\alpha^*\Sigma\subseteq X_\alpha$ is either
  empty, equal to $X_\alpha$, or is an snc divisor on $X_\alpha$. As $H$ is a general
  member of a basepoint-free linear system,
  $\pi^H_\kdot: H_\kdot\leteq H\times_XX_\kdot\to H$ is also a hyperresolution.
  Furthermore, $\Sigma_\alpha+H_\alpha$ is an snc divisor on $X_\alpha$ for each
  $\alpha$ for which $\Sigma_\alpha\neq X_\alpha$.  Then on each component of the
  hyperresolution for which $\Sigma_\alpha\neq X_\alpha$, there exist \sess
  (cf.\cite[2.3]{EV92}),
  \begin{equation}
    \label{eq:24}
    \begin{gathered}
    \xymatrix{%
      0\ar[r] & \Omega^p_{X_\alpha}(\log\Sigma_\alpha) \ar[r]^-{\phi_\alpha} &
      \Omega^p_{X_\alpha}\left(\log ({\Sigma_\alpha+H_\alpha})\right)\ar[r] &
      \Omega^{p-1}_{H_\alpha}(\log ({\Sigma_\alpha}\resto {H_\alpha}) \ar[r] & 0,
    }\\
    \xymatrix{%
      0\ar[r] & \Omega^p_{X_\alpha}\left(\log
        ({\Sigma_\alpha+H_\alpha})\right)\otimes\sO_{X_\alpha}(-{H_\alpha})
      \ar[r]^-{\psi_\alpha} & \Omega^p_{X_\alpha}(\log\Sigma_\alpha)\ar[r] &
      \Omega^{p}_{H_\alpha}(\log ({\Sigma_\alpha}\resto {H_\alpha}) \ar[r] & 0. }      
    \end{gathered}
    \end{equation}
  Applying $\myR(\pi_\kdot)_*$ to these \sess gives the desired \dts.

  The composition of the morphisms $\phi_\alpha$ and $\psi_\alpha$ in the above
  diagrams agree with the morphism induced by the embedding $\sO_X(-H)\into\sO_X$ by
  construction, so the same is true for $\phi\circ\psi$.
\end{proof}
\noin
We obtain similar \dts for the corresponding filtration complexes.

\begin{lem}\label{cor:ses-for-filt-of-log}
  For each $p\in\bN$ there exist the following \dts:
  \begin{gather}
    \label{eq:25}
    \xymatrix{%
      \uf^p_X(\log\Sigma) \ar[r]^-{\phi_1}
      & \uf^p_X\left(\log (\Sigma+H)\right)
      \ar[r]& \uf^{p-1}_H(\log\Sigma\resto H)[-1] \ar[r]^-{+1} &
    }\\
     \label{eq:23}
    \xymatrix{%
      \uf_p^X(\log\Sigma) \ar[r]^-{\phi_2} & \uf_p^X\left(\log (\Sigma+H)\right)
      \ar[r]& \uf_{p-1}^H(\log\Sigma\resto H)[-1] \ar[r]^-{+1} &
    }
  \end{gather}
\end{lem}

\begin{proof}
  We will prove \autoref{eq:25}, and
  \autoref{eq:23} follows essentially the same way.

  As in the proof of \autoref{lem:ses-for-log}, let $\pi_\kdot:X_\kdot\to X$ be a
  hyperresolution of $X$ such that for each $\alpha$,
  $\Sigma_\alpha\leteq\pi_\alpha^*\Sigma\subseteq X_\alpha$ is either empty, equal to
  $X_\alpha$, or is an snc divisor on $X_\alpha$. As $H$ is a general member of a
  basepoint-free linear system, $\pi^H_\kdot: H_\kdot\leteq H\times_XX_\kdot\to H$ is
  also a hyperresolution.  Furthermore, $\Sigma_\alpha+H_\alpha$ is an snc divisor on
  $X_\alpha$ for each $\alpha$ for which $\Sigma_\alpha\neq X_\alpha$.  Then on each
  component of the hyperresolution for which $\Sigma_\alpha\neq X_\alpha$, there
  exists a diagram:  
  \[
    \xymatrix@R01.5em{%
      & 0\ar[d] & 0\ar[d] & 0\ar[d] &\\
      0\ar[r] & \ar[d] \f^{p+1}_{X_\alpha}(\log\Sigma_\alpha) \ar@{..>}[r] & \ar[d]
      \f^{p+1}_{X_\alpha}(\log(\Sigma_\alpha+H_\alpha)) \ar@{..>}[r]& \ar[d]
      \f^{p}_{H_\alpha}(\log{\Sigma_\alpha}\resto{H_\alpha})[-1] \ar[r] & 0\\
      0\ar[r] & \f^{p}_{X_\alpha}(\log\Sigma_\alpha) \ar[d] \ar[r] &
      \f^{p}_{X_\alpha}(\log(\Sigma_\alpha+H_\alpha)) \ar[d] \ar[r]&
      \f^{p-1}_{H_\alpha}(\log{\Sigma_\alpha}\resto{H_\alpha})[-1] \ar[r]\ar[d]
      & 0. \\
      0\ar[r] & \Omega^p_{X_\alpha}(\log\Sigma_\alpha)[-p] \ar[r]\ar[d] &
      \Omega^p_{X_\alpha}(\log(\Sigma_\alpha+H_\alpha))[-p] \ar[r]\ar[d] &
      \Omega^{p-1}_{H_\alpha}(\log{\Sigma_\alpha}\resto{H_\alpha})[-p] \ar[r]\ar[d]  & 0 \\
      & 0 & 0 & 0 &}
  \]
  Here the last row exists and is a \ses by \autoref{eq:24}, the columns are \sess,
  and come from \autoref{eq:5} and its log analogue, while the second row is a \ses
  by induction on $p$.
  Then the morphisms indicated by the dotted arrows exist and the first row is also a
  \ses by the 9-lemma.  Applying $\myR(\pi_\kdot)_*$ to this \ses (the first row)
  gives the desired \dt.
\end{proof}

\noin For the previous statement to be useful, we need a similar comparison for the
\hz-complexes.

\begin{lem}\label{cor:log-sequence-for-wt}
  If $X$ has \premdb \sings. Then
  \begin{enumerate}
  \item\label{eq:28} $\xymatrix{%
      0\ar[r] & \wt\Omega^p_X \ar[r] & \wt\Omega^p_X\left(\log H\right) \ar[r] &
      \wt\Omega^{p-1}_H \ar[r] & 0}$ is a \ses for each $p\leq m$,
  \item\label{eq:30} $\xymatrix{%
      \tf_{p}^X \ar[r] & \tf_{p}^X\left(\log H\right) \ar[r]& \tf_{p-1}^H[-1]
      \ar[r]^-{+1} & }$ is a \dt for each $p\leq m$, and
  \item\label{eq:68} \
    $\tf_{p}^X\left(\log H\right)\simeq \uf_{p}^X\left(\log H\right)$ for each
    $p\leq m$.
  \end{enumerate}
\end{lem}

\begin{proof}
  Consider the long exact sequence of cohomology for the \dt in \autoref{eq:20} (with
  $\Sigma=\emptyset$) and observe that $h^1$ of the first term is zero by the
  assumption and otherwise the sheaves in \autoref{eq:28} are the $h^0$ sheaves of
  the respective complexes. This proves \autoref{eq:28} and \autoref{eq:30}. Then
  \autoref{eq:68} follows from \autoref{eq:25}, \autoref{eq:30},
  \autoref{prop:hyperplane-sections-of-premdb}, and the derived category $5$-lemma.
\end{proof}

\begin{prop}\label{prop:hyperplane-sections-of-premdb}
  If $X$ has \premdb \sings then so does $H$.
\end{prop}

\begin{proof}
  Use induction on $p\leq m$ and prove that the conditions in \autoref{def:mdb-mrtl}
  are satisfied. The $p=-1$ case is trivially true. Assume that $p\geq 0$.  Applying
  \autoref{cor:prime-avoidance--injectivity} and \autoref{lem:ltimes-is-times} to the
  cohomology sheaves of $\Om^p_X$ implies that
  \[
    h^i(\Om^p_X\lotimes\sO_H)\simeq h^i(\Om^p_X)\lotimes\sO_H
  \]
  It follows that if $X$ has \premdb \sings, then $h^i(\Om^p_X\lotimes\sO_H)=0$ for
  $i>0$.

  Next, consider the distinguished triangle from
  \autoref{lem:ses-for-hyperplane-section}:
  \[
    \xymatrix{%
      \Om^{p-1}_H\otimes\sO_H(-H) \ar[r] & \Om^p_X\lotimes\sO_H \ar[r] & \Om^p_H
      \ar[r]^-{+1} & .}
  \]
  By induction $h^i(\Om^{p-1}_H)=0$ for $i>0$,
  so the desired statement follows.
\end{proof}

\begin{prop}\label{prop:hyperplane-sections-of-mdb}
  Assume that $X$ has \rtl \sings. Then, if $X$ has \wmdb (respectively \mdb) \sings,
  then so does $H$.
\end{prop}

\begin{proof}
  $H$ has \premdb \sings by \autoref{prop:hyperplane-sections-of-premdb}.  If $X$ has
  \rtl \sings, then $H$ has \rtl \sings by \autoref{lem:hyperplane-sections-of-rtl}
  and consequently, it has \wmdb \sings by \autoref{prop:wmdb+rtl=mdb}.  The \mdb
  case follows from \autoref{lem:Bertini-for-Sing}.
\end{proof}

\begin{rem}
  Note that \autoref{prop:hyperplane-sections-of-premdb} and the \mdb case of
  \autoref{prop:hyperplane-sections-of-mdb} were already obtained in \cite[Thm.~A,
  Thm.~D(1)]{SVV23}. It is included here for completeness, because we use slightly
  different definitions of some of these \sings.
\end{rem}

\noin The conclusion of \autoref{prop:hyperplane-sections-of-mdb} also
holds under different conditions.

\begin{lem}\label{claim:Sq}
  If $H$ has \premdb \sings and for some $p\leq m+1$ and for some $r$,
  $\wt\Omega_X^p$ is $S_r$ and $\wt\Omega_H^{p-1}$ is $S_{r+1}$. Then $\wt\Omega_H^p$
  is $S_{r}$.
\end{lem}
\begin{proof}
  Consider the \ses from \autoref{cor:premmdb-ses}:
  \[
    \xymatrix{%
      0\ar[r] & \wt\Omega_H^{p-1} \otimes \sO_H(-H) \ar[r] &
      \wt\Omega_X^p\otimes\sO_H \ar[r] & \wt\Omega_H^p \ar[r] & 0.  }
  \]
  The statement follows from \cite[1.2.9]{MR1251956}.
\end{proof}

\begin{lem}\label{cor:SQ}
  If there is an $a\in\bN$ such that $\wt\Omega_X^p$ is $S_{a-p}$ for $p\leq m$.
  Then $\wt\Omega_H^p$ is $S_{a-p}$ for $p\leq m$.
\end{lem}

\begin{proof}
  For $p=0$, $\wt\Omega_X^0\simeq\sO_{X_{\text{sn}}}$, by
  \cite[5.2]{MR1741272}, where $X_{\text{sn}}$ is the
  semi-normalization of $X$.
  The pre-image of $H$ in $X_{\text{sn}}$ is the semi-normalization of
  $H$ by \cite[2.5]{MR1017925}, and hence the statement in this case
  follows from \cite[12.1.6]{EGAIV3}.
  By induction on $p$ we may assume that $\wt\Omega_H^{p-1}$ is
  $S_{a+1-p}$ and then the statement follows from \autoref{claim:Sq}.
\end{proof}

\begin{prop}\label{prop:wmdb-for-non-normal}
  If $X$ has \wmdb (respectively \mdb) \sings and $\wt\Omega_X^p$ is $S_{m+2-p}$ for
  $p\leq m$ (e.g., it is \CM), then so does $H$.
\end{prop}


\begin{proof}
  $H$ has \premdb \sings by
  \autoref{prop:hyperplane-sections-of-premdb} and it is semi-normal
  by \cite[2.5]{MR1017925}. Furthermore, $\wt\Omega_H^p$ is $S_2$ by
  \autoref{cor:SQ} for $p\leq m$, so $H$ has \wmdb \sings.
  The \mdb case, as before, follows from \autoref{lem:Bertini-for-Sing}.
  %
\end{proof}

\begin{rem}
  As opposed to \autoref{prop:hyperplane-sections-of-mdb}, 
  \autoref{prop:wmdb-for-non-normal} works even when $X$ is not normal.
\end{rem}




\section{Cyclic covers}\label{sec:cyclic-covers}
\noindent
\subsection{The effect of cyclic covers on complexes of differential forms}
\begin{notation}\label{setup}
  Let $X$ be a complex scheme, $\imath:\Sigma\into X$ a closed subset,
  $\jmath:V\leteq X\setminus\Sigma\into X$ the corresponding open embedding, $\sL$ a
  semi-ample line bundle on $X$, and $s\in H^0(X,\sL^N)$ a general section for some
  $N\gg 0$.  Let $H=(s=0)$ be its zero locus and let
  \[
    \eta: Y\leteq \Spec \oplus _{i=0}^{N-1}\sL^{-i} \to
    X
  \]
  denote the cyclic cover corresponding to the section $s$. Finally, let
  $H'\leteq (\eta^*H)_\red$, $\Sigma'\leteq\eta^{-1}\Sigma$,
  $V'\leteq\eta^{-1}V=Y\setminus\Sigma'$, and $p\in\bN$.
\end{notation}

\noin %
The following is a slightly generalized singular version of \cite[3.16]{EV92}:

\begin{thm}\label{lem:finite-cover}
  Using the notation and setup of \eqref{setup}, we have the following:
  \begin{enumerate}
  \item\label{item:10} The functor $\eta_*=\myR\eta_*$ is exact,
  \item\label{item:11}
    $\eta_*\Om_Y^p\left(\log (\Sigma'+H')\right)\simeq \oplus_{i=0}^{N-1}\Om_X^p\left(\log
      (\Sigma+H)\right) \otimes \sL^{-i}$, and
  \item\label{item:13}
    $\eta_*\Om_Y^p(\log \Sigma')\simeq \Om_X^p(\log \Sigma) \oplus
    \left(\oplus_{i=1}^{N-1}\Om_X^p\left(\log (\Sigma+H)\right) \otimes
      \sL^{-i}\right)$.
  \item\label{item:12}
    $\eta_*\Om_Y^p\simeq \Om_X^p \oplus \left(\oplus_{i=1}^{N-1}\Om_X^p(\log H)
      \otimes \sL^{-i}\right)$.
  \end{enumerate}
\end{thm}

\begin{proof}
  As $\eta$ is finite, statement \autoref{item:10} is trivial.
  
  Let $\pi_\kdot:X_\kdot\to X$ be a hyperresolution of $X$ such that for each
  $\alpha$, $\Sigma_\alpha\leteq\pi_\alpha^*\Sigma\subseteq X_\alpha$ is either
  empty, equal to $X_\alpha$, or is an snc divisor on $X_\alpha$. As $H$ is a general
  member of a basepoint-free linear system,
  $\pi^H_\kdot: H_\kdot\leteq H\times_XX_\kdot\to H$ is also a hyperresolution.
  Furthermore, $\Sigma_\alpha+H_\alpha$ is an snc divisor on $X_\alpha$ for each
  $\alpha$ for which $\Sigma_\alpha\neq X_\alpha$.

  Recall that $s$ is a general section of a globally generated line bundle and hence
  for each $\alpha$ there exists a finite cyclic cover
  $\eta_\alpha:Y_\alpha\to X_\alpha$ given by
  $\pi_\alpha^*s\in H^0(X_\alpha,\pi_\alpha^*\sL^N)$.  Let
  $\Sigma_\alpha\leteq \pi_\alpha^{-1}\Sigma$, $H_\alpha\leteq\pi_\alpha^*H$,
  $\Sigma'_\alpha\leteq \eta_\alpha^{-1}\Sigma_\alpha$ and
  $H'_\alpha\leteq(\eta_\alpha^*H_\alpha)_\red$. The $Y_\alpha$ are smooth by
  construction and hence they form a cubic hyperresolution of $Y$:
  $\varrho_\kdot:Y_\kdot\to Y$. Furthermore, $\Sigma_\alpha+H_\alpha$ and
  $\Sigma'_\alpha+H'_\alpha$ are snc divisors on $X_\alpha$ and $Y_\alpha$
  respectively for each $\alpha$ for which $\Sigma_\alpha\neq X_\alpha$, and for
  those $\alpha$ there exists a natural embedding
  \begin{equation}
    \label{eq:43}
    \eta_\alpha^*\Omega_{X_\alpha}^p\left(\log (\Sigma_\alpha+H_\alpha)\right) \into
    \Omega_{Y_\alpha}^p\left(\log (\Sigma'_\alpha+H'_\alpha)\right),
  \end{equation}
  which is an isomorphism outside of $\Sigma'_\alpha\cap H'_\alpha$. This can be confirmed
  by a simple local calculation, cf.~\cite[1.6, Remark]{Viehweg82b},
  \cite[1.2]{Esn-Vie82}.
  Applying the exact functor $(\eta_\alpha)_*$ to \autoref{eq:43} yields an embedding
  which is an isomorphism in codimension $1$ and since the sheaf on the left hand
  side is locally free and hence reflexive, that embedding is actually an
  isomorphism:
  \begin{equation}
    \label{eq:44}
    (\eta_\alpha)_*    \eta_\alpha^*\Omega_{X_\alpha}^p\left(\log
      (\Sigma_\alpha+H_\alpha)\right) \simeq
    (\eta_\alpha)_*  \Omega_{Y_\alpha}^p\left(\log (\Sigma'_\alpha+H'_\alpha)\right),
  \end{equation}
  which implies that
  \begin{equation}
    \label{eq:54}
    (\eta_\alpha)_*  \Omega_{Y_\alpha}^p\left(\log
      (\Sigma'_\alpha+H'_\alpha)\right)\simeq
    \Omega_{X_\alpha}^p\left(\log
      (\Sigma_\alpha+H_\alpha)\right) \otimes (\eta_\alpha)_*  \sO_{Y_\alpha}.
  \end{equation}
  By construction, $\eta\circ\varrho_\kdot= \pi_\kdot\circ \eta_\kdot$, so it
  follows, using \autoref{eq:54}, that
  \begin{multline*}
    \eta_*\Om_Y^p\left(\log(\Sigma'+H')\right) =
    \myR\eta_*\Om_Y^p\left(\log(\Sigma'+H')\right) \simeq
    \myR\eta_*\myR\varrho_{\kdot
      *}\Omega_{Y_\kdot}^p\left(\log(\Sigma'_\kdot+H'_\kdot)\right)     \simeq \\
    \simeq \myR\pi_{\kdot *} \myR\eta_{\kdot
      *}\Omega_{Y_\kdot}^p\left(\log(\Sigma'_\kdot+H'_\kdot)\right) \simeq
    \myR\pi_{\kdot *}\left(\oplus_{i=0}^{N-1}\Omega_{X_\kdot}^p\left(\log
        (\Sigma_{\kdot}+H_\kdot)\right) \otimes\pi_\kdot^*\sL^{-i}\right) \simeq \\
    \simeq \oplus_{i=1}^{N-1}\myR\pi_{\kdot *}\Omega_{X_\kdot}^p\left(\log
      (\Sigma_{\kdot}+H_\kdot)\right) \otimes \sL^{-i} \simeq
    \oplus_{i=0}^{N-1}\Om_X^p\left(\log (\Sigma+H)\right) \otimes\sL^{-i}.
  \end{multline*}
  This proves \autoref{item:11}.
  Next, the local computation in the proof of \cite[3.16(d),p.29]{EV92}, accomodating
  for the logarithmic poles along $\Sigma_\alpha$ shows that
  \begin{equation}
    \label{eq:16}
    \eta_{\alpha*}\Omega_{Y_\alpha}^p(\log \Sigma'_\alpha)
    \simeq
    \Omega_{X_\alpha}^p(\log \Sigma_\alpha)\oplus
    \left(\oplus_{i=1}^{N-1}\Omega_{X_\alpha}^p\left(\log 
        (\Sigma_\alpha+H_\alpha)\right) 
      \otimes \pi_\alpha^*\sL^{-i}\right)
  \end{equation}
  Similarly to the above calculation, we have that
  \begin{multline*}
    \eta_*\Om_Y^p(\log \Sigma') = \myR\eta_*\Om_Y^p(\log \Sigma') \simeq
    \myR\eta_*\myR\varrho_{\kdot *}\Omega_{Y_\kdot}^p(\log \Sigma'_\kdot) \simeq
    \myR\pi_{\kdot *}\myR\eta_{\kdot *}\Omega_{Y_\kdot}^p(\log \Sigma'_\kdot) \simeq \\
    \simeq \myR\pi_{\kdot *}\left(\Omega_{X_\kdot}^p(\log
      \Sigma_\kdot)\oplus\left(\oplus_{i=1}^{N-1}\Omega_{X_\kdot}^p\left(\log
          (\Sigma_{\kdot}+H_\kdot)\right) \otimes\pi_\kdot^*\sL^{-i}\right)\right)
    \simeq \qquad \\ \qquad \simeq \myR\pi_{\kdot *}\Omega_{X_\kdot}^p (\log
    \Sigma_\kdot)\oplus \left( \oplus_{i=1}^{N-1}\myR\pi_{\kdot
        *}\Omega_{X_\kdot}^p\left(\log (\Sigma_{\kdot}+H_\kdot)\right) \otimes
      \sL^{-i}\right) \simeq \\ \simeq \Om_X^p(\log \Sigma) \oplus
    \left(\oplus_{i=1}^{N-1}\Om_X^p\left(\log (\Sigma+H)\right)
      \otimes\sL^{-i}\right),
  \end{multline*}
  which proves \autoref{item:13}. Choosing $\Sigma=\emptyset$ gives
  \autoref{item:12}.
\end{proof}

\begin{cor}\label{cor:finite-cover}
  We also have that
  \begin{enumerate}
  \item $\eta_*\wt\Omega_Y^p\left(\log(\Sigma'+H')\right) \simeq
    \oplus_{i=0}^{N-1} \wt\Omega_X^p\left(\log(\Sigma+H)\right) \otimes \sL^{-i}$,

    \noin
    and the direct sum decomposition $\eta_*\sO_Y\simeq
    \oplus_{i=0}^{n+1}\sL^{-i}$ is compatible with the natural morphism
    \begin{multline*}
      \eta_*\wt\Omega_Y^p\left(\log(\Sigma'+H')\right)\simeq \oplus_{i=0}^{N-1}
      \wt\Omega_X^p\left(\log(\Sigma+H)\right) \otimes \sL^{-i} \longrightarrow \\
      \longrightarrow \eta_*\Om_Y^p\left(\log(\Sigma'+H')\right)
      \simeq\oplus_{i=0}^{N-1}\Om_X^p\left(\log(\Sigma+H)\right) \otimes \sL^{-i},
    \end{multline*}
  \item $\eta_*\wt\Omega_Y^p(\log \Sigma')\simeq \wt\Omega_X^p(\log
    \Sigma)\oplus\left(\oplus_{i=1}^{N-1} \wt\Omega_X^p\left(\log (\Sigma+H)\right)
      \otimes \sL^{-i}\right)$,

    \noin and the direct sum decomposition
    $\eta_*\sO_Y\simeq \oplus_{i=0}^{n+1}\sL^{-i}$ is compatible with the natural
    morphism
    \begin{multline*}
      \eta_*\wt\Omega_Y^p(\log \Sigma')\simeq \wt\Omega_X^p(\log
      \Sigma)\oplus\left(\oplus_{i=1}^{N-1} \wt\Omega_X^p\left(\log (\Sigma+H)\right)
        \otimes \sL^{-i}\right) \longrightarrow \\ \longrightarrow \eta_*\Om_Y^p(\log
      \Sigma')\simeq \Om_X^p(\log \Sigma) \oplus
      \left(\oplus_{i=1}^{N-1}\Om_X^p\left(\log (\Sigma+H)\right) \otimes
        \sL^{-i}\right),
    \end{multline*}
  \item $\eta_*\wt\Omega_Y^p\simeq \wt\Omega_X^p\oplus\left(\oplus_{i=1}^{N-1}
      \wt\Omega_X^p(\log H) \otimes
      \sL^{-i}\right)$,

    \noin
    and the direct sum decomposition $\eta_*\sO_Y\simeq
    \oplus_{i=0}^{n+1}\sL^{-i}$ is compatible with the natural morphism
    \[
      \eta_*\wt\Omega_Y^p\simeq \wt\Omega_X^p\oplus\left(\oplus_{i=1}^{N-1}
        \wt\Omega_X^p(\log H) \otimes \sL^{-i}\right) \longrightarrow
      \eta_*\Om_Y^p\simeq \Om_X^p \oplus \left(\oplus_{i=1}^{N-1}\Om_X^p(\log H)
        \otimes \sL^{-i}\right),
    \]
  \end{enumerate}  
\end{cor}

\begin{proof}
  In the following, $(\ast)$ stands for one of
  $\Sigma,\Sigma',\Sigma+H,\Sigma'+H'$ or
  $\emptyset$, whichever makes sense in the context. The morphism
  $\eta_*$ is exact, hence
  \[
    h^0(\eta_*\Om_Y^p\left(\log(\ast)\right))\simeq
    \eta_*\wt\Omega_Y^p\left(\log(\ast)\right),
  \]
  and the sheaves $\sL^{-i}$ are locally free, so
  \begin{equation*}
    h^0(\Om_X^p(\log (\ast)) \otimes \sL^{-i})
    \simeq \wt\Omega_X^p(\log (\ast)) \otimes
    \sL^{-i} 
  \end{equation*}
  The compatibility with the direct sum decomposition follows from the proof of
  \autoref{lem:finite-cover}.
\end{proof}

\subsection{The effect of cyclic covers on filtrations and co-filtrations}

\begin{prop}\label{cor:connection-for-f}
  Continuing to use the notation and setup of \eqref{setup}, we have the following:
  For each $0<i<N$, there exist compatible hyperfiltered complexes of connections and
  a filtered morphism
  \[
    \wt\Omega_X^\kdot(\log(\Sigma+H))\otimes\sL^{-i}\to
    \Om_X^\kdot(\log(\Sigma+H))\otimes\sL^{-i},
  \]
  where the induced morphism on the associated graded objects agrees with the
  morphism induced by taking the $0^\text{th}$ cohomology sheaf of the complex on the
  right.
\end{prop}

\begin{proof}
  This follows from \autoref{lem:finite-cover} and \autoref{cor:finite-cover} applied
  to the pushforward of the filtered morphism
  $\wt\Omega_Y^\kdot(\log(\Sigma'+H'))\to \Om_Y^\kdot(\log(\Sigma'+H'))$,
  cf.~\cite[3.16]{EV92}. 
\end{proof}

\begin{rem}\label{rem:connection-for-f}
  \autoref{cor:connection-for-f} implies that there exist compatible hyperfiltered
  complexes of connections, $\tf_X^p\left(\log(\Sigma+H)\right)\otimes\sL^{-i}$,
  $\uf_X^p\left(\log(\Sigma+H)\right)\otimes\sL^{-i}$,
  $\tf^X_p\left(\log(\Sigma+H)\right)\otimes\sL^{-i}$, and
  $\uf^X_p\left(\log(\Sigma+H)\right)\otimes\sL^{-i}$ for each $p\in\bZ$ and $0<j<N$
  with the expected hyperfiltered morphisms among them.
\end{rem}

\begin{cor}\label{cor:finite-cover-filt} 
  For each $p$,
  \begin{align*}
    \eta_*\tf^Y_p\left(\log(\Sigma'+H')\right)
    &\simeq \oplus_{i=1}^{N-1}\tf^X_p\left(\log(\Sigma+H)\right)\otimes\sL^{-i}, \\  
    \eta_*\uf^Y_p\left(\log(\Sigma'+H')\right)
    &\simeq \oplus_{i=1}^{N-1}\uf^X_p\left(\log(\Sigma+H)\right)
      \otimes \sL^{-i}\\
    \eta_*\tf^Y_p(\log \Sigma')
    &\simeq \tf^X_p(\log \Sigma)\oplus\left(\oplus_{i=1}^{N-1}\tf^X_p\left(\log
      (\Sigma+H)\right)  
      \otimes\sL^{-i}\right),\\
    \eta_*\uf^Y_p(\log \Sigma')
    &\simeq \uf^X_p(\log \Sigma)\oplus\left(\oplus_{i=1}^{N-1}\uf^X_p\left(\log
      (\Sigma+H)\right)\otimes        \sL^{-i}\right)\\
    \eta_*\tf^Y_p &\simeq \tf^X_p\oplus\left(\oplus_{i=1}^{N-1}\tf^X_p(\log H)
                    \otimes\sL^{-i}\right), \text{and}\\ 
    \eta_*\uf^Y_p &\simeq \uf^X_p\oplus\left(\oplus_{i=1}^{N-1}\uf^X_p(\log H)
                    \otimes \sL^{-i}\right). 
  \end{align*}
\end{cor}

\begin{proof}
  The objects in the statement exist by \autoref{cor:connection-for-f} and
  \autoref{rem:connection-for-f}.
  We use induction on $p$ and prove the last statement. The proofs of the other
  statements are essentially identical. For $p=0$, $\uf_0^Y=\Om_Y^0$ and hence the
  statement follows from \autoref{lem:finite-cover}. Suppose we know that the
  statement holds for $p-1$ and consider the commutative diagram of distinguished
  triangles:
  \begin{equation}
    \label{eq:80}
    \begin{aligned}
      \xymatrix@R1.25em{%
        \eta_*\Om_Y^p[-p] \ar[r]^-\alpha \ar[d] & \Om_X^p \oplus
        \left(\oplus_{i=1}^{N-1}\Om_X^p(\log H) \otimes
          \sL^{-i}\right)[-p] \ar[d] \\
        \eta_*\uf_{p+1}^Y \ar[r]^-\beta \ar[d] &
        \uf^X_{p+1}\oplus\left(\oplus_{i=1}^{N-1}\uf^X_{p+1}(\log H) \otimes
          \sL^{-i}\right) \ar[d] \\
        \eta_*\uf_{p}^Y \ar[r]^-\gamma \ar[d]^-{+1} &
        \uf^X_{p}\oplus\left(\oplus_{i=1}^{N-1}\uf^X_{p}(\log H) \otimes
          \sL^{-i}\right) \ar[d]^-{+1} \\ & .}
    \end{aligned}
  \end{equation}
  Then $\alpha$ is an isomorphism by \autoref{lem:finite-cover} and $\gamma$ is an
  isomorphism by the inductive hypothesis. It follows that then $\beta$ is also an
  isomorphism, which is the desired statement.
\end{proof}

\begin{cor}\label{cor:finite-cover-upper-filt}
  For each $p$,
  \begin{align*}
    \eta_*\tf_Y^p\left(\log(\Sigma'+H')\right)
    &\simeq \oplus_{i=1}^{N-1}\tf_X^p\left(\log(\Sigma+H)\right)\otimes\sL^{-i}, \\  
    \eta_*\uf_Y^p\left(\log(\Sigma'+H')\right)
    &\simeq \oplus_{i=1}^{N-1}\uf_X^p\left(\log(\Sigma+H)\right)
      \otimes \sL^{-i}\\
    \eta_*\tf_Y^p(\log \Sigma')
    &\simeq \tf_X^p(\log \Sigma)\oplus\left(\oplus_{i=1}^{N-1}\tf_X^p\left(\log
      (\Sigma+H)\right)  \otimes\sL^{-i}\right),\\
    \eta_*\uf_Y^p(\log \Sigma')
    &\simeq \uf_X^p(\log \Sigma)\oplus\left(\oplus_{i=1}^{N-1}\uf_X^p\left(\log
      (\Sigma+H)\right)       \otimes \sL^{-i}\right)\\
    \eta_*\tf_Y^p &\simeq \tf_X^p\oplus\left(\oplus_{i=1}^{N-1}\tf_X^p(\log H)
                    \otimes\sL^{-i}\right), \text{and}\\ 
    \eta_*\uf_Y^p &\simeq \uf_X^p\oplus\left(\oplus_{i=1}^{N-1}\uf_X^p(\log H)
                    \otimes \sL^{-i}\right). 
  \end{align*}
\end{cor}

\begin{proof}
  Essentially the same proof as above works, except that we need to start with
  $p=\dim X$ and use descending induction.
\end{proof}


\section{The filtered Deligne-\DB complex of a pair and
  cohomology}\label{sec:filtered-deligne-db} 
\noindent
\noin We will need the following lemma. It is probably known to experts, but I do not
know an available reference.  In the smooth case it follows from
\cite[Prop~7.5]{MR2451566} and otherwise from the more general
\autoref{prop:coh-of-filt-surj}. It is stated for ease of reference.

\begin{lem}\label{lem:coh-of-filt-surj}
  Let $X$ be a proper variety, $\imath:\Sigma\into X$ a closed subset,
  $\jmath:V\leteq X\setminus\Sigma\into X$ the corresponding open embedding,
  $q\in\bN$, and let $F^pH_{\rm c}^q (V,\bC)$ for $0\leq p\leq q$ denote the
  filtration on $H_{\rm c}^q(V,\bC)$ coming from the degeneration of the
  Hodge-to-de~Rham spectral sequence of a pair
  (cf.~\autoref{thm:hodge}\autoref{item:9}), i.e., the filtration for which
  \[
    F^pH_{\rm c}^q(V,\bC)/F^{p+1}H_{\rm c}^q(V,\bC)\simeq
    \bH^{q-p}\left(X,\Om_{X,\Sigma}^p\right).
  \]
  Then there exists a natural isomorphism
  \[
    \bH^{q}\left(X,\uf^p_{X,\Sigma}\right)\overset\simeq\longrightarrow
    F^p\bH^q\left(X,\Om^\kdot_{X,\Sigma}\right)\simeq F^pH^q_{\rm c}(V,\bC).
  \]
\end{lem}

\begin{proof}
  This follows directly from \autoref{prop:coh-of-filt-surj} and Serre's GAGA
  principle \cite{MR0082175}.
\end{proof}

\begin{rem}
  Note that in \autoref{lem:coh-of-filt-surj} and in all following statements we may
  choose $\Sigma=\emptyset$ in which case $X=V$. In fact, in this paper we will only
  use that case.
\end{rem}

\begin{cor}\label{cor:coh-of-filt-to-Om}
  Using the notation and assumptions of \autoref{lem:coh-of-filt-surj}, the natural
  morphism induced by the pair analogue of the morphism in \autoref{eq:9} is
  surjective:
  \[
    \bH^{q}\left(X,\uf^p_{X,\Sigma}\right) \onto
    \bH^{q-p}\left(X,\Om^p_{X,\Sigma}\right).
  \]
\end{cor}

\begin{proof}
  This is a direct consequence of \autoref{lem:coh-of-filt-surj} and
  \autoref{cor:ss-surjectivity}.
\end{proof}

\noin
And we also obtain the following important consequence.

\begin{cor}\label{cor:surjective-cohomology}
  Using the notation and assumptions of \autoref{lem:coh-of-filt-surj}, we also have
  another natural morphism which is surjective: For every $q$ and $p$, the following
  morphism is surjective:
  \[
    H_{\rm c}^q(V,\bC)\simeq\bH^{q}\left(X,\Om^\kdot_{X,\Sigma}\right) \onto
    \bH^{q}\left(X,\uf_p^{X,\Sigma}\right).
  \]
\end{cor}

\begin{proof}
  By \autoref{lem:coh-of-filt-surj} for every $q$ and $p$ the following natural
  morphism is an injection:
  \[
    \bH^{q}\left(X,\uf^p_{X,\Sigma}\right) \into \bH^{q}\left(X,\Om^\kdot_{X,\Sigma}\right).
  \]
  Then the long exact sequence of cohomology associated to the pair analogue of the
  distinguished triangle in
  \autoref{eq:10} 
  implies the statement.
\end{proof}

\begin{cor}\label{cor:surj-coh-2}
  Under the same assumptions the following morphisms are also surjective:
  \begin{align*}
    & \bH^q\left(X,\f_p^{X,\Sigma}\right) \onto
      \bH^{q}\left(X,\uf_p^{X,\Sigma}\right), &  
    & \bH^q\left(X,\tf_p^{X,\Sigma}\right) \onto
      \bH^{q}\left(X,\uf_p^{X,\Sigma}\right) 
  \end{align*}
\end{cor}

\begin{proof}
  By \autoref{cor:surjective-cohomology} the morphism $\sigma$ the following
  commutative diagram induces a surjective morphism on cohomology which implies both
  statements. 
  \begin{equation}
    \label{eq:34}
    \begin{aligned}
      \xymatrix@R2em{%
        \jmath_!\bC_V \ar[r] \ar@/^.5em/[rrrd]_(.175)\sigma
        |!{[r];[rd]}\hole|!{[rr];[rrd]}\hole & \Omega^\kdot_{X,\Sigma} \ar[d] \ar[r]
        & \wt\Omega^\kdot_{X,\Sigma} \ar[d] \ar[r] & \Om_{X,\Sigma}^\kdot
        \ar[d] \\
        & \f_p^{X,\Sigma} \ar[r] & \tf_p^{X,\Sigma} \ar[r] & \uf_p^{X,\Sigma} . }
    \end{aligned} \qedhere
  \end{equation}
\end{proof}

\noin The next result is an application of the above, combined with results from
\autoref{sec:cyclic-covers}. For simplicity, here we assume that $\Sigma=\emptyset$.

\begin{cor}\label{cor:H-surjectivity-on-log-and-line-bundle}
  Using \autoref{setup}, further assume that $X$ is a connected proper 
  variety. Then for each $q$ and each $0< i< n$, the following morphism is
  surjective.
  \begin{equation*}
    \bH^q(X,\tf^X_p\left(\log H\right)\otimes\sL^{-i}) \onto
    \bH^q(X,\uf^X_p\left(\log H\right)\otimes\sL^{-i}).
  \end{equation*}
\end{cor}

\begin{proof}
  By \autoref{cor:surj-coh-2}, there is a surjective morphism,
  $\bH^q(Y,\tf^Y_p) \onto \bH^{j}(Y,\uf^Y_{p}).$
  Then, by \autoref{cor:finite-cover-filt},
  \begin{align*}
    \bH^q(Y,\tf^Y_p) &\simeq \bH^q(X,\eta_*\tf^Y_p)\simeq
                       \bH^q\left(X,\tf^X_p\right)\oplus\left(\oplus_{i=1}^{n-1}
                       \bH^q\left(X,\tf^X_p(\log H)
                       \otimes\sL^{-i}\right)\right), \text{and}\\
    \bH^q(Y,\uf^Y_p) &\simeq \bH^q(X,\eta_*\uf^Y_p)\simeq \bH^q\left(X,
                       \uf^X_p\right)\oplus\left(\oplus_{i=1}^{n-1}
                       \bH^q\left(X,\uf^X_p(\log H)\otimes
                       \sL^{-i}\right)\right). \qedhere 
  \end{align*}
\end{proof}






\section{L'\'eminence grise}\label{sec:eminence-grise}
\noindent
\renewcommand\wegpx{\ensuremath{\bsfg\vphantom{\sfG}^{X}_{p}(\sL^{-j})}}
\renewcommand\hegpxz{\ensuremath{{\bsfg}\vphantom{\sfG}^{X,H}_{p}(\sL^{-j})}}
\newcommand{\ii}{q}
\renewcommand\egpx{\ensuremath{\sfG_{X,Z,\qc}^{p}}}
\newcommand\egpxc{\ensuremath{\sfG_{X,Z}^{p}}}
\newcommand\qpx{\ensuremath{\sfQ\vphantom{\sfQ}_{X}^{p}(\sL^{-j})}}
\newcommand\HZ{\ensuremath{\sH_Z}}


\begin{defini}
  Let $X$ be a variety.  The \emph{$p^\text{th}$-DB defect} of $X$ is the mapping
  cone $\Om_X^{p,+}\leteq\cone\left[\wt\Omega_X^p\to\Om_X^p\right]$
\end{defini}

\begin{notation}
  Let $X$ be a variety and $Z\subseteq X$ a closed subset. Then $\HZ$ will denote the
  functor of the \emph{subsheaf of sections supported on $Z$}, $\myR\HZ$ its right
  derived functor, and $\eta_Z:\myR\HZ\to\Id$ the corresponding natural
  transformation.
\end{notation}

\begin{thm}\label{thm:coherent-G}
  Let $X$ be a variety, $Z\subseteq X$ a closed subset and $U\leteq X\setminus
  Z$. Then
  \begin{enumerate}
  \item\label{item:25} there exists an object $\egpx\in\Ob D_\qc^b(X)$ that fits into
    the following commutative diagram in $D_\qc^b(X)$ where the rows are \dts, and
    $\nu
    :\myR\HZ\Om_X^{p,+}\to\Om_X^{p,+}$ is the natural morphism induced by
    $\eta_Z$. (Naively, $\egpx$ is cobbled together from $\wt\Omega_X^p$ on $U$ and
    $\Om_X^p$ on $Z$): \vskip-1em
    \begin{equation}
      \label{eq:79}
      \begin{aligned}
        \xymatrix{%
          \wt\Omega_X^p\ar[d]_\id\ar[r] & \egpx\ar[d]^\zeta \ar[r] &
          \myR\HZ\Om_X^{p,+}\ar[r]^-{+1} \ar[d]^\nu & \\
          \wt\Omega_X^p\ar[r] & \Om_X^p\ar[r] & \Om_X^{p,+}\ar[r]^-{+1} &. }
      \end{aligned}
    \end{equation}
  \item\label{item:26} Let $\sfG\in\Ob D_\filt(X)$ and $\sM$ a line bundle on $X$.
    Further let $\alpha:\wt\Omega_X^p\otimes\sM\to\sfG$ and
    $\beta:\sfG\to\Om_X^p\otimes\sM$ be morphisms such that $\beta\circ\alpha$ is
    equal to the natural morphism $\wt\Omega_X^p\otimes\sM\to\Om_X^p\otimes\sM$. If
    $\alpha\resto U$ is an isomorphism, then $\beta$ factors through
    $\zeta\otimes\sM$:\vskip-1.5em
    \begin{equation}
      \label{eq:84}
      \xymatrix@C4em{%
        \sfG\ar[r]_-\sigma\ar@/^1em/[rr]^-\beta &
        \egpx\otimes\sM\ar[r]_-{\zeta\otimes\sM} & \Om_X^p\otimes\sM 
      }%
    \end{equation}
  \end{enumerate}
\end{thm}
\begin{proof}
  The \dt
  $\xymatrix@C1.35em{\wt\Omega_X^p\ar[r] & \Om_X^p\ar[r] & \Om_X^{p,+}\ar[r]^-{+1}
    &}$ induces a natural morphism $\Om_X^{p,+}\to\wt\Omega_X^p[1]$ and the
  composition of this morphism with the natural morphism
  $\nu:\myR\HZ\Om_X^{p,+}\!\!\!\to\!\Om_X^{p,+}$ gives a morphism
  $\myR\HZ\Om_X^{p,+}\!\!\to\wt\Omega_X^{p}[1]$. Let $\egpx$ be the mapping cone of of
  this latter morphism twisted by $-1$:
  \[
    \egpx\leteq\cone\left[\myR\HZ\Om_X^{p,+}[-1]\longrightarrow
      \wt\Omega_X^{p}\right].
  \]
  Then $\egpx$ fits into a \dt as in the diagram in \autoref{eq:79} and the existing
  outside morphisms induce $\zeta$ in the middle. Therefore this object satisfies the
  required property in \autoref{item:25}.  Further note that by the construction, the
  morphism $\wt\Omega_U^p\to\egpx\resto U$ restricted to $U$ is an isomorphism.



  Let $\sfQ\leteq\cone\left[\alpha:\wt\Omega_X^p\otimes\sM\to\sfG\right]$. Then there
  is a commutative diagram of \dts,
  \[
    \xymatrix@R1.75em{%
      \wt\Omega_X^p\otimes\sM\ar[d]_\id\ar[r]^-\alpha & \sfG\ar[d]^\beta \ar[r] &
      \sfQ\ar[r]^-{+1} \ar[d]^\gamma & \\
      \wt\Omega_X^p\otimes\sM\ar[r] & \Om_X^p\otimes\sM\ar[r] &
      \Om_X^{p,+}\otimes\sM\ar[r]^-{+1} &.  }
  \]
  Here $\gamma$ is induced by $\id_{\wt\Omega_X^p}$ and $\beta$, so turning around
  the triangle, the morphism $\sfG\to\Om_X^p$ induced by $\id_{\wt\Omega_X^p}$ and
  $\gamma$ is equivalent to $\beta$. Replace $\beta$ with this morphism and note that
  if \autoref{item:26} holds for this morphism, then it also holds for the original
  $\beta$.
  
  The assumption that $\alpha\resto U$ is an isomorphism implies that
  $\sfQ\resto U\simeq 0$ and hence 
  $\gamma:\sfQ\to\Om_X^{p,+}\otimes\sM$ factors through the natural morphism,
  $\myR\HZ\Om_X^{p,+}\to\Om_X^{p,+}$. Therefore we have the following diagram:

  \begin{equation}
    \label{eq:85}    
    \begin{aligned}      
      \xymatrix@C4em@R2em{%
        \wt\Omega_X^p\otimes\sM\ar[d]_\id\ar[r]^-\alpha & \sfG\ar@{..>}[d]^\sigma \ar[r]
        \ar@/_4.65em/[dd]_(.3)\beta &
        \sfQ\ar[r]^-{+1} \ar[d] \ar@/_5.35em/[dd]_(.3)\gamma & \\
        \wt\Omega_X^p\otimes\sM\ar[d]_\id\ar[r]|!{[u];[rd]}\hole &
        \egpx\otimes\sM\ar[d]^{\zeta\otimes\sM} 
        \ar[r]|!{[u];[rd]}\hole  &
        \myR\HZ\Om_X^{p,+}\otimes\sM\hskip-4em\phantom{\qquad\qquad}
        \ar[r]^-{+1} \ar[d]^{\nu\otimes\sM} &  \\
        \wt\Omega_X^p\otimes\sM\ar[r] & \Om_X^p\otimes\sM\ar[r] &
        \Om_X^{p,+}\otimes\sM\ar[r]^-{+1} &.  }
    \end{aligned}
  \end{equation}
  As $\beta$ is induced by $\id_{\wt\Omega_X^p}$ and $\gamma$, and $\zeta$ is induced
  by $\id_{\wt\Omega_X^p}$ and $\nu$, the fact that $\gamma$ factors through
  $\nu\otimes\sM$ implies that $\beta$ factors through $\zeta\otimes\sM$.
\end{proof}

\begin{lem}\label{lem:coh-subsheaf}
  Let $X$ be a noetherian scheme, $\alpha:\sA\to\sB$ a morphism of quasi-coherent
  sheaves on $X$, and $\sG\subseteq\im\alpha$ a coherent subsheaf. Then there exists
  a coherent subsheaf $\sF\subseteq\sA$ such that $\im(\alpha\resto\sF)=\sG$.
\end{lem}

\begin{proof}
  Replacing $\sA$ with $\alpha^{-1}\sG$ and $\sB$ with $\sG$ we may assume that $\sB$
  is coherent and $\alpha$ is surjective. 
  Let $X=\cup_{i=1}^r U_i$ be a finite affine cover, and choose generators
  $s_{ij}\in\sG(U_i)$ for $j=1,\dots,r'$. For each $(i,j)$ let $t_{ij}\in\sA(U_i)$ be
  an element such that $\alpha(U_i)(t_{ij})=s_{ij}$. Then $\sF\subseteq\sA$, the
  subsheaf generated by $t_{ij}$ for $i=1,\dots,r$, $j=1,\dots,r'$ satisfies the
  required properties.
\end{proof}

\begin{lem}\label{lem:coherent-image}
  Let $X$ be a noetherian scheme and $\phi:\sfM\to\sfN$ a morphism in $D_\qc(X)$
  with $\sfM\in\Ob D^-_\coh(X)$. Then \tes a $\sfK\in\ob D^-_\coh(X)$ (and if
  $\sfM\in\Ob D^b_\coh(X)$, then $\sfK\in\ob D^b_\coh(X)$) such that
  \begin{enumerate}
  \item\label{item:5} $\phi$ factors through $\sfK$, i.e., \te morphisms $\xymatrix{%
      \sfM\ar[r]^\psi&\sfK\ar[r]^\kappa&\sfN,}$ such that $\phi=\kappa\circ\psi$, and
  \item\label{item:6} for each $i$,
    $\ker\left[h^i(\sfM)\to h^i(\sfN)\right] = \ker\left[h^i(\sfM)\to
      h^i(\sfK)\right]$.
  \end{enumerate}
\end{lem}

\begin{proof}
  First choose a representative of $\sfM$ such that $\sfM^i$ is coherent for each $i$
  (cf.~\cite[\href{https://stacks.math.columbia.edu/tag/0FDA}{Tag
    0FDA}]{stacks-project}) and then choose a representative of $\sfN$ such that
  $\phi$ is an actual morphism between complexes.
  
  Let
  $\sfA^{i+1}\leteq\left((\phi^{i+1})^{-1}\im d^{i}_\sfN\right)\cap\ker
  d^{i+1}_\sfM\subseteq\sfM^{i+1}$, and let
  $\sfB^{i}\subseteq(d^{i}_\sfN)^{-1}\im\phi^{i+1}\subseteq\sfN^i$ be a coherent
  subsheaf that maps surjectively on to $\im\phi^{i+1}\cap\im d^i_\sfN$ (let
  $\sfB^i\leteq 0$ if $\sfM^i=0$), cf.~\autoref{lem:coh-subsheaf}. Let $\sfC^{i}$ be
  the fibered direct sum of $\sfA^{i+1}$ and $\sfB^i$ over $\sfN^{i+1}$;
  $\sfC^{i}\leteq\sfA^{i+1}\oplus_{\sfN^{i+1}}\sfB^{i}=
  \{(a,b)\in\sfA^{i+1}\oplus\sfB^i \skvert \phi^{i+1}(a)=d^{i}_\sfN(b)\}$. Then there
  is a commutative diagram with appropriate morphisms $\alpha^i$ and $\beta^i$,
  \[
    \xymatrix@C3em@R2em{%
      &\sfC^{i} \ar[d]_{\beta^{i}} \ar[r]^-{\alpha^i} & \sfA^{i+1}
      \ar[d]<-.5em>^{\phi^{i+1}\resto{\sfA^{i+1}}} &\hskip-3.6em\subseteq\sfM^{i+1}
      \\
      \sfN^i\supseteq\hskip-4.7em&\sfB^{i} \ar[r]_-{d^{i}_\sfN} & \sfN^{i+1}. }
  \]
  Next let $\sfK^{i}\leteq \sfM^i\oplus \sfC^{i}$ and $d^i_\sfK:\sfK^i\to\sfK^{i+1}$
  defined by $\left(d^i_\sfK\right)\resto{\sfM^i}=d^i_\sfM$ and
  $\left(d^i_\sfK\right)\resto{\sfC^i}=\alpha^i$. Notice that with this definition
  $\im d^i_\sfK\subseteq\ker d^{i+1}_\sfM\subseteq \sfM^{i+1}$ and hence
  $d^{i+1}_\sfK\circ d^i_\sfK=0$, so $\sfK$ is indeed a complex. By the construction
  $\sfA^{i+1}$, $\sfB^i$, and hence $\sfC^i$ and $\sfK^i$ are all coherent. 
  Next define $\psi\leteq(\id_\sfM,0):\sfM\to\sfK$ and
  $\kappa\leteq(\phi,\beta):\sfK\to\sfN$. This shows that $\sfK$, $\psi$, and
  $\kappa$ satisfy \autoref{item:5}.

  Note that the containment
  $\ker\left[h^i(\sfM)\to h^i(\sfK)\right]\subseteq\ker\left[h^i(\sfM)\to
    h^i(\sfN)\right]$ is trivial, so we only need to prove the other direction.  Let
  $U\subseteq X$ be an open affine subset,
  $\ol x\in\ker\left[h^i(\sfM)\to h^i(\sfN)\right](U)$ and $x\in\ker d^{i}_\sfM(U)$
  that maps to $\ol x$. Then $\phi^i(x)\in\im d^{i-1}_\sfN(U)$, so there is a
  $y\in\sfB^{i-1}(U)$ such that $d^{i-1}_\sfN(y)=\phi^i(x)$. Then
  $(x,y)\in\sfC^{i-1}(U)$, and $z\leteq(0,(x,y))\in\sfK^{i-1}(U)$ is such that
  $d^{i-1}_\sfK(z)=x$, which implies that
  $\ol x\in\ker\left[h^i(\sfM)\to h^i(\sfK)\right](U)$, and hence \autoref{item:6}
  follows.
\end{proof}

\begin{cor}\label{cor:leminence-grise-coherent}
  Using the notation from \autoref{thm:coherent-G},
  %
  there exists an object $\egpxc\in\Ob D_\coh(X)$ such that the morphism $\zeta$
  factors through $\egpxc$, i.e., there are morphism $\xymatrix{%
    \egpx\ar[r] & \egpxc\ar[r] & \Om_X^p}$ and for the induced morphisms $\xymatrix{%
    \bD_X(\Om_X^p)\ar[r] & \bD_X(\egpxc)\ar[r] & \bD_X(\egpx)},$ we have that for
  each $i\in\bZ$,
  \[
    \ker\left[h^i(\bD_X(\Om_X^p))\to h^i(\bD_X(\egpx))\right] =
    \ker\left[h^i(\bD_X(\Om_X^p))\to h^i(\bD_X(\egpxc))\right].
  \]
\end{cor}

\begin{proof}
  $\sfM\leteq\bD_X(\Om_X^p)$ and $\sfN\leteq\bD_X(\egpx)$, and apply
  \autoref{lem:coherent-image} to $\bD_X(\zeta)$. Let $\egpxc\leteq\bD_X(\sfK)$,
  where $\sfK$ is the complex provided by \autoref{lem:coherent-image}. Applying
  \cite[\href{https://stacks.math.columbia.edu/tag/0A8U}{Tag 0A8U}]{stacks-project}
  with $L=M=\omega_X^\kdot$, we obtain that $\zeta$ factors as the compostion of
  morphisms $\egpx\to\bD_X(\bD_X(\egpx))\to\egpxc\to\Om_X^p$. This establishes the
  first requirement. The second requirement follows from Grothendieck duality, the
  construction, and \autoref{lem:coherent-image}.
\end{proof}


\begin{thm}\label{cor:surj-for-Om}
  Let $\sL$ be a semi-ample line bundle on a proper variety $X$, $N\in\bN$ such that
  $\sL^N$ is generated by global sections, $s\in H^0(X,\sL^N)$ a general section of
  $\sL^N$, and $H\leteq (s=0)$.
  Then for each $p\in\bN$ and $0<j<N$ there exist objects
  $\wegpx,\hegpxz\in\Ob D_{\rm filt, coh}^b(X)$ and morphisms
  $\gamma_X^{p,j}:\wegpx\to\hegpxz$,
  $\wt\nu_X^{p,j}:\wt\Omega_X^p\otimes\sL^{-j}\to\wegpx$,
  $\oegnu{p,j}X:\wegpx\to\Om_X^p\otimes\sL^{-j}$,
  $\wt\nu_{X,H}^{p,j}:\wt\Omega_X^p(\log H)\otimes\sL^{-j}\to\hegpxz$ and
  $\oegnu{p,j}{X,H}:\hegpxz\to\Om_X^p(\log H)\otimes\sL^{-j}$, such that:
  \begin{enumerate}
  \item\label{item:35} The compositions of the morphisms above,
    $\egnu{p,j}X\leteq\oegnu{p,j}X\circ\wt\nu_X^{p,j}:
    \wt\Omega_X^p\otimes\sL^{-j}\to\Om_X^p\otimes\sL^{-j}$ and
    $\egnu{p,j}{X,H}\leteq\oegnu{p,j}{X,H}\circ\wt\nu_{X,H}^{p,j}: \wt\Omega_X^p(\log
    H)\otimes\sL^{-j}\to\Om_X^p(\log H)\otimes\sL^{-j}$, agree with the natural
    morphisms induced by taking the $0^\text{th}$ cohomology sheaves of the
    respective complexes, cf.~\autoref{prop:morphs-of-cxs}.
  \item\label{item:34} The above morphisms fit into a commutative diagram,
    \[
      \xymatrix@C4em{
        \wt\Omega_X^p\otimes\sL^{-j}\ar[d] \ar[r]^-{\wt\nu_X^{p,j}} & \wegpx
        \ar[d]^{\gamma_X^{p,j}}
        \ar[r]^-{\oegnu{p,j}X} & \Om_X^p\otimes\sL^{-j}\ar[d]\\
        \wt\Omega_X^p(\log H)\otimes\sL^{-j}\ar[r]_-{\wt\nu_{X,H}^{p,j}} & \hegpxz
        \ar[r]_-{\oegnu{p,j}{X,H}} & \Om_X^p(\log H)\otimes\sL^{-j}, }
    \]
    where the unmarked vertical morphisms are induced by the natural
    morphisms 
    constructed in \autoref{eq:20}.
  \item\label{item:37} The restriction $\gamma_X^{p,j}\resto{X\setminus H}$ is an
    isomorphism, and
    $\wt\nu_{X,H}^{p,j}\resto{X\setminus H}\simeq\wt\nu_X^{p,j}\resto{X\setminus H}$, and
    $\oegnu{p,j}{X,H}\resto{X\setminus H} \simeq \oegnu{p,j}X\resto{X\setminus H}$.
  \item\label{item:30} For each $0<j<N$, and $q$, 
    $\oegnu{p,j}X$ and $\oegnu{p,j}{X,H}$ induce surjective maps on hypercohomology:
    \begin{enumerate}
    \item\label{item:33}
      $\bH^q(X,\wegpx) \onto \bH^q(X,\Om_X^p\otimes\sL^{-j})$, and 
    \item\label{item:32}
      $\bH^q(X,\hegpxz) \onto \bH^q(X,\Om_X^p(\log H)\otimes\sL^{-j})$.
    \end{enumerate}
  \item\label{item:31} If $U\subseteq X$ is an open subset that has \premmdb \sings
    for some $m\in\bN$, then
    ${\wt\nu_{X,H}^{p,j}}\resto{U}:\wt\Omega_X^p(\log
    H)\otimes\sL^{-j}\resto{U}\overset\simeq\longrightarrow\hegpxz\resto{U}$, and
    ${\wt\nu_{X}^{p,j}}\resto {U}:
    \wt\Omega_X^p\otimes\sL^{-j}\resto{U}\overset\simeq\longrightarrow\wegpx\resto{U}$
    are isomorphisms for each $p\leq m$ and $0<j<N$.
  \item\label{item:27} If $U\subseteq X$ is an open subset that has \premmdb \sings
    for some $m\in\bN$, let $Z\leteq X\setminus U$, and $\egpx$ the object defined in
    \autoref{thm:coherent-G}\autoref{item:25}.  Then $\oegnu{p,j}X$ factors through
    $\zeta\otimes\sL^{-j}$.
  \end{enumerate}
\end{thm}


\begin{proof}%
  Consider the following diagram, cf.~\autoref{eq:11},\autoref{eq:7tf},
  \autoref{cor:finite-cover-filt}, \autoref{eq:80}:
  \begin{equation}
    \label{eq:67}
    \begin{aligned}
      \xymatrix{
        \wt\Omega_X^p(\log H)\otimes\sL^{-j}[-p] \ar[r]\ar[d]_{\egnu{p,j}{X,H}[-p]} &
        \tf^X_{p}(\log H)\otimes\sL^{-j} \ar[r]\ar[d]\ar@/^1em/[rd]_\alpha &
        \tf^X_{p-1}(\log H)\otimes\sL^{-j} \ar[d] \ar[r]^-{+1} &
        \\
        \Om_X^p(\log H)\otimes\sL^{-j}[-p]\ar[r] & \uf^X_{p}(\log
        H)\otimes\sL^{-j}\ar[r] & \uf^X_{p-1}(\log H)\otimes\sL^{-j} \ar[r]^-{+1} &}
    \end{aligned}
  \end{equation}
  Using the morphism $\alpha$, 
  let
  $\hegpxz\leteq\cone\big[\tf^X_{p}(\log
  H)\otimes\sL^{-j}\overset\alpha\longrightarrow\uf^{X}_{p-1}(\log
  H)\otimes\sL^{-j}\big][p-1].$ This induces morphisms
  $\wt\nu_{X,H}^{p,j}:\wt\Omega_X^p(\log H)\otimes\sL^{-j}\to\hegpxz$ and
  $\oegnu{p,j}X:\hegpxz\to\Om_X^p(\log H)\otimes\sL^{-j}$, and the diagram in
  \autoref{eq:67} can be extended with a middle row:
  \begin{equation}
    \label{eq:69}
    \begin{aligned}
      \xymatrix@R1.5em{%
        \wt\Omega_X^p(\log H)\otimes\sL^{-j} [-p]
        \ar[r]\ar[d]^{\wt\nu_{X,H}^{p,j}[-p]} \ar@/_4em/[dd]_{\egnu{p,j}{X,H}[-p]\!}
        & \tf^X_{p}(\log H)\otimes\sL^{-j} \ar[r]\ar[d]^{\id} & \tf^X_{p-1}(\log
        H)\otimes\sL^{-j} \ar[d] \ar[r]^-{+1} &
        \\
        \hegpxz[-p] \ar[r]\ar[d]^{\oegnu{p,j}{X,H}[-p]} & \tf^X_{p}(\log
        H)\otimes\sL^{-j} \ar[r]\ar[d] & \uf^X_{p-1}(\log H)\otimes\sL^{-j}
        \ar[d]^{\id}
        \ar[r]^-{+1} & \\
        \Om_X^p(\log H)\otimes\sL^{-j} [-p]\ar[r] & \uf^X_{p}(\log
        H)\otimes\sL^{-j}\ar[r] & \uf^X_{p-1}(\log H)\otimes\sL^{-j} \ar[r]^-{+1} &
        .}
    \end{aligned}
  \end{equation}
  Next, compose $\oegnu{p,j}{X,H}$ from \autoref{eq:69} with
  $\psi:\Om^p_X\left(\log H\right)\otimes\sL^{-j}\to \Om^{p-1}_H\otimes\sL^{-j}$ from
  \autoref{eq:20}
  and let
  \[
    \wegpx\leteq\cone\big[\hegpxz\longrightarrow\Om^{p-1}_H\otimes\sL^{-j}\big][-1].
  \]
  This leads to the following commutative diagram, where the last two rows are \dts:
  \begin{equation}
    \label{eq:78}
    \begin{aligned}
      \xymatrix@R1.5em{%
        \wt\Omega_X^p\otimes\sL^{-j} \ar@{..>}[d]_\exists^{\wt\nu_X^{p,j}} \ar[r]
        \ar@/_.45em/[rrd]|!{[r];[rd]}\hole^(.75){=0} \ar@/_3em/[dd]_{\egnu{p,j}{X}\!}
        & \wt\Omega_X^p(\log H)\otimes\sL^{-j}\ar[d]^(.45){\wt\nu_{X,H}^{p,j}} \ar[r]
        &
        \wt\Omega_H^{p-1}\otimes\sL^{-j} \ar[d] & \\
        \wegpx \ar@{..>}[d]_\exists^{\oegnu{p,j}X} \ar[r]^-{\gamma_X^{p,j}} & \hegpxz
        \ar[d]^{\oegnu{p,j}{X,H}} \ar[r] & \Om^{p-1}_H\otimes\sL^{-j} \ar[d]^\id
        \ar[r]^-{+1} & \\
        \Om^{p}_X\otimes\sL^{-j}\ar[r] & \Om_X^p(\log H)\otimes\sL^{-j} \ar[r] &
        \Om^{p-1}_H\otimes\sL^{-j} \ar[r]^-{+1} & }
    \end{aligned}
  \end{equation}
  The morphisms in the first column exist such that their composition is the
  natural morphism $\wt\Omega_X^p\to\Om_X^p$ twisted by $\sL^{-j}$.  Along with
  \autoref{eq:69} this proves the existence of the claimed morphisms, and
  \autoref{item:35} and \autoref{item:34}.
  Restricting \autoref{eq:78} to $X\setminus H$ shows that the first two columns
  become isomorphic. This proves \autoref{item:37}.

  Taking hypercohomology of the complexes in \autoref{eq:69} leads to the following
  diagram:
  \[
    \resizebox{\hsize}{!}{$\xymatrix@C1em{%
        \bH^{q-1}(X,\uf^X_{p-1}(\log H)\otimes\sL^{-j}) \ar[d]_\id \ar[r] &
        \bH^{q-p}(X,\hegpxz) \ar[r]\ar[d] & \bH^q(X,\tf^X_{p}(\log
        H)\otimes\sL^{-j}) \ar[r]\ar@{->>}[d] & \bH^q(X,\uf^X_{p-1}(\log
        H)\otimes\sL^{-j}) \ar[d]^\id
        \\
        \bH^{q-1}(X,\uf^X_{p-1}(\log H)\otimes\sL^{-j})\ar[r] &
        \bH^{q-p}(X,\Om_X^p(\log H)\otimes\sL^{-j}) \ar[r] & \bH^q(X,\uf^X_{p}(\log
        H)\otimes\sL^{-j})\ar[r] & \bH^q(X,\uf^X_{p-1}(\log H)\otimes\sL^{-j}).}$}
  \]      
  The outside vertical morphisms are isomorphisms, and the third vertical morphism is
  surjective by \autoref{cor:H-surjectivity-on-log-and-line-bundle}.  Then
  \autoref{item:32} follows by the $4$-lemma.
  Taking hypercohomology of \autoref{eq:78} leads to
  \[
    \resizebox{\hsize}{!}{$\xymatrix@R1.25em@C1em{%
      \bH^{q-1}(X,\Om_H^{p-1}\otimes\sL^{-j}) \ar[d]_\id \ar[r] &
      \bH^{q}(X,\wegpx) \ar[r]\ar[d] & \bH^q(X,\hegpxz)
      \ar[r]\ar@{->>}[d] & \bH^q(X,\Om_H^{p-1}\otimes\sL^{-j}) \ar[d]^\id
      \\
      \bH^{q-1}(X,\Om_H^{p-1}\otimes\sL^{-j})\ar[r] &
      \bH^{q}(X,\Om_X^p\otimes\sL^{-j}) \ar[r] & \bH^q(X,\Om_X^{p}(\log
      H)\otimes\sL^{-j})\ar[r] & \bH^q(X,\Om_H^{p-1}\otimes\sL^{-j}).}$}
  \]
  As above, the outside vertical morphisms are isomorphisms, and the third vertical
  morphism is surjective by \autoref{item:32}, so \autoref{item:33} follows by the
  $4$-lemma.

  Next, assume that \tes an open subset $U\subseteq X$ that has \premmdb \sings for
  some $m\in\bN$. Then the morphism
  $\tf^X_{p-1}(\log H)\resto U \overset\simeq\longrightarrow \uf^X_{p-1}(\log
  H)\resto U$ is an isomorphism for $p\leq m$ by
  \autoref{cor:log-sequence-for-wt}\autoref{eq:68}, so the
  the last vertical morphism between the first two rows of 
  \autoref{eq:69} is an isomorphism on $U$ 
  and hence ${\wt\nu_{X,H}^{p,j}}\resto{U}$ is an isomorphism.  Furthermore,
  \autoref{prop:hyperplane-sections-of-premdb} implies that $U\cap H$ also has
  \premmdb \sings and hence the last two vertical morphisms between the first two
  rows of 
  \autoref{eq:78} are isomorphisms on $U$. It follows that then
  ${\wt\nu_{X}^{p,j}}\resto{U}$ is also an isomorphism.  This proves
  \autoref{item:31}.

  Finally, let $\sfG\leteq\wegpx$, $\sM\leteq\sL^{-j}$, $\alpha\leteq\wegnu{p,j}X$
  and $\beta\leteq\oegnu{p,j}X$. Then the assumptions of
  \autoref{thm:coherent-G}\autoref{item:26} are satisfied by \autoref{item:35} and
  \autoref{item:31}. Hence $\oegnu{p,j}X$ factors through
  $\zeta\otimes\sL^{-j}:\egpx\otimes\sL^{-j}\to\Om_X^p\otimes\sL^{-j}$ by
  \autoref{thm:coherent-G}\autoref{item:26}.
  %
  %
  This proves \autoref{item:27} and completes the proof of \autoref{cor:surj-for-Om}.
\end{proof}

\begin{cor}\label{cor:ultimate-surjectivity}
  Using the notation from \autoref{thm:coherent-G} and
  \autoref{cor:leminence-grise-coherent}, further assume that $X$ is proper, and $U$
  has \premmdb \sings for some $m\in\bN$.
  If $\sL$ is a semi-ample line bundle on $X$, then $\zeta\otimes\sL^{-j}$ induces a
  surjective map on hypercohomology, for all $j>0$,
  \[
    \xymatrix{%
      \bH^q(X,\egpx\otimes\sL^{-j})\ar@{->>}@/_1em/[rr]\ar[r] &
      \bH^q(X,\egpxc\otimes\sL^{-j}) \ar[r] & \bH^q(X,\Om_X^p\otimes\sL^{-j})}
  \]
\end{cor}

\begin{proof}
  Let $N\gg j$ 
  and $s\in H^0(X,\sL^N)$ a general section. Then $\oegnu{p,j}X$, constructed in
  \autoref{cor:surj-for-Om}, factors through $\egpx\otimes\sL^{-j}$ by
  \autoref{cor:surj-for-Om}\autoref{item:27}, and hence the statement follows from
  \autoref{cor:surj-for-Om}\autoref{item:33}.
\end{proof}


\label{sec:projective-vs-local}
\noin Next, we will turn this surjectivity into an injectivity for the cohomology
sheaves of the duals using Serre vanishing and global generation as in the proofs of
\cite[3.3]{MR3617778} and \cite[3.2]{KS13}.

\begin{lem}\label{lem:injectivity}
  Let $X$ be a proper variety, $V\subseteq X$ an open set, and
  $\alpha:\sfA\to\sfB$ a morphism in $D_\coh(X)$. If
  \begin{enumerate}
  \item\label{item:3} $\bH^{j}(X,\sfA)\onto\bH^{j}(X,\sfB)$ is surjective for each
    $j\in\bZ$,
  \item\label{item:2} $H^{j}(X,h^a(\bD_X(\sfA)))=H^{j}(X,h^a(\bD_X(\sfB)))=0$ for
    each $a\in\bZ$ and each $j>0$, and
  \item\label{item:4} for some $\ii\in\bZ$,
    $\ker\left[h^{\ii}(\bD_X(\sfB))\to h^{\ii}(\bD_X(\sfA))\right]$ is generated by
    global sections on $V$,
  \end{enumerate}
  then 
  $h^{\ii}(\bD_X(\sfB))\resto V\into h^{\ii}(\bD_X(\sfA))\resto V$ is injective.
\end{lem}

\begin{proof}
  \autoref{lem:grothendieck-duality} and \autoref{item:3} imply that
  $\bH^{j}(X,\bD_X(\sfB))\into \bH^{j}(X,\bD_X(\sfA))$ is injective for each
  $j\in\bZ$, and \autoref{item:2} implies that then
  $\bH^{j}(X,\bD_X(\sfB))\simeq H^0(X,h^{j}(\bD_X(\sfB)))$ and
  $\bH^{j}(X,\bD_X(\sfA))\simeq H^0(X,h^{j}(\bD_X(\sfA)))$ for each $j\in\bZ$.
  Hence
  $H^0\left(X,\ker\left[h^{\ii}(\bD_X(\sfB))\to
      h^{\ii}(\bD_X(\sfA))\right]\right)=0$, so the statement follows from
  \autoref{item:4}.
\end{proof}

\begin{rem}\label{rem:injectivity}
  If
  $\dim\supp\left(\ker\left[h^{\ii}(\bD_X(\sfB))\to h^{\ii}(\bD_X(\sfA))\right]\resto
    V\right)=0$, then the assumption of \autoref{item:4} holds.
\end{rem}

\noin
The next theorem confirms {\cite[Conjecture~G]{PSV24}}.
\begin{thm}
  \label{cor:inj-on-dual-of-Om-for-mdb}
  \label{thm:key-injectivity} 
  Let $U$ be a (not necessarily irreducible) variety of pure dimension $n$.
  If $U$ has \premmdb \sings, then for each $\ii$ and $p\leq m$, the following
  natural morphism is injective:
  \centerline{$\xymatrix{%
      h^{\ii}(\bD_U(\Om_U^p))\, \ar@{^(->}[r] & h^{\ii}(\bD_U(\wt\Omega_U^p)) }$.}
\end{thm}
\begin{proof}
  The statement is local on $U$, so we may assume that it is affine and hence
  quasi-projective.  Let $X\supseteq U$ be a projective closure, %
  $Z\leteq X\setminus U$. Further let $\egpx$ be the object defined in
  \autoref{thm:coherent-G}\autoref{item:25} and $\egpxc$ the object obtained in
  \autoref{cor:leminence-grise-coherent}.  As $\Om_X^p$ is bounded, so are
  $\wt\Omega_X^p$, $\Om_X^{p,+}$, $\egpx$ and $\egpxc$, and hence $\bD_X(\Om_X^p)$
  and $\bD_X(\egpxc)$ have finitely many non-zero cohomology sheaves, i.e., the
  sheaves
  $\sK_\ii\leteq\ker\left[h^{\ii}(\bD_X(\Om_X^p))\to
    h^{\ii}(\bD_X(\egpx))\right]=\ker\left[h^{\ii}(\bD_X(\Om_X^p))\to
    h^{\ii}(\bD_X(\egpxc))\right]$ are zero for all but finitely many $q\in\bZ$.
  Finally, let $\sL$ be a very ample line bundle on $X$ such that
  \begin{enumerate}
  \item[$(\ast)$]\label{item:28} $\sK_\ii\otimes\sL^j$ is generated by global
    sections for each $q\in\bZ$, $j>0$, and
  \item[$(\ast\ast)$]\label{item:29} $H^{b}(X,h^{a}(\bD_X({\Xi^p}))\otimes \sL^j)=0$,
    for each $a\in\bZ$, $p\leq m$, $b>0$, $j>0$, where $\Xi^p$ is either
    $\Om_X^p$, 
    or $\egpxc$.
  \end{enumerate}
  We may apply \autoref{lem:injectivity} with $\sfA\leteq\egpxc\otimes\sL^{-1}$,
  $\sfB\leteq\Om_X^p\otimes\sL^{-1}$, $\alpha\leteq\kappa\otimes\sL^{-1}$, 
  $V=X$ by $(\ast)$, $(\ast\ast)$, and \autoref{cor:ultimate-surjectivity}. It
  follows that $h^{\ii}(\bD_X(\Om_X^p))\resto U\into h^{\ii}(\bD_X(\egpxc))\resto U$
  and hence $h^{\ii}(\bD_X(\Om_X^p))\resto U\into$ $h^{\ii}(\bD_X(\egpx))\resto U$ is
  injective. As $\egpx\resto U\simeq \wt\Omega_U^p$, this proves the desired
  statement.
\end{proof}


%
\section{Applications}\label{sec:applications}
\noindent



\noin
\autoref{thm:key-injectivity} also enables us to compare the depths of $\Om_X^p$ and
$\wt\Omega_X^p$. 
(For the definition of $\depth$ for objects in $D^b_\coh(X)$, we refer to
\cite[p.8]{PSV24}).

\begin{cor}
  If $X$ is a variety with \premmdb \sings, then
  $\depth\Om_X^m\geq\depth\wt\Omega_X^m$.\qed
\end{cor}

\noin
This in turn implies the following by the argument on \cite[p.14]{PSV24}.

\begin{cor}[\protect{\cite[Conjecture~H]{PSV24}}]
  Let $X$ be a variety with only \premmdb \sings and assume that $X$ has
  \premdb \sings away from a closed subset of dimension $r$. Then
  \[
    h^i(\Om_X^m)=0 \qquad \text{for} \qquad 0<i<\depth\wt\Omega_X^m-r-1.
  \]
\end{cor}

\noin
We also obtain a surjectivity statement for local cohomology:

\begin{thm}\label{cor:surj-for-loc-coh-on-dual-of-Om-for-mdb}
  Let $X$ be a variety and $x\in X$ a point. Assume that $X$ has \premmdb \sings near
  $x$.  Then for each $q$ and $p\leq m$ the following natural morphism is surjective:
  \begin{equation}
    \label{eq:73}
    \xymatrix{%
      H^q_x(X,\wt\Omega_X^p) \ar@{->>}[r] & \bH^q_x(X,\Om_X^p). }
  \end{equation}
\end{thm}

\begin{proof}
  The proof is essentially the same as that of \cite[3.4]{MR3617778}.

  Let $E(\kappa(x))$ be the injective hull of the residue field
  $\kappa(x)=\sO_{X,x}/\frm_{X,x}$ at $x$ and apply the faithful and exact functor
  $\Hom_{\sO_{X,x}}(\blank, E(\kappa(x)))$ to the map in \autoref{eq:73}. By Matlis
  duality it is enough to prove that the resulting map is injective.
  %
  By local duality (cf.~\cite[V.6.2]{MR0222093}), this map is the localization at $x$
  of the injective morphism (for $X$) in
  \autoref{thm:key-injectivity}, and hence injective as desired.
\end{proof}

\noin
Next, we obtain a splitting criterion for \premdb \sings.

\begin{thm}\label{thm:splitting-for-Om}
  Let $X$ be a 
  variety and $m\in\bN$.  If the natural morphism $\wt\Omega_X^p\to\Om_X^p$ has a
  left inverse for each $p\leq m$, then $X$ has \premdb \sings.
\end{thm}

\noin \autoref{thm:splitting-for-Om} follows easily from the following by induction
on $m$:

\begin{thm}\label{thm:splitting-for-Om-sub}
  Let $X$ be a 
  variety and $m\in\bN$.  If $X$ has \premmdb \sings, and the natural morphism
  $\wt\Omega_X^m\to\Om_X^m$ has a left inverse, then $X$ has \premdb \sings.
\end{thm}

\begin{proof}
  Consider the morphisms provided by the assumption:
  $ \xymatrix@C4em{%
    \wt\Omega_X^m\ar[r] \ar@/_1.5em/[rr]^-\simeq & \Om_X^m \ar[r] & \wt\Omega_X^m.
  }$%
  Apply $\bD_X$ and take cohomology: 
  $\xymatrix@C4em{%
      h^\ii(\bD_X(\wt\Omega_X^m))\ar[r] \ar@{->>}@/^1.5em/[rr]_-\simeq &
      h^\ii(\bD_X(\Om_X^m)) \ar@{^(->}[r] & h^\ii(\bD_X(\wt\Omega_X^m)).  }$%
    The second morphism is injective by \autoref{cor:inj-on-dual-of-Om-for-mdb} and
    the composition is an isomorphism. 
    It follows that the second morphism is also an isomorphism for each $\ii$, which
    means that $\wt\Omega_X^m\to\Om_X^m$ is a quasi-isomorphism.
\end{proof}

\begin{cor}\label{thm:splitting-for-fX}
  Let $X$ be a 
  variety and $m\in\bN$.  If the natural (hyperfiltered) morphism
  $\tf^X_{m}\to\uf^X_{m}$ has a (hyperfiltered) left inverse, then it is an
  isomorphism (in the derived category). In particular, then $X$ has \premdb \sings.
\end{cor}

\begin{proof}
  The \dts \autoref{eq:11}, \autoref{eq:7tf} form a commutative diagram and show
  that 
  the natural morphism $\wt\Omega_X^p\to\Om_X^p$ has a left inverse for all
  $p\leq m$.  The statement follows by \autoref{thm:splitting-for-Om}.
\end{proof}

\begin{cor}\label{cor:finite-splitting}
  Let $f:Y\to X$ be a proper morphism. If either
  \begin{enumerate}
  \item\label{item:19} $\wt\Omega_X^p\to\myR f_*\Om_Y^p$ has a left inverse for all
    $p\leq m$, or
  \item\label{item:21} $\wt\Omega_X^p\to\myR f_*\wt\Omega_Y^p$ has a left inverse for
    all $p\leq m$, and $Y$ has \premdb \sings,
  \end{enumerate}
  Then $X$ has \premdb \sings.
\end{cor}
\begin{proof}
  Observe that \autoref{item:21} implies \autoref{item:19}, so the latter holds in
  both cases.  The natural morphism $\wt\Omega_X^p\to\myR f_*\Om_Y^p$ factors through
  $\Om_X^p$, so \autoref{item:19}, and hence \autoref{item:21}, imply the statement
  by \autoref{thm:splitting-for-Om}.
\end{proof}

\begin{lem}\label{lem:finite-split-for-wt}
  Let $f:Y\to X$ be a morphism and $p\in\bN$. If $\phi:\Om_X^p\to\myR f_*\Om_Y^p$ has
  a left inverse, then so does $\wt\Omega_X^p\to f_*\wt\Omega_Y^p$. In particular, if
  $f$ is finite and $X$ is normal, then $\wt\Omega_X^p\to\myR f_*\wt\Omega_Y^p$ has a
  left inverse.
\end{lem}
\begin{proof}
  By assumption there exists a morphism
  $\psi:\myR f_*\Om_Y^p\to\Om_X^p$ such that
  $\psi\circ\phi=\id_{\Om_X^p}$:
  \[
    \xymatrix@R1.5em{%
      \wt\Omega_X^p\ar[r]\ar[d] & \myR f_*\wt\Omega_Y^p \ar[d]_\eta
      \ar[rd]^{\psi\circ\eta}  &  \\
      \Om_X^p \ar[r]\ar@/_1.75em/[rr]^{\id} & \myR f_*\Om_Y^p \ar[r]^\psi & \Om_X^p }
  \]
  Considering $h^0$ of all the objects in this diagram shows the first claim. The
  second claim follows from the first and from \cite[Cor~1.3]{HKim25}.
\end{proof}

\begin{cor}[cf.~\protect{\cite[Cor~1.4]{HKim25}}]
  Let $f:Y\to X$ be a finite morphism. If $X$ is normal and $Y$ has \premdb \sings,
  then $X$ has \premdb \sings. 
\end{cor}
\begin{proof}
  This follows from \autoref{lem:finite-split-for-wt} and
  \autoref{cor:finite-splitting}\autoref{item:21}.
\end{proof}

\begin{cor}
  Let $f:Y\to X$ be a general cyclic cover as in \autoref{setup}. Then $X$ has
  \premdb 
  \sings if and only if $Y$ does.
\end{cor}
\begin{proof}
  If $Y$ has \premdb \sings, then so does $X$ by \autoref{cor:finite-cover} and
  \autoref{cor:finite-splitting}.
  If $X$ has \premdb \sings, then so does $H$ by
  \autoref{prop:hyperplane-sections-of-premdb}. Hence $\Om_X^p(\log H)$ has no $h^i$
  for $i\neq 0$, $p\leq m$ by \autoref{eq:20}. Then $Y$ has \premdb \sings by
  \autoref{lem:finite-cover}\autoref{item:12}.
\end{proof}

\begin{cor}\label{cor:premrtl-is-premdb}
  Let $X$ be a normal variety with \premrtl \sings. Then $X$ has \wmdb \sings (and
  hence also \premdb \sings).
\end{cor}

\begin{proof}
  By assumption $\wt\Irr_X^p\simeq\Irr_X^p$ for each $p\leq m$. Furthermore,
  $\wt\Irr_X^p\simeq\wt\Omega_X^p$ by \autoref{cor:wtOm-of-rtl-sing} and hence
  \autoref{prop:morphs-of-cxs} implies that the natural morphism
  $\wt\Omega_X^p\to\Om_X^p$ has a left inverse for each $p\leq m$. The statement
  follows from \autoref{thm:splitting-for-Om} and \autoref{prop:wmdb+rtl=mdb}.
\end{proof}

\begin{cor}\label{cor:mrtl-is-mdb}
  If $X$ has \mrtl \sings, then it has \mdb \sings.
\end{cor}

\begin{proof}
  $X$ has \wmdb \sings by \autoref{cor:premrtl-is-premdb} and $\wt\Omega_X^p$ is
  reflexive by \autoref{cor:wtOm-of-rtl-sing}. As $X$ has \mrtl \sings, its singular
  set satisfies 
  the required codimension condition. 
\end{proof}

\noin
We also have the strict version of the same statement:

\begin{cor}\label{cor:smrtl-is-smdb}
  If $X$ has strict \mrtl \sings, then it has strict \mdb \sings.
\end{cor}

\begin{proof}
  By \autoref{cor:mrtl-is-mdb}, $X$ has \mdb \sings and
  $\Omega_X^p\simeq\wt\Irr_X^p\simeq\wt\Omega_X^p$ by assumption and by
  \autoref{cor:wtOm-of-rtl-sing}, so it follows that $X$ has strict \mdb \sings.
\end{proof}

\begin{rem}
  \autoref{cor:mrtl-is-mdb}, {\bf\sf\small \ref{cor:smrtl-is-smdb}} and the \premrtl
  implies \premdb part of \autoref{cor:premrtl-is-premdb} were 
  shown in \cite[Thm.~B, Cor.~C, Thm.~D(b)]{SVV23} using different arguments.
\end{rem}






\begin{thebibliography}{MOPW23}

\bibitem[Art66]{MR0199191}
{\sc M.~Artin}: \emph{On isolated rational singularities of surfaces}, Amer. J.
  Math. \textbf{88} (1966), 129--136.

\bibitem[BH93]{MR1251956}
{\sc W.~Bruns and J.~Herzog}: \emph{Cohen-{M}acau\-lay rings}, CSAM, vol.~39,
  CUP, 1993.

\bibitem[Car85]{Carlson85}
{\sc J.~A. Carlson}: \emph{Polyhedral resolutions of algebraic varieties},
  Trans. Amer. Math. Soc. \textbf{292} (1985), 2, 595--612.

\bibitem[CGM87]{MR1017925}
{\sc C.~Cumino, S.~Greco, and M.~Manaresi}: \emph{Weakly normal algebraic
  varieties and {B}ertini theorem}, Proceedings of the {G}eometry {C}onference
  ({M}ilan and {G}argnano, 1987), vol.~57, 1987, pp.~135--148.

\bibitem[CGM86]{MR0825140}
{\sc C.~Cumino, S.~Greco, and M.~Manaresi}: \emph{An axiomatic approach to the
  second theorem of {B}ertini}, J. Algebra \textbf{98} (1986), 1, 171--182.

\bibitem[Del71a]{MR0441965}
{\sc P.~Deligne}: \emph{Th\'eorie de {H}odge. {I}}, Actes du {C}ongr\`es
  {I}nternational des {M}ath\'ematiciens ({N}ice, 1970), {T}ome 1,
  Gauthier-Villars, Paris, 1971, pp.~425--430.

\bibitem[Del71b]{MR0498551}
{\sc P.~Deligne}: \emph{Th\'eorie de {H}odge. {II}}, Inst. Hautes \'Etudes Sci.
  Publ. Math. (1971), 40, 5--57.

\bibitem[Del74]{MR0498552}
{\sc P.~Deligne}: \emph{Th\'eorie de {H}odge. {III}}, Inst. Hautes \'Etudes
  Sci. Publ. Math. (1974), 44, 5--77.

\bibitem[DB81]{DuBois81}
{\sc {\relax Ph}.~{D}u {B}ois}: \emph{Complexe de de {R}ham f{i}ltr\'e d'une
  vari\'et\'e singuli\`ere}, Bull.~Soc.~Math.~France \textbf{109} (1981), 1,
  41--81.

\bibitem[Elk78]{Elkik78}
{\sc R.~Elkik}: \emph{Singularit\'es rationnelles et d\'eformations}, Invent.
  Math. \textbf{47} (1978), 2, 139--147.

\bibitem[EV82]{Esn-Vie82}
{\sc H.~Esnault and E.~Viehweg}: \emph{Rev\^etements cycliques}, Algebraic
  threefolds (Varenna, 1981), Lecture Notes in Math., vol. 947, Springer,
  Berlin, 1982, pp.~241--250.

\bibitem[EV92]{EV92}
{\sc H.~Esnault and E.~Viehweg}: \emph{Lectures on vanishing theorems}, DMV
  Seminar, vol.~20, Birkh\"auser, 1992.

\bibitem[FOV99]{Flenner-OCarrol-Vogel}
{\sc H.~Flenner, L.~O'Carroll, and W.~Vogel}: \emph{Joins and intersections},
  Springer Monographs in Mathematics, Springer-Verlag, Berlin, 1999.

\bibitem[FLT1612]{GE-wiki}
{\sc {Fran{\c c}ois Leclerc du Tremblay}}:
  \emph{\href{https://en.wikipedia.org/wiki/Eminence_grise}{
  {https://en.wikipedia.org/wiki/{\'E}minence\_grise}}}.

\bibitem[FL24a]{FL-isolated}
{\sc R.~Friedman and R.~Laza}: \emph{The higher {D}u {B}ois and higher rational
  properties for isolated singularities}, J. Algebraic Geom. \textbf{33}
  (2024), 3, 493--520.

\bibitem[FL24b]{FL-lci}
{\sc R.~Friedman and R.~Laza}: \emph{Higher {D}u {B}ois and higher rational
  singularities}, Duke Math. J. \textbf{173} (2024), 10, 1839--1881, Appendix
  by Morihiko Saito.

\bibitem[Gre11]{MR2856154}
{\sc D.~Greb}: \emph{Rational singularities and quotients by holomorphic group
  actions}, Ann. Sc. Norm. Super. Pisa Cl. Sci. (5) \textbf{10} (2011), 2,
  413--426.

\bibitem[EGA-IV/3]{EGAIV3}
{\sc A.~Grothendieck}: \emph{\'{E}l\'ements de g\'eom\'etrie alg\'ebrique,
  {IV}. \'{E}tude locale des sch\'emas et des morphismes de sch\'emas. {III}},
  Inst. Hautes \'Etudes Sci. Publ. Math. (1966), 28, 255.

\bibitem[GNPP88]{GNPP88}
{\sc F.~Guill{\'e}n, V.~Navarro{\ }Aznar, P.~Pascual{\ }Gainza, and F.~Puerta}:
  \emph{Hyperr\'esolutions cubiques et descente cohomologique}, LNM, vol. 1335,
  1988.

\bibitem[Har66]{MR0222093}
{\sc R.~Hartshorne}: \emph{Residues and duality}, Lecture notes of a seminar on
  the work of A. Grothendieck, Harvard 1963/64. With an appendix by P. Deligne.
  Lecture Notes in Mathematics, No. 20, Springer-Verlag, 1966.

\bibitem[Har77]{Hartshorne77}
{\sc R.~Hartshorne}: \emph{Algebraic geometry}, Springer-Verlag, 1977, Graduate
  Texts in Mathematics, No. 52.

\bibitem[HJ14]{MR3272910}
{\sc A.~Huber and C.~J\"order}: \emph{Differential forms in the h-topology},
  Algebr. Geom. \textbf{1} (2014), 4, 449--478.

\bibitem[JKSY22]{MR4480883}
{\sc S.-J. Jung, I.-K. Kim, M.~Saito, and Y.~Yoon}: \emph{Higher {D}u {B}ois
  singularities of hypersurfaces}, Proc. Lond. Math. Soc. (3) \textbf{125}
  (2022), 3, 543--567.

\bibitem[KS21]{MR4280862}
{\sc S.~Kebekus and C.~Schnell}: \emph{Extending holomorphic forms from the
  regular locus of a complex space to a resolution of singularities}, J. Amer.
  Math. Soc. \textbf{34} (2021), 2, 315--368.

\bibitem[Kim25]{HKim25}
{\sc H.~Kim}: \emph{Trace for the {D}u~{B}ois complex}, {\sf\scriptsize
  arXiv:2507.07350}.

\bibitem[Kol13]{SingBook}
{\sc J.~Koll{\'a}r}: \emph{Singularities of the minimal model program},
  Cambridge Tracts in Mathematics, vol. 200, Cambridge University Press, 2013,
  with the collaboration of {\sc S{\'a}ndor J Kov{\'a}cs}.

\bibitem[Kol23]{ModBook}
{\sc J.~Koll\'{a}r}: \emph{Families of varieties of general type}, Cambridge
  Tracts in Mathematics, vol. 231, Cambridge University Press, Cambridge, 2023,
  With the collaboration of Klaus Altmann and S{\'a}ndor Kov{\'a}cs.

\bibitem[KK10]{KK10}
{\sc J.~Koll{\'a}r and S.~J. Kov{\'a}cs}: \emph{Log canonical singularities are
  {D}u {B}ois}, JAMS \textbf{23} (2010), 3, 791--813.

\bibitem[Kov97]{Kovacs97c}
{\sc S.~J. Kov{\'a}cs}: \emph{Relative de {R}ham complex for non-smooth
  morphisms}, Birational algebraic geometry (Baltimore, MD, 1996), Contemp.
  Math., vol. 207, Amer. Math. Soc., Providence, RI, 1997, pp.~89--100.

\bibitem[Kov99]{Kovacs99}
{\sc S.~J. Kov{\'a}cs}: \emph{Rational, log canonical, {D}u {B}ois
  singularities: on the conjectures of {K}oll\'ar and {S}teenbrink}, Compositio
  Math. \textbf{118} (1999), 2, 123--133.

\bibitem[Kov00]{Kovacs00b}
{\sc S.~J. Kov{\'a}cs}: \emph{A characterization of rational singularities},
  Duke Math. J. \textbf{102} (2000), 2, 187--191.

\bibitem[Kov05]{Kovacs05a}
{\sc S.~J. Kov{\'a}cs}: \emph{Spectral sequences associated to morphisms of
  locally free sheaves}, Recent progress in arithmetic and algebraic geometry,
  Contemp. Math., vol. 386, Amer. Math. Soc., Providence, RI, 2005, pp.~57--85.

\bibitem[Kov11]{Kovacs10a}
{\sc S.~J. Kov{\'a}cs}: \emph{Du {B}ois pairs and vanishing theorems}, Kyoto J.
  Math. \textbf{51} (2011), 1, 47--69.

\bibitem[Kov26]{Kovacs-in-prep}
{\sc S.~J. Kov{\'a}cs}, in preparation.

\bibitem[KS16a]{MR3617778}
{\sc S.~J. Kov\'acs and K.~Schwede}: \emph{Du {B}ois singularities deform},
  Minimal models and extremal rays ({K}yoto, 2011), Adv. Stud. Pure Math.,
  vol.~70, Math. Soc. Japan, [Tokyo], 2016, pp.~49--65.

\bibitem[KS16b]{KS13}
{\sc S.~J. Kov{\'a}cs and K.~Schwede}: \emph{Inversion of adjunction for
  rational and {D}u~{B}ois pairs}, Algebra Number Theory \textbf{10} (2016), 5,
  969--1000.

\bibitem[KS11]{MR2796408}
{\sc S.~J. Kov{\'a}cs and K.~E. Schwede}: \emph{Hodge theory meets the minimal
  model program: a survey of log canonical and {D}u {B}ois singularities},
  Topology of stratified spaces, MSRI Publ., vol.~58, CUP, 2011, pp.~51--94.

\bibitem[KT23]{KT21}
{\sc S.~J. Kov\'{a}cs and B.~Taji}: \emph{Hodge sheaves underlying flat
  projective families}, Math. Z. \textbf{303} (2023), 303:75.

\bibitem[MOPW23]{MR4583654}
{\sc M.~Musta\c{t}\u{a}, S.~Olano, M.~Popa, and J.~Witaszek}: \emph{The {D}u
  {B}ois complex of a hypersurface and the minimal exponent}, Duke Math. J.
  \textbf{172} (2023), 7, 1411--1436.

\bibitem[MP22]{MR4491455}
{\sc M.~Musta\c{t}\u{a} and M.~Popa}: \emph{Hodge filtration on local
  cohomology, {D}u {B}ois complex and local cohomological dimension}, Forum
  Math. Pi \textbf{10} (2022), Paper No. e22, 58.

\bibitem[MP25]{MP22}
{\sc M.~Musta\c{t}\u{a} and M.~Popa}: \emph{On {$k$}-rational and {$k$}-{D}u
  {B}ois local complete intersections}, Algebr. Geom. \textbf{12} (2025), 2,
  237--261.

\bibitem[Nam02]{MR1891205}
{\sc Y.~Namikawa}: \emph{Projectivity criterion of {M}oishezon spaces and
  density of projective symplectic varieties}, Internat. J. Math. \textbf{13}
  (2002), 2, 125--135.

\bibitem[PP24]{PP24}
{\sc S.~G. Park and M.~Popa}: \emph{Lefschetz theorems,
  $\mathbb{Q}$-factoriality, and {H}odge symmetry for singular varieties}.

\bibitem[PS08]{MR2393625}
{\sc C.~A.~M. Peters and J.~H.~M. Steenbrink}: \emph{Mixed {H}odge structures},
  EMG, vol.~52, Springer-Verlag, 2008.

\bibitem[PSV24]{PSV24}
{\sc M.~Popa, W.~Shen, and A.~D. Vo}: \emph{Injectivity and vanishing for the
  {D}u {B}ois complexes of isolated singularities}, {\sf\scriptsize
  arXiv:2409.18019}.

\bibitem[Sai00]{MR1741272}
{\sc M.~Saito}: \emph{Mixed {H}odge complexes on algebraic varieties}, Math.
  Ann. \textbf{316} (2000), 2, 283--331.

\bibitem[Ser56]{MR0082175}
{\sc J.-P. Serre}: \emph{G\'eom\'etrie alg\'ebrique et g\'eom\'etrie
  analytique}, Ann.~Inst.~Fourier, Grenoble \textbf{6} (1955-56), 1--42.

\bibitem[SVV]{SVV23}
{\sc W.~Shen, S.~G. Venkatesh, and A.~D. Vo}: \emph{On k-{D}u {B}ois and
  k-rational singularities}, {\sf\scriptsize arXiv:2306.03977}.

\bibitem[Ste83]{Steenbrink83}
{\sc J.~H.~M. Steenbrink}: \emph{Mixed {H}odge structures associated with
  isolated singularities}, Singularities, Part 2 (Arcata, Calif., 1981), Proc.
  Sympos. Pure Math., vol.~40, Amer. Math. Soc., 1983, pp.~513--536.

\bibitem[Ste85]{Steenbrink85}
{\sc J.~H.~M. Steenbrink}: \emph{Vanishing theorems on singular spaces},
  Ast\'erisque (1985), 130, 330--341, Differential systems and singularities
  (Luminy, 1983).

\bibitem[Tig23]{Tighe23}
{\sc B.~Tighe}: \emph{The holomorphic extension property for k-{D}u {B}ois
  singularities}, {\sf\scriptsize arXiv.2312.01245}.

\bibitem[Vie82]{Viehweg82b}
{\sc E.~Viehweg}: \emph{Vanishing theorems}, J. Reine Angew. Math. \textbf{335}
  (1982), 1--8.

\bibitem[Voi07]{MR2451566}
{\sc C.~Voisin}: \emph{Hodge theory and complex algebraic geometry. {I}}, CSAM,
  vol.~76, CUP, 2007.

\bibitem[StacksProject]{stacks-project}
{\sc \zzzzz{Stacks Project Authors}}: \emph{Stacks {P}roject},
  \url{http://stacks.math.columbia.edu}.

\end{thebibliography}

\def\cprime{$'$} \def\cprime{$'$} \def\cprime{$'$} \def\cprime{$'$}
  \def\cprime{$'$} \def\polhk#1{\setbox0=\hbox{#1}{\ooalign{\hidewidth
  \lower1.5ex\hbox{`}\hidewidth\crcr\unhbox0}}} \def\cdprime{$''$}
  \def\cprime{$'$} \def\cprime{$'$} \def\cprime{$'$} \def\cprime{$'$}
  \def\cprime{$'$}
\providecommand{\bysame}{\leavevmode\hbox to3em{\hrulefill}\thinspace}
\providecommand{\MR}{\relax\ifhmode\unskip\space\fi MR}
\providecommand{\MRhref}[2]{%
  \href{http://www.ams.org/mathscinet-getitem?mr=#1}{#2}
}
\providecommand{\href}[2]{#2}

\end{document}
